\documentclass{amsart}
\usepackage{amssymb}
\usepackage{amsmath}
\usepackage{amsfonts}

\newtheorem{theorem}{Theorem}
\theoremstyle{plain}

\newtheorem{corollary}{Corollary}

\newtheorem{definition}{Definition}

\newtheorem{lemma}{Lemma}

\newtheorem{proposition}{Proposition}
\newtheorem{remark}{Remark}

\numberwithin{equation}{section}
\input{tcilatex}

\begin{document}
\title[Weyl Algebra]{On the homogeneized Weyl Algebra}
\author{Roberto Martinez-Villa}
\address[A. One and A. Two]{Author OneTwo address line 1\\
Author OneTwo address line 2}
\email[A. One]{aone@aoneinst.edu}
\urladdr{http://www.authorone.oneuniv.edu}
\thanks{Thanks for Author One.}
\author{Jeronimo Mondragon}
\curraddr[A. Two]{Author Two current address line 1\\
Author Two current address line 2}
\email[A.~Two]{atwo@atwoinst.edu}
\urladdr{http://www.authortwo.twouniv.edu}
\thanks{}
\date{October 23, 2012}
\dedicatory{}
\thanks{This paper is in final form and no version of it will be submitted
for publication elsewhere.}

\begin{abstract}
The aim of this paper is to give relations between the category of finetely
generated graded modules over the homogeneized Weyl algebra $B_{n}$, the
finetely generated modules over the Weyl algebra $A_{n}$ and the finetely
generated graded modules over the Yoneda algebra $B_{n}^{!}$ of $B_{n}$. We
will give these relations both at the level of the categories of modules and
at the level of the derived categories.
\end{abstract}

\maketitle

\section{ Various algebras associated to $A_{n}$}

We assume through the paper that the reader is familiar with basic results
on Weyl algebras, as in the book by Coutinho [Co] and the notes by Mili\v{c}i%
\'{c} [Mil], as well as, with basic results on Koszul algebras [GM1], [GM2],
and derived categories, for which we will refer to Miyachi's notes [Mi].

\bigskip

We begin the paper stating without proofs some basic results on the
homogenized Weyl algebras and refer the reader to [Mo] for the proofs.

\bigskip

Let $K$ be a field of zero characteristic, it is well known that the Weyl
algebra $A_{n}$ has the following description by generators and relations:

$A_{n}=K<X_{1},X_{2},...X_{n},\delta _{1},\delta _{2},...\delta
_{n}>/\{[X_{i},\delta _{j}]=\partial _{ij},[X_{i},X_{j}`,[\delta _{i},\delta
_{j}]\}$, where $K<X_{1},X_{2},...X_{n},\delta _{1},\delta _{2},...\delta
_{n}>$ is the free algebra in $2n$ generators and $[X,Y`]$ denotes the
commutator $XY-YX$.

Several filtrations can be given to $A_{n}$ but it is not graded by path
length. We will associate to $A_{n}$ a quadratic algebra the so called
homogenized Weyl algebra defined by quiver and relations as follows: $%
B_{n}=K<X_{1},X_{2},...X_{n},\delta _{1},\delta _{2},...\delta
_{n},Z>/\{[X_{i},\delta _{j}]=\partial _{ij}Z^{2},[X_{i},X_{j}`,[\delta
_{i},\delta _{j}],[X_{i},Z`],[\delta _{i},Z`\}.$

The algebras $B_{n}$ are related to the Weyl algebras as follows: Take the
quotients $A_{n,c}=B_{n}/\{Z-c\}$ with $c\in K.$

When $n=0$ the algebra $A_{n,0}$ is isomorphic to the polynomial algebra
\linebreak $K[X_{1},X_{2},...X_{n},\delta _{1},\delta _{2},...\delta _{n}]$
and $A_{n,1}$ is isomorphic to $A_{n}$ for $c\neq 0$, and $K$ algebraically
closed the algebras $A_{n,c}$ are all isomorphic to $A_{n}.$

By construction, the polynomial algebra $K[z]$ is contained in the center of 
$B_{n}$.

The family of monomials of $B_{n}$, $\{Z^{i}X^{J}\delta ^{L}\mid i\geq 0$, $%
J,L\in N^{n}\}B_{n}$ is a Poincare-Birkoff basis.

An element $b\in B_{n}$ can be written as $b=\sum b_{i,P,Q}Z^{i}X^{P}\delta
^{Q}$ define $\partial (b)=\max \{\mid P\mid +\mid Q\mid \mid b_{i,P,Q}\neq
0\}$

1) $\partial (a+b)\leq \max \{\partial (a),\partial (b)\}$

2)$\partial (ab)=\partial (a)+\partial (b)$

\begin{lemma}
\bigskip The center of $B_{n}$ is isomorphic to $K[Z]$.
\end{lemma}

\begin{proof}
Let $b\in Z(B_{n})$ be an element of the center,

$b=\underset{n=0}{\overset{m}{\Sigma }}(\underset{k+\mid \alpha \mid +\mid
\beta \mid =n}{\Sigma }a_{k,\alpha ,\beta }Z^{k}X_{1}^{\alpha
_{1}}X_{2}^{\alpha _{2}}...X_{n}^{\alpha _{n}}\delta _{1}^{\beta _{1}}\delta
_{2}^{\beta _{2}}...\delta _{n}^{\beta _{n}}).$

We have an equality $\delta _{1}X_{1}=X_{1}\delta _{1}+Z^{2},$ which implies

$\delta _{1}X_{1}^{2}=X_{1}\delta _{1}X_{1}+Z^{2}X_{1}=X_{1}(X_{1}\delta
_{1}+Z^{2})+Z^{2}X_{1}=X_{1}^{2}\delta _{1}+2Z^{2}X_{1}$.

It follows by induction $\delta _{1}X_{1}^{\alpha _{1}}=X_{1}^{\alpha
_{1}}\delta _{1}+\alpha _{1}Z^{2}X_{1}^{\alpha _{1}-1}.$

Hence $\delta _{1}Z^{k}X_{1}^{\alpha _{1}}X_{2}^{\alpha
_{2}}...X_{n}^{\alpha _{n}}\delta _{1}^{\beta _{1}}\delta _{2}^{\beta
_{2}}...\delta _{n}^{\beta _{n}}=$

$Z^{k}X_{1}^{\alpha _{1}}X_{2}^{\alpha _{2}}...X_{n}^{\alpha _{n}}\delta
_{1}^{\beta _{1}+1}\delta _{2}^{\beta _{2}}...\delta _{n}^{\beta
_{n}}+\alpha _{1}Z^{k+2}X_{1}^{\alpha _{1}-1}X_{2}^{\alpha
_{2}}...X_{n}^{\alpha _{n}}\delta _{1}^{\beta _{1}}\delta _{2}^{\beta
_{2}}...\delta _{n}^{\beta _{n}}$ and

$Z^{k}X_{1}^{\alpha _{1}}X_{2}^{\alpha _{2}}...X_{n}^{\alpha _{n}}\delta
_{1}^{\beta _{1}}\delta _{2}^{\beta _{2}}...\delta _{n}^{\beta _{n}}\delta
_{1}=Z^{k}X_{1}^{\alpha _{1}}X_{2}^{\alpha _{2}}...X_{n}^{\alpha _{n}}\delta
_{1}^{\beta _{1}+1}\delta _{2}^{\beta _{2}}...\delta _{n}^{\beta _{n}}$

From the equality $\delta _{1}b=b\delta _{1}$, after cancellation we have

$(\underset{k+\mid \alpha \mid +\mid \beta \mid =n}{\Sigma }\alpha
_{1}a_{k,\alpha ,\beta }Z^{k+2}X_{1}^{\alpha _{1-1}}X_{2}^{\alpha
_{2}}...X_{n}^{\alpha _{n}}\delta _{1}^{\beta _{1}}\delta _{2}^{\beta
_{2}}...\delta _{n}^{\beta _{n}})=0.$

It follows $\alpha _{1}=0.$

Multiplying by: $\delta _{2},...\delta _{n},$we get $b=\underset{n=0}{%
\overset{m}{\Sigma }}(\underset{k+\mid \beta \mid =n}{\Sigma }a_{k,\beta
}Z^{k}\delta _{1}^{\beta _{1}}\delta _{2}^{\beta _{2}}...\delta _{n}^{\beta
_{n}})$ and multiplying by: $X_{1},X_{2},...X_{n}$ we obtain by a similar
calculation $b=\underset{k=0}{\overset{m}{\Sigma }}a_{k}Z^{k}$.
\end{proof}

\subsection{A Filtration on $B_{n}.$}

Define a filtration on $B_{n}$ as follows: $F_{t}B_{n}=\{b\in B_{n}\mid
\partial (b)\leq t\}$

It has the following properties:

0. $F_{0}B_{n}=K[Z]$

1. $F_{t}B_{n}=0$ for n\TEXTsymbol{<}0

2. $\underset{n\in Z}{\cup }F_{t}B_{n}=B_{n}$

3. 1$\in F_{0}B_{n}$

4. \bigskip $F_{p}B_{n}F_{t}B_{n}\subset F_{p+t}B_{n}$

5. For any pair of integers, $p$,$t$ and elmernts $a\in F_{p}B_{n}$ and $%
b\in F_{t}B_{n}$ the commutator $[a,b]=ab-ba$ is in $F_{p+t-1}B_{n}$.

6. $GrB_{n}\cong K[Z,X_{1},X_{2},...X_{n},\delta _{1},\delta _{2},...\delta
_{n}]$

7. $Gr_{1}B_{n}$ is generated as $K$-vector space by $%
Z,X_{1},X_{2},...X_{n},\delta _{1},\delta _{2},...\delta _{n}.$

These conditions imply $\{F_{t}B_{n}\}$ is a "good filtration" (Mili\v{c}i%
\'{c}) as a consequence we get:

\begin{proposition}
$B_{n}$ is noetherian as left and right ring.
\end{proposition}

\begin{proposition}
The ring $B_{n}$ has global dimension $2n+1.$
\end{proposition}

\begin{proposition}
$B_{n}$ has Gelfand-Kirillov dimension $2n+1.$

\begin{theorem}
Let $M$ be a finitely generated $B_{n}$-module, denote by $d_{\lambda }(M)$
the Gelfand-Kirillov dimension of $M$. Then:

i) $Ext_{B_{n}}^{i}(M,B_{n})=0$ $fori<2n+1$-$d_{\lambda }(M)$

ii) $d_{\lambda }(Ext_{B_{n}}^{i}(M,B_{n}))\leq 2n+1-i$ for all $0\leq i\leq
2n+1$

iii) $d_{\lambda }(Ext_{B_{n}}^{2n+1-d_{\lambda }(M)}(M,B_{n}))=d_{\lambda
}(M)$
\end{theorem}
\end{proposition}

\begin{corollary}
\bigskip The algebra $B_{n}$ is Artin Schelter regular.
\end{corollary}

\bigskip Since $B_{n}$ has a Poincare-Birkoff basis and it is quadratic by
[Li], [GH] it is Koszul. Let $B_{n}^{!}$ be its Yoneda algebra $B_{n}^{!}=%
\underset{k\geq 0}{\oplus }Ext_{B}^{k}(K,K)$, by [Sm] $B_{n}^{!}$ is
selfinjective. It follows by general properties of Koszul algebras [GM1]
,[GM2] that $B_{n}^{!}$ has the same quiver as $B_{n}$ and relations
orthogonal with respect to the canonical bilinear form. It is easy to see
that $B_{n}^{!}$ has the following form:

$B_{n}^{!}=K_{q}[X_{1}$, $X_{2}$, $...X_{n}$, $\delta _{1}$, $\delta
_{2},...\delta _{n}$, $Z]/\{X_{i}^{2}$, $\delta _{j}^{2}$, $\overset{n}{%
\underset{i=0}{\tsum }X_{i}\delta _{i}+Z^{2}\}}$,

where $K_{q}[X_{1},X_{2},...X_{n},\delta _{1},\delta _{2},...\delta _{n},Z]$
denotes the quantum polynomial ring.

$K<X_{1},X_{2},...X_{n},\delta _{1},\delta _{2},...\delta
_{n},Z>/\{(X_{i},X_{j}),(\delta _{i},\delta _{j}),(X_{i},Z),(\delta
_{i},Z)\} $. Here $(X,Y)$ denotes the anti commutators $XY+YX$.

The polynomial algebra $C_{n}=K[X_{1},X_{2},...X_{n},\delta _{1},\delta
_{2},...\delta _{n},Z]$ is a Koszul algebra with Yoneda algebra the exterior
algebra:

$C_{n}^{!}=K_{q}[X_{1},X_{2},...X_{n},\delta _{1},\delta _{2},...\delta
_{n},Z]/\{X_{i}^{2},\delta _{j}^{2}\}$.

Observe we obtain $C_{n}$ as a quotient of $B_{n}$ and $C_{n}^{!}$ is a sub
algebra of $B_{n}^{!}$. We want to prove $B_{n}^{!}$ is a free $C_{n}^{!}$%
-module of rank two.

\begin{lemma}
The quantum polynomial algebra $K_{q}[X_{1},X_{2},...X_{m}]$ has a
Poincare-Birkoff basis, in particular it is a noetherian Koszul algebra.
\end{lemma}

\begin{proof}
It is easy to see $K_{q}[X_{1},X_{2},...X_{m}]$ has a quadratic Groebner
basis then by [Li] it has a Poincare-Birkoff basis hence it is noetherian
and by [GH] it is Koszul.
\end{proof}

It is clear that $\ B_{n}^{!}$ is generated as $K$- vector space by square
free words: $X_{j_{1}}$ $X_{j_{2}}$ ...$X_{j_{t}}\delta _{i_{1}}\delta
_{i_{2}}...\delta _{i_{s}}$ , $X_{j_{1}}$ $X_{j_{2}}$ ...$X_{j_{t}}\delta
_{i_{1}}\delta _{i_{2}}...\delta _{i_{s}}Z$ with $X_{j_{u}}\neq $ $X_{j_{v}}$
for $u\neq v$ and $\delta _{i_{k1}}\neq \delta _{i\ell }$ for $k\neq \ell $.
We then have $B_{n}^{!}=C_{n}^{!}+ZC_{n}^{!}$.

\begin{proposition}
There exists a $C_{n}^{!}$-module decomposition: $B_{n}^{!}=C_{n}^{!}\oplus
ZC_{n}^{!}$.
\end{proposition}

\begin{proof}
Since $B_{n}^{!}$ is graded by path length, we can consider only linear
combinations of paths (words) of the same length. We will use the standard
notation: $X^{\alpha }=X^{\alpha _{1}}.X^{\alpha _{2}}...X^{\alpha _{n}}$
and $\delta ^{\beta }=\delta ^{\beta _{1}}\delta ^{\beta _{2}}...\delta
^{\beta _{n}}$ where $\alpha _{i},\beta _{j}\in \{1,0\}$ and $\mid \alpha
\mid =\underset{i=1}{\overset{n}{\tsum }}\alpha _{i}$, $\mid \beta \mid =%
\underset{i=1}{\overset{n}{\tsum }}\beta _{i}.$

To see that the sum is direct we consider linear combinations: $\underset{%
\mid \alpha \mid +\mid \beta \mid =m}{\tsum }\overline{c_{\alpha ,\beta
}X^{\alpha }\delta ^{\beta }}+$ $\underset{\mid \alpha ^{\prime }\mid +\mid
\beta ^{\prime }\mid =m-1}{\tsum }\overline{c_{\alpha ,\beta }X^{\alpha
^{\prime }}\delta ^{\beta ^{\prime }}Z}=\overline{0}.$

This is: $\underset{\mid \alpha \mid +\mid \beta \mid =m}{\tsum }c_{\alpha
,\beta }X^{\alpha }\delta ^{\beta }+$ $\underset{\mid \alpha ^{\prime }\mid
+\mid \beta ^{\prime }\mid =m-1}{\tsum }c_{\alpha ^{\prime },\beta ^{\prime
}}X^{\alpha ^{\prime }}\delta ^{\beta ^{\prime }}Z=\overset{n}{\underset{i=1}%
{\tsum }}q_{i}X_{i}^{2}+\overset{n}{\underset{i=1}{\tsum }}q_{i}^{\prime
}\delta _{i}^{2}+p^{\prime }\overset{n}{(\underset{i=1}{\tsum }}X_{i}\delta
_{i}+Z^{2})+p^{\prime \prime }\overset{n}{(\underset{i=1}{\tsum }}%
X_{i}\delta _{i}+Z^{2})$

Where $p^{\prime }\overset{t}{=\underset{i=1}{\tsum }}m_{i}^{\prime }$ , $%
p^{\prime \prime }\overset{s}{=\underset{i=1}{\tsum }}m_{i}^{\prime \prime }$%
and $m_{i}^{\prime }$ are monomials divisible by a square element and $%
m_{i}^{\prime \prime }$are square free.

Let%
\'{}%
s assume $c_{\alpha ,\beta }\neq 0.$Comparing both terms of the equality,
the term $X^{\alpha }\delta ^{\beta }$is square free then $c_{\alpha ,\beta
}X^{\alpha }\delta ^{\beta }$ does not cancel with elements of the form: $%
q_{i}X_{i}^{2}$, $q_{i}^{\prime }\delta _{i}^{2}$, $p^{\prime }Z^{2}$, $%
p^{\prime \prime }Z^{2}$ or $p^{\prime }X_{i}\delta _{i}.$Therefore $%
c_{\alpha ,\beta }X^{\alpha }\delta ^{\beta }=m_{i}^{\prime \prime
}X_{j}\delta _{j}$ and $m_{i}^{\prime \prime }=$ $c_{\alpha ,\beta
}X^{\alpha -\alpha _{j}}\delta ^{\beta -\beta _{j}}$.

Then $m_{i}^{\prime \prime }Z^{2}$does not cancel with elements of the form $%
X^{\alpha }\delta ^{\beta }$, $X^{\alpha ^{\prime }}\delta ^{\beta ^{\prime
}}Z$, $X_{i}^{2}$, $\delta _{i}^{2}$ or $m_{j}^{\prime \prime }X_{i}\delta
_{i}$.

Therefore: $m_{i}^{\prime \prime }Z^{2}=m_{j}^{\prime }X_{k}\delta
_{k}=c_{\alpha ,\beta }X^{\alpha -\alpha _{j}}\delta ^{\beta -\beta
_{j}}Z^{2}$ and

$m_{j}^{\prime }=c_{\alpha ,\beta }X^{\alpha -\alpha _{j}-\alpha _{k}}\delta
^{\beta -\beta _{j}-\beta _{k}}Z^{2}$ with $j\neq k$, because $X^{\alpha
}\delta ^{\beta }$ is square free.

Now $m_{j}^{\prime }Z^{2}=c_{\alpha ,\beta }X^{^{\alpha -\alpha _{j}-\alpha
_{k}}}\delta ^{\beta -\beta _{j}-\beta _{k}}Z^{4}$ can be canceled only with
an element of the form $m_{s}^{\prime }X_{t}\delta _{t}$ and $m_{s}^{\prime
}=c_{\alpha ,\beta }X^{^{\alpha -\alpha _{j}-\alpha _{k}-\alpha _{t}}}\delta
^{\beta -\beta _{j}-\beta _{k}-\beta _{t}}Z^{4}$.

We continue by induction until getting an element of the form $c_{\alpha
,\beta }Z^{m}$ that can not be canceled, contradicting the assumption $%
c_{\alpha ,\beta }\neq 0$.

In a similar way we get a contradiction when assuming $c_{\alpha ^{\prime
},\beta ^{\prime }}\neq 0$.
\end{proof}

We have proved $B_{n}^{!}$ is a free $C_{n}^{!}$-module of rank two, this
fact will have interesting applications.

For completeness we include the proof of the following well known result.

\begin{lemma}
The Weyl algebra $A_{n}$ is isomorphic to $B_{n}/(z-1)B_{n}.$
\end{lemma}

\begin{proof}
We have a morphism $\varphi :K<X_{1},X_{2},...X_{n},\delta _{1},\delta
_{2},...\delta _{n},Z>\longrightarrow A_{n}$ given by: $\varphi
(X_{i})=X_{i} $, $\varphi (\delta _{i})=\delta _{i}$, $\varphi (Z)=1$.

It follows $\varphi ([X_{i},X_{j}])=\varphi ([\delta _{i},\delta _{j}])=0$,
for $i\neq j$ $\varphi ([X_{i},\delta _{j}])=0$ and $\varphi (X_{i}\delta
_{i}-\delta _{i}X_{i}-Z^{2})=0$, $\varphi (X_{i}Z-ZX_{i})=\varphi (\delta
_{i}Z-Z\delta _{i})=0$.

The map $\varphi $ induces a surjection $\overline{\varphi }%
:B_{n}\rightarrow A_{n}$ such that $(Z-1)B_{n}\sqsubseteq Ker\overline{%
\varphi }.$

Let $\gamma =\underset{i,\alpha ,\beta }{\tsum }c_{i,\alpha ,\beta
}Z^{i}X^{\alpha }\delta ^{\beta }$ be an element in $Ker\overline{\varphi }$%
. We write $Z^{i}=(Z-1)f_{i}(Z)+1.$

Then $\gamma =\underset{i,\alpha ,\beta }{\tsum }c_{i,\alpha ,\beta
}(Z-1)f_{i}(Z)X^{\alpha }\delta ^{\beta }$ +$\underset{i,\alpha ,\beta }{%
\tsum }c_{i,\alpha ,\beta }X^{\alpha }\delta ^{\beta }$.

Hence, $\overline{\varphi }(\gamma )=\underset{\alpha ,\beta }{\tsum }(%
\underset{i}{\sum }c_{i,\alpha ,\beta })X^{\alpha }\delta ^{\beta }=0$,
implies $\underset{i}{\sum }c_{i,\alpha ,\beta }=0$.

Therefore $\gamma =\underset{i,\alpha ,\beta }{\tsum }c_{i,\alpha ,\beta
}(Z-1)f_{i}(Z)X^{\alpha }\delta ^{\beta }\in (Z-1)B_{n}$.
\end{proof}

The embedding $K[Z]\rightarrow B_{n}$ is a morphism of graded algebras.

\subsection{The Nakayama automorphism}

In this subsection we will recall some basic facts about selfinjective
finite dimensional algebras, that will be needed in the particular situation
we are considering. We refer to the paper by Yamagata [Y] for more details.

Let $A$ be a finite dimensional selfinjective $K$-algebra, over a field.
Denote by $D(A_{A})=Hom_{K}(A,K)$ the standard bimodule. There is an
isomorphism of left $A$-modules $\varphi :A\rightarrow D(A),$which induces
by adjunction a map $\beta ^{\prime }:A\otimes _{A}A\rightarrow K.$ By
definition, $\beta ^{\prime }(a\otimes b)=\varphi (b)(a).$The composition: $%
A\times A\overset{p}{\rightarrow }A\otimes _{A}A\overset{\beta ^{\prime }}{%
\rightarrow }K$, $\beta =\beta ^{\prime }p$ is a non degenerated $A-$%
bilinear form, and $A\otimes _{A}A$ is the cokernel of the map $A\otimes
_{K}A\otimes _{K}A\rightarrow A\otimes _{K}A$ given by $a\otimes b\otimes
c\rightarrow ab\otimes c-a\otimes bc.$ Let $\pi :A\otimes _{K}A\rightarrow
A\otimes _{A}A$ be the cokernel map.

The map $\beta ^{\prime }$ induces a map $\overline{\beta }:A\otimes
_{K}A\rightarrow K$ by $\overline{\beta }(x\otimes y)=\beta ^{\prime
}(y\otimes x).$ Hence $\overline{\beta }$ is also non degenerated. In
consequence, there is a $K$-linear isomorphism $\psi :A\rightarrow D(A)$,
given by $\psi (a_{1})(a_{2})=$ $\overline{\beta }(a_{2}\otimes a_{1})=\beta
(a_{1}\otimes a_{2})=\varphi (a_{2})(a_{1})$ , set $\sigma =\psi
^{-1}\varphi .$

There is a chain of equalities: $\beta (\sigma (y),x)=\beta (\psi
^{-1}\varphi (y),x)=\psi \psi ^{-1}\varphi (y)(x)=\varphi (y)(x)=\overline{%
\beta }(y\otimes x)$.

The map $\sigma :A\rightarrow A$ is an isomorphism of $K$-algebras.

$\beta (\sigma (y_{1}y_{2})\otimes z)=\beta (z\otimes y_{1}y_{2})=\beta
(zy_{1}\otimes y_{2})=\overline{\beta }(y_{2},zy_{1})=\beta (\sigma
(y_{2}),zy_{1})=\beta (\sigma (y_{2})z,y_{1})=\overline{\beta }(y_{1},\sigma
(y_{2})z)=\beta (\sigma (y_{1}),\sigma (y_{2})z)=\beta (\sigma (y_{1})\sigma
(y_{2}),z).$

Since $z$ is arbitrary and $\beta $ non degenerated $\sigma
(y_{1}y_{2})=\sigma (y_{1})\sigma (y_{2})$.

Let $D(A)_{\sigma }$ ($_{\sigma ^{-1}}D(A)$) be the $A$- $A$ bimodule with
right (left) multiplication shifted by $\sigma $ ($\sigma ^{-1})$. Then $%
\varphi :A\rightarrow D(A)_{\sigma },$ and $\psi :A\rightarrow _{\sigma
^{-1}}D(A)$ are isomorphisms of $A$- $A$ bimodules.

$\varphi (xa)(y)=\beta (y,xa)=\beta (\sigma (x)\sigma (a),y)=\beta (\sigma
(x),\sigma (a)y)=\beta (\sigma (a)y,x)=$

$\varphi (x)(\sigma (a)y)=\varphi (x)\sigma (a)(y)$, for all $y$. Therefore $%
\varphi (xa)=\varphi (x)\sigma (a)$. As claimed.

In a similar way, $\psi (xb)(y)=\varphi (y)(xb)=\beta (xb,y)=\beta
(x,by)=\varphi (by)(x)=\psi (x)b(y)$. Since $y$ is arbitrary, $\psi
(xb)=\psi (x)b.$

In the other hand, $\varphi (y)(x)=(\psi (x)(y)=\beta (x,y)=\beta (\sigma
\sigma ^{-1}(x),y)=\beta (y,\sigma ^{-1}(x))$\newline
$=\varphi (\sigma ^{-1}(x))(y)$. Hence, $\psi (x)=\varphi (\sigma ^{-1}(x))$.

It follows: $\psi (bx)=\varphi (\sigma ^{-1}(b)\sigma ^{-1}(x))=$ $\sigma
^{-1}(b)\varphi (\sigma ^{-1}(x))=\sigma ^{-1}(b)\psi (x).$

Let $M$ be a finitely generated $A$-module. Since $_{\sigma }A\cong D(A)$ as
bimodule, there are natural isomorphisms:

$D(M^{\ast })=Hom_{A}(Hom_{A}(M,A),D(A)=Hom_{A}(Hom_{A}(M,A),_{\sigma
}A)\cong \sigma M^{\ast \ast }\cong \sigma M$, where $\sigma M=M$ as abelian
group and multiplication by $A$ shifted by $\sigma $.

We look now to the case $A$ a positively graded selfinjective $K$-algebra
and $\varphi :A\rightarrow D(A)[n]$ an isomorphism of graded $A$-modules.
Let $a,x$ be elements of $A$ of degrees $k$ and $j$, respectively. Then $%
\varphi (xa)=\varphi (x)\sigma (a)$ is an homogeneous element of degree $k+j$
and $\sigma (a)=\Sigma \sigma (a)_{i}$, with $\sigma (a)_{i}$ homogeneous
elements of degree $i$. Hence $\varphi (x)\sigma (a)_{i}=0$ for all $i\neq j$%
. $\varphi (x)\sigma (a)_{i}(1)=\varphi (x)(\sigma (a)_{i})=\beta (\sigma
(a)_{i},x)=0$ for all homogeneous elements $x$. Hence $\beta (\sigma
(a)_{i},A)=0$ and $\sigma (a)_{i}=0$ for $i\neq j$.

We have proved $\sigma $ is an isomorphism of graded $K$-algebras, which
induces isomorphisms of graded $A$- $A$ bimodules: $\varphi :A\rightarrow
D(A)_{\sigma }[n],$ and $\psi :A\rightarrow _{\sigma ^{-1}}D(A)[n].$

We only need to check $\psi $ is an isomorphism of graded $K$-vector spaces.

Being $\varphi $ a graded map, $\varphi =\{\varphi _{\ell }\}$ and $\varphi
_{\ell }:A_{\ell }\rightarrow Hom_{K}(A_{n-\ell },K)[n]$ isomorphisms,
inducing maps $\beta _{\ell }^{\prime }:A_{n-\ell }\otimes _{A}A_{\ell
}\rightarrow K$ and $\overline{\beta }_{\ell }:A_{\ell }\otimes
_{K}A_{n-\ell }\rightarrow K$, with $\overline{\beta }_{\ell }(x\otimes y)=$ 
$\beta _{\ell }^{\prime }(y\otimes x).$Each $\overline{\beta }_{\ell }$%
induces maps $\psi _{\ell }:A_{n-\ell }\rightarrow D(A_{\ell })[n]$ such
that $\psi =\{$ $\psi _{\ell }\}$ becomes a graded map.

In case $M$ is a finitely generated graded left $A$-module there is a chain
of isomorphisms of graded $A$-modules:

$D(M^{\ast })=Hom_{A}(Hom_{A}(M,A),D(A))=Hom_{A}(Hom_{A}(M,A),_{\sigma
}A[-n])\cong $\newline
$\sigma M^{\ast \ast }[-n]\cong \sigma M[-n]$.

Assume now $A$ is Koszul sefinjective with Nakayama automorphism $\sigma $
and Yoneda algebra $B$. It was remarked in [M] that under these conditions
there is natural action of $\sigma $ as a graded automorphism of $B$, we
will recall now this construction.

Let $x$ be an element of $Ext_{A}^{n}(A_{0},A_{0})=\underset{i,j}{\oplus }%
Ext_{A}^{n}(S_{i},S_{j})$, $x$=$(x_{i,j})$ with $x_{i,j}$ the extension:

\begin{center}
$x_{i,j}:$ $0\rightarrow S_{j}[-n]\rightarrow E_{1}\rightarrow
E_{2}\rightarrow ...\rightarrow E_{n}\rightarrow S_{i}\rightarrow 0.$
\end{center}

Then $\sigma x_{i,j}$ is the extension:

\begin{center}
$\sigma x_{i,j}:$ $0\rightarrow \sigma S_{j}[-n]\rightarrow \sigma
E_{1}\rightarrow \sigma E_{2}\rightarrow ...\rightarrow \sigma
E_{n}\rightarrow \sigma S_{i}\rightarrow 0.$
\end{center}

Since $\sigma $ is a permutation of the graded simple, $\sigma x=($ $\sigma
x_{i,j})$ is an element of $Ext_{A}^{n}(A_{0},A_{0})$ and $\sigma
:Ext_{A}^{n}(A_{0},A_{0})\rightarrow Ext_{A}^{n}(A_{0},A_{0})$ is an
isomorphism of $K$-vector spaces which extends to a graded automorphism of $%
B=\underset{n\geq 0}{\oplus }Ext_{A}^{n}(A_{0},A_{0})$.

Let $M$ be a finitely generated (graded) $A$-module and $x=(x_{j})\in $ $%
Ext_{A}^{n}(M,A_{0})=\underset{j\geq 0}{\oplus }Ext_{A}^{n}(M,S_{j})$.

\begin{center}
\bigskip\ $x_{,j}:$ $0\rightarrow S_{j}[-n]\rightarrow E_{1}\rightarrow
E_{2}\rightarrow ...\rightarrow E_{n}\rightarrow M\rightarrow 0.$
\end{center}

Then $\sigma x=(\sigma x_{j})$ with

\begin{center}
$\sigma x_{,j}:$ $0\rightarrow \sigma S_{j}[-n]\rightarrow \sigma
E_{1}\rightarrow \sigma E_{2}\rightarrow ...\rightarrow \sigma
E_{n}\rightarrow \sigma M\rightarrow 0.$
\end{center}

In this case there is an isomorphism of $K$-vector spaces: $\sigma
:Ext_{A}^{n}(M,A_{0})\rightarrow Ext_{A}^{n}(\sigma M,A_{0}),$ which induces
a graded isomorphism: $\sigma :\underset{n\geq 0}{\oplus }%
Ext_{A}^{n}(M,A_{0})\rightarrow \underset{n\geq 0}{\oplus }%
Ext_{A}^{n}(\sigma M,A_{0}).$

We will also call the Nakayama automorphism to the automorphism $\sigma $ of 
$B.$

We will look now in more detail to the Nakayama automorphism $\sigma $ of $%
B_{n}^{!}$ of the shriek algebra of the homogenized algebra $B_{n}.$

The graded ring $B_{n}^{!}$ has a sum decomposition: ($%
B_{n}^{!})_{0}=(C_{n}^{!})_{0}$, ($B_{n}^{!})_{1}=(C_{n}^{!})_{1}\oplus
Z(C_{n}^{!})_{0}$, ($B_{n}^{!})_{2}=(C_{n}^{!})_{2}\oplus Z(C_{n}^{!})_{1}$,
.... ($B_{n}^{!})_{i}=(C_{n}^{!})_{i}\oplus Z(C_{n}^{!})_{i-1}$,... ($%
B_{n}^{!})_{2n}=(C_{n}^{!})_{2n}\oplus Z(C_{n}^{!})_{2n-1}$, ($%
B_{n}^{!})_{2n+1}=Z(C_{n}^{!})_{2n}$

The algebra $C_{n}^{!}$ is the exterior algebra in $2n$ variables, hence,

$\dim _{K}(C_{n}^{!})_{j}$=$\left( 
\begin{array}{c}
2n \\ 
j%
\end{array}%
\right) $=$\left( 
\begin{array}{c}
2n \\ 
2n-j%
\end{array}%
\right) $=$\dim _{K}(C_{n}^{!})_{2n-j}.$

Since ($B_{n}^{!})_{j}=(C_{n}^{!})_{j}\oplus Z(C_{n}^{!})_{j-1}$, it follows:

$\dim _{K}(B_{n}^{!})_{j}$=$\left( 
\begin{array}{c}
\text{2n} \\ 
\text{j}%
\end{array}%
\right) $+$\left( 
\begin{array}{c}
\text{2n} \\ 
\text{j-1}%
\end{array}%
\right) $=$\left( 
\begin{array}{c}
\text{2n} \\ 
\text{2n+1-j}%
\end{array}%
\right) $+$\left( 
\begin{array}{c}
\text{2n} \\ 
\text{2n-j}%
\end{array}%
\right) $=$\dim _{K}(B_{n}^{!})_{2n+1-j}$.

The graded left module $D($ ($B_{n}^{!})$ decomposes in homogeneous
components:

$D($ ($B_{n}^{!})=D($ ($B_{n}^{!})_{2n+1})+D($ ($B_{n}^{!})_{2n})+...D($ ($%
B_{n}^{!})_{0})$.

Each component ($B_{n}^{!})_{i}$ has as basis paths of length $i$ either of
the form:

$X_{j_{1}}X_{j_{2}}...X_{j_{s}}\delta _{j_{s}+1}\delta _{j_{s}+2}...\delta
_{j_{i-1}}Z$ or $X_{j_{1}}X_{j_{2}}...X_{j_{s}}\delta _{j_{s}+1}\delta
_{j_{s}+2}...\delta _{j_{i}}$ and D((B$_{n}^{!}$)$_{\text{2n+1-i}}$) has as
basis the dual basis of the 
\'{}%
paths of length $2n+1-i$.

The isomorphism of graded left modules: $\varphi :B_{n}^{!}\rightarrow D($ $%
B_{n}^{!})[-2n-1]$, sends a path of the form $\gamma
=X_{j_{1}}X_{j_{2}}...X_{j_{s}}\delta _{j_{s}+1}\delta _{j_{s}+2}...\delta
_{j_{i-1}}Z$ or of the form

$\gamma =X_{j_{1}}X_{j_{2}}...X_{j_{s}}\delta _{j_{s}+1}\delta
_{j_{s}+2}...\delta _{j_{i}}$ to the dual basis $f_{\delta -\gamma }$ of the
path $\delta -\gamma $ of length $2n+1-i$ , with $\delta $ the path of
maximal length $\delta =X_{1}X_{2}...X_{n}\delta _{1}\delta _{2}...\delta
_{n}Z.$

Since ($B_{n}^{!})_{i}=(C_{n}^{!})_{i}\oplus Z(C_{n}^{!})_{i-1}$ the
isomorphisn $\phi $ restricts to isomorphisms of $K$-vector spaces $\varphi
: $ $(C_{n}^{!})_{i}\rightarrow D(Z(C_{n}^{!})_{2n-i-1})$ and $\varphi :$ $($
$ZC_{n}^{!})_{i-1}\rightarrow D((C_{n}^{!})_{2n-i})$, hence, $\varphi $
induces isomorphisms of $C_{n}^{!}$-modules $\varphi :$ $(C_{n}^{!})%
\rightarrow D(Z(C_{n}^{!}))$ and $\varphi :$ $($ $ZC_{n}^{!})\rightarrow
D((C_{n}^{!}).$

Now the isomorphism $\psi :B_{n}^{!}\rightarrow D($ $B_{n}^{!})[-2n-1]$
given $\psi (b_{1})(b_{2})=\varphi (b_{2})(b_{1})$ is such that for $%
c_{1}\in (C_{n}^{!})_{i}$ and $b\in B_{n}$, $b=\overset{2n+1}{\underset{i=0}{%
\sum }}b_{i}$, $\psi (c_{1})(b)=$ $\overset{2n+1}{\underset{k=0}{\sum }}%
\varphi (b_{k})(c_{1})$, since for all $k\neq i$ the length of $b_{k}$ is
different from $2n+1-i$, then $\varphi (b_{k})(c_{1})=0$ and $\psi
(c_{1})\in D(Z(C_{n}^{!})_{2n-i-1})$, hence, $\psi $ induces an isomorphism
of graded $C_{n}^{!}$-modules $\psi :$ $(C_{n}^{!})\rightarrow
D(Z(C_{n}^{!}))$ and in a similar way an isomorphism $\psi :$ $($ $%
ZC_{n}^{!})\rightarrow D((C_{n}^{!}).$It follows the Nakayama automorphism $%
\sigma $ restricts to an automorphisms of graded rings: $\sigma
:C_{n}^{!}\rightarrow C_{n}^{!}$ and of $C_{n}^{!}$-bimodules $\sigma
:ZC_{n}^{!}\rightarrow ZC_{n}^{!}$.

Any automorphism $\sigma $ of a ring $B$ takes the center to the center,
since $z\in Z(B)$ implies that for any $b\in B,$ $\sigma (zb)=\sigma
(z)\sigma (b)=\sigma (bz)=\sigma (b)\sigma (z).$

In case $B$ is the homogenized Weyl algebra $\sigma (Z)$ is an homogeneous
element of degree one in $Z(B)=K[Z]$. Therefore $\sigma (Z)=kZ$ with $k$ a
non zero element of the field $K$.

\begin{definition}
\bigskip Let $B$ be the homogenized Weyl algebra. A left $B$-module $M$ is
of Z-torsion if for any element $m$ in $M,$there is a non negative integer $%
k $ such that Z$^{k}m=0.$
\end{definition}

Let $M$ be a Koszul left $B^{!}$-module such that $F(M)=\underset{n\geq 0}{%
\oplus }Ext_{A}^{n}(M,A_{0})$ is a $B$-module of $Z$-torsion. Then $F(\sigma
M)$ is of $Z$-torsion, in particular $F(D(M^{\ast })[n])$ is of $Z$-torsion.

Since $F(\sigma M)=\sigma FM$ for $x\in F(M)$ there is an integer $k\geq 0$
such that $Z^{k}x=0$ and in $\sigma FM$, $Z^{k}\ast x=\sigma
(Z^{k})x=c^{k}Z^{k}x=0.$

\section{ The graded localization of the homogenized Weyl algebra}

Consider the multiplicative subset $S=\{1,Z,Z^{2},...\}$of $K[Z]$. The
localization $K[Z]_{S}$ $=\{f/z^{n}\mid n\geq 0\}$is a $Z$-graded algebra
with homogeneous elements $Z^{n}/Z^{m}$of degree $n-m$. It is clear $%
K[Z]_{S}=K[Z,Z^{-1}]$ , the Laurent polynomials.

The natural map: $K[Z]\rightarrow K[Z,Z^{-1}]$ is flat.

Our aim in this section is to study the graded localization:

$(B_{n})_{Z}=B_{n}\underset{K[Z]}{\otimes }K[Z,Z^{-1}]$. It is a $Z$-graded $%
K$-algebra with homogeneous elements $b/Z^{k}$of degree$(b/Z^{k})$= degree($%
b $)$-k$.

We will study this algebra and its relations with the Weyl algebra $A_{n}$.

The natural map $\varphi :B_{n}\rightarrow (B_{n})_{Z}$ given by $%
b\rightarrow b\otimes 1$is a morphism of graded algebras.

Since $Z$ is in the center of $B_{n}$ the ideal $(Z-1)B_{n}$ is two sided.
Let's consider the composition:

$B_{n}\overset{\varphi }{\longrightarrow }(B_{n})_{Z}\overset{\pi }{%
\longrightarrow }(B_{n})_{Z}/(Z-1)(B_{n})_{Z}$.

Let $b$ be an element of $Ker\pi \varphi $. Then $b/1\in (Z-1)(B_{n})_{Z}$
implies $b/1=(Z-1)b^{\prime }/Z^{k}$, hence there exist $t,\ell \geq 0$ such
that $Z^{\ell }b=Z^{t}(Z-1)b^{\prime }=(Z-1)g(Z)b+b$ and $b=(Z-1)(b%
{\acute{}}%
Z^{t}-g(Z)b)\in (Z-1)B_{n}.$It follows $Ker\pi \varphi =(Z-1)B_{n}$ and we
have a commutative diagram:

\begin{center}
$%
\begin{array}{ccccc}
B_{n} & \overset{\varphi }{\longrightarrow } & (B_{n})_{Z} & \overset{\pi }{%
\longrightarrow } & (B_{n})_{Z}/(Z-1)(B_{n})_{Z} \\ 
& q\searrow &  & \nearrow \psi &  \\ 
&  & B_{n}/(Z-1)B_{n} &  & 
\end{array}%
$
\end{center}

where $\psi $ is an injective ring morphism. We will prove it is an
isomorphism.

Let $b/Z^{k}+(Z-1)(B_{n})_{Z}$be an element of $(B_{n})_{Z}/(Z-1)(B_{n})_{Z}$%
. We re write $b/Z^{k}$ as follows:

$b/Z^{k}=$ $b_{\ell }/Z^{\ell }+b_{\ell -1}/Z^{\ell
-1}+...b_{0}+b_{1}Z+...b_{m}Z^{m}.$

But we can re write $b_{i}/1$ as $%
b_{i}/1=Z^{i}b_{i}/Z^{i}=(Z-1)g(Z)b_{i}/Z^{i}+b_{i}/Z^{i}$. Similarly, $%
b_{j}Z^{j}=b_{j}(Z-1)g(Z)+b_{j}$.

Then $b_{i}/1+(Z-1)(B_{n})_{Z}=b_{i}/Z^{i}+(Z-1)(B_{n})_{Z}$ and $%
b_{i}Z^{i}/1+(Z-1)(B_{n})_{Z}=b_{i}/1+(Z-1)(B_{n})_{Z}.$

Therefore: $b/Z^{k}+(Z-1)(B_{n})_{Z}=(b_{\ell }+b_{\ell
-1}+...b_{0}+b_{1}+...b_{m})/1+(Z-1)(B_{n})_{Z}$.

We have proved $\psi $ is an isomorphism.

In fact we proved the following:

\begin{lemma}
With the above notation, there exists ring isomorphisms:

$B_{n}/(Z-1)B_{n}\cong (B_{n})_{Z}/(Z-1)(B_{n})_{Z}\cong A_{n}$.
\end{lemma}

Denote by $\tau $ the composition $\psi ^{-1}\pi $, then the following
triangle commutes:

\begin{center}
$%
\begin{array}{cccc}
B_{n} &  & \overset{\varphi }{\longrightarrow } & (B_{n})_{Z} \\ 
& q\searrow &  & \swarrow \tau \\ 
&  & B_{n}/(Z-1)B_{n} & 
\end{array}%
.$
\end{center}

Since $(B_{n})_{Z}$ is $Z$-graded we have the inclusion $((B_{n})_{Z})_{0}%
\rightarrow (B_{n})_{Z}$ and the projection $(B_{n})_{Z}\longrightarrow
(B_{n})_{Z}/(Z-1)(B_{n})_{Z}$.

Let $\theta :((B_{n})_{Z})_{0}\longrightarrow (B_{n})_{Z}/(Z-1)(B_{n})_{Z}$
be the composition.

\begin{proposition}
The map $\theta :((B_{n})_{Z})_{0}\longrightarrow
(B_{n})_{Z}/(Z-1)(B_{n})_{Z}$ is an isomorphism.
\end{proposition}

\begin{proof}
We prove first $\theta $ is injective. Let $\overset{\wedge }{b}$ be an
element of $((B_{n})_{Z})_{0}$. It can be written as $\overset{\wedge }{b}=%
\underset{i=0}{\overset{m}{\sum }}g_{i}(X,\delta )Z^{-n_{i}}$ with $%
g_{i}(X,\delta )$homogeneous polynomials of degree $n_{i}$ and $%
n_{0}>n_{1}>...n_{m}$.

We have the following equalities: $\overset{\wedge }{b}$=$Z^{-n_{0}}\underset%
{i=0}{\overset{m}{\sum }}g_{i}(X,\delta )Z^{n_{0}-n_{i}}$ and $\underset{i=0}%
{\overset{m}{\sum }}g_{i}(X,\delta )Z^{\text{n}_{0}\text{-n}_{i}}$ =$%
g_{0}(X,\delta )+g_{1}(X,\delta )...g_{m}(X,\delta )+(Z-1)h(Z)g^{\prime
}(X,\delta ).$

Therefore: $\overset{\wedge }{b}=Z^{-n_{0}}\underset{i=0}{\overset{m}{\sum }}%
g_{i}(X,\delta )+(Z-1)Z^{-n_{0}}b^{\prime }.$

Hence, $\theta (\overset{\wedge }{b})=0$ means $Z^{-n_{0}}\underset{i=0}{%
\overset{m}{\sum }}g_{i}(X,\delta )\in (Z-1)(B_{n})_{Z}.$

There exists $s,t\geq 0$ such that $Z^{t}\underset{i=0}{(\overset{m}{\sum }}%
g_{i}(X,\delta ))=Z^{s}(Z-1)b^{\prime \prime }.$

Set $g(X,\delta )=\underset{i=0}{\overset{m}{\sum }}g_{i}(X,\delta )$, then $%
Z^{t}g(X,\delta )=g(X,\delta )+(Z-1)h(Z)g(X,\delta )$ and $g(X,\delta )=(Z-1)%
\overline{b}$ with $\overline{b}\in B_{n}.$

Hence $\overline{b}=b_{0}(X,\delta )+b_{1}(X,\delta )Z+...b_{k}(X,\delta
)Z^{k}.$

Then $g(X,\delta )=-b_{0}(X,\delta )+(b_{0}(X,\delta )-b_{1}(X,\delta
))Z+...(b_{k-1}(X,\delta )-b_{k}(X,\delta ))Z^{k}+b_{k}(X,\delta )Z^{k+1}$

It follows $g(X,\delta )=-b_{0}(X,\delta )$ and $b_{0}(X,\delta
)=b_{1}(X,\delta )=b_{2}(X,\delta )=...b_{k}(X,\delta )=0$

Therefore: $g(X,\delta )=\underset{i=0}{\overset{m}{\sum }}g_{i}(X,\delta
)=0 $ but each $g_{i}(X,\delta )$ has degree $n_{i}$ with $%
n_{0}>n_{1}>...n_{m}.$

It follows each $g_{i}(X,\delta )=0$ and $\overset{\wedge }{b}=0.$

We prove now $\theta $ is surjective.

Take an element $b/Z^{k}+(Z-1)(B_{n})_{Z}$ in $(B_{n})_{Z}/(Z-1)(B_{n})_{Z}$.

The element b decompose into homogeneous components: $b=\overset{m}{\underset%
{i=0}{\sum }}b_{i}$ with degree $b_{i}=n_{i}$ and $n_{0}<n_{1}<...n_{m}.$

As above, $Z^{n_{m}-n_{i}}b_{i}=(Z-1)h(Z)b_{i}+b_{i}$. Set $b_{i}^{\prime
}=Z^{n_{m}-n_{i}}b_{i}$, $b\prime =\overset{m}{\underset{i=0}{\sum }}%
b_{i}^{\prime }.$

Hence $b_{i}/Z^{k}+(Z-1)(B_{n})_{Z}=b_{i}^{\prime }/Z^{k}+(Z-1)(B_{n})_{Z}.$

If $\ell >k$, then $b^{\prime }/Z^{k}+(Z-1)(B_{n})_{Z}=Z^{\ell -k}b^{\prime
}/Z^{\ell }=b^{\prime }/Z^{\ell }+(Z-1)f(Z)b^{\prime \prime }$ and $%
b/Z^{k}+(Z-1)(B_{n})_{Z}=b^{\prime }/Z^{\ell }+(Z-1)(B_{n})_{Z}.$

The case $\ell <k$ is similar: $b^{\prime }/Z^{\ell }=Z^{k-\ell }b^{\prime
}/Z^{k}=b^{\prime }/Z^{k}+(Z-1)f(Z)b^{\prime \prime }$and $%
b/Z^{k}+(Z-1)(B_{n})_{Z}=b^{\prime }/Z^{\ell }+(Z-1)(B_{n})_{Z}.$

In any case $\theta (b^{\prime }/Z^{\ell })=b/Z^{k}+(Z-1)(B_{n})_{Z}$.

We have proved $\theta $ is an isomorphism.
\end{proof}

With the identification $B_{n}/(Z-1)B_{n}\cong
(B_{n})_{Z}/(Z-1)(B_{n})_{Z}\cong A_{n}$ we have proved the graded algebra $%
(B_{n})_{Z}$ has $A_{n}$ in degree zero.

\begin{theorem}
There exists a graded rings isomorphism: $A_{n}\underset{K}{\otimes }%
K[Z,Z^{-1}]\longrightarrow B_{n}\underset{K[Z]}{\otimes }K[Z,Z^{-1}]$.
\end{theorem}

\begin{proof}
Since $((B_{n})_{Z})_{0}\cong A_{n}$, we will prove that multiplication
induces a ring isomorphism: $\mu :((B_{n})_{Z})_{0}\underset{K}{\otimes }%
K[Z,Z^{-1}]\longrightarrow (B_{n})_{Z}:$ $b/Z^{\ell }\otimes
Z^{k}\longrightarrow bZ^{k}/Z^{\ell }$, assume $b/Z^{\ell }$is homogeneous
with degree $b=\ell +k$, then $b/Z^{\ell }=b/Z^{\ell +k}.Z^{k}$ and degree$%
(b/Z^{\ell +k})=0.$

Then $b/Z^{\ell +k}\otimes Z^{k}\longrightarrow b/Z^{\ell }$, similarly if
degree $b=\ell -k$, $b/Z^{\ell -k}\otimes Z^{k}\longrightarrow b/Z^{\ell }.$%
The map $\mu $ is onto.

Consider now an element $b/Z^{k}$of degree zero. We can write $b$ as: $b=%
\underset{i=0}{\overset{k}{\sum }}f_{i}(X,\delta )Z^{k-i}$ with degree$%
f_{i}(X,\delta )=i.$

$b/Z^{k}\otimes Z^{\ell }\longrightarrow Z^{\ell -k}\underset{i=0}{\overset{k%
}{\sum }}f_{i}(X,\delta )Z^{k-i}=0$ in $(B_{n})_{Z}.$

It follows there exist $t\geq 0$ such that $Z^{t}\underset{i=0}{\overset{k}{%
\sum }}f_{i}(X,\delta )Z^{k-i}=0$ in $B_{n}.$

Therefore: $b/Z^{k}\otimes Z^{\ell }=bZ^{t}/Z^{k+t}\otimes Z^{\ell }=0$.

We have proved $\mu $ is an isomorphism.
\end{proof}

The inclusion $K\longrightarrow K[Z,Z^{-1}]$ induces a flat morphism: $%
A_{n}\longrightarrow A_{n}\underset{K}{\otimes }K[Z,Z^{-1}]=A_{n}[Z,Z^{-1}]$
and there is a pair of adjoint functors: $A_{n}[Z,Z^{-1}]$ $\otimes
-:Mod_{A_{n}}\longrightarrow Gr_{A_{n}[Z,Z^{-1}]}$, $%
res_{A}:Gr_{A_{n}[Z,Z^{-1}]}\longrightarrow Mod_{A_{n}}$, where $res_{A}$ is
the restriction.

The following result is a particular case of a theorem given by Dade [Da].
We include the proof for completeness.

\begin{theorem}
The functors $res_{A}$, $A_{n}[Z,Z^{-1}]$ $\otimes -$ are exact inverse
equivalences.
\end{theorem}

\begin{proof}
It is clear both functors are exact.

Let $M=\underset{i\in Z}{\oplus }M_{i}$ be a graded $A_{n}[Z,Z^{-1}]$%
-module. Multiplication induces a morphism of graded modules: $\mu
:A_{n}[Z,Z^{-1}]\underset{A}{\otimes }M_{0}\rightarrow M.$

If $Z^{k}m=0$, then $m=Z^{-k}Z^{k}m=0$ and $\mu $ is injective, but given $%
m\in M_{k}$, $\mu (Z^{k}\otimes Z^{-k}m)=m$.

It follows $\mu $ is an isomorphism.

Moreover, if $f:M\longrightarrow N$ is a morphism of graded $A_{n}[Z,Z^{-1}] 
$-modules, the following diagram commutes:

$%
\begin{array}{ccc}
A_{n}[Z,Z^{-1}]\underset{A}{\otimes }M_{0} & \longrightarrow & M \\ 
\downarrow 1\otimes f_{0} &  & \downarrow f \\ 
A_{n}[Z,Z^{-1}]\underset{A}{\otimes }N_{0} & \longrightarrow & N%
\end{array}%
$

It follows, $(A_{n}[Z,Z^{-1}]\underset{A}{\otimes }-)res_{A}\cong 1.$

Given an $A_{n}$-module $M$, it is clear $res_{A}(A_{n}[Z,Z^{-1}]\underset{A}%
{\otimes }M)\cong M$ and given a morphism of $A_{n}$-modules $f:$ $%
M\rightarrow N$, $res_{A}(1\otimes f)=f.$

Therefore $res_{A}(A_{n}[Z,Z^{-1}]\underset{A}{\otimes }-)\cong 1$.
\end{proof}

\begin{corollary}
The equivalences $res_{A}$, $A_{n}[Z,Z^{-1}]$ $\otimes -$ preserve
projective modules, irreducible modules, send left ideals to left ideals
giving an order preserving bijection.
\end{corollary}

We will study now the relations between $Gr_{B_{n}}$ and $Gr_{(B_{n})_{Z}}$.
We denote by $Q$ the localization functor $Q:$ $Gr_{B_{n}}$ $\rightarrow
Gr_{(B_{n})_{Z}}$, $M\rightarrow M_{Z}$, where

$M_{Z}=(B_{n})_{z}\underset{B}{\otimes }M\cong K[Z,Z^{-1}]$ $\underset{K[Z]}{%
\otimes }B_{n}\underset{B_{n}}{\otimes }M\cong K[Z,Z^{-1}]$ $\underset{K[Z]}{%
\otimes }M$.

If we denote by $gr_{B_{n}}$ and $gr_{(B_{n})_{Z}}$ the categories of
finitely generated graded $B_{n}$ and $(B_{n})_{Z}$-modules, respectively.
Then $Q$ restricts to a functor $Q:$ $gr_{B_{n}}$ $\rightarrow $ $%
gr_{(B_{n})_{Z}}$.

\begin{definition}
Given a $B_{n}$-module $M,$define the $Z$-torsion of $M$ as: $%
t_{Z}(M)=\{m\in M\mid $there exists $n>0$ with $Z^{n}m=0$ \}. It is clear $%
t_{Z}$ is an idempotent radical. We say $M$ is $Z$-torsion when $t_{Z}(M)=M$
and $Z$-torsion free if $t_{Z}(M)=0.$
\end{definition}

The kernel of the natural map $M\longrightarrow M_{Z}$ is $t_{Z}(M).$

\begin{proposition}
i) Let $f:M\longrightarrow N$ be a morphism of graded $B_{n}$-modules. Then $%
f_{Z}:M_{Z}\longrightarrow N_{Z}$ is zero if and only if $f$ factors through
a Z -torsion module.

ii) Let $\varphi :M_{Z}\longrightarrow N_{Z}$ be a morphism of finitely
generated graded $(B_{n})_{Z}$-modules, there exists an integer $k\geq 0$
and a map $f:Z^{k}M\longrightarrow N$ such that the composition $%
M_{Z}\longrightarrow (Z^{k}M)_{Z}\longrightarrow N_{Z}$ is $\varphi $ and $%
M_{Z}\longrightarrow (Z^{k}M)_{Z}$ is an isomorphism of graded modules.

ii) Let $M$ be a finitely generated graded $(B_{n})_{Z}$-module. The there
exists a finitely generated $B_{n}$-sub module $\overline{M}$ of $M$ such
that $(\overline{M})_{Z}\cong M_{Z}.$
\end{proposition}

\begin{proof}
i) Let $f:M\rightarrow N$ be a morphism such that $f_{Z}:M_{Z}%
\longrightarrow N_{Z}$, $f_{Z}=0.$Let $m\in M$ with $f(m)/1=0.$Then there
exist some $k\geq 0$ such that $Z^{k}f(m)=0$. It follows $f(M)$is of $Z$%
-torsion and the map $f$ factors as $f=j\overline{f}$ with $j:$ $%
t_{Z}(M)\rightarrow M$ the inclusion and $\overline{f}:M\rightarrow t_{Z}(M)$
the restriction of $f.$

Conversely, $f=gh$ with $h:M\rightarrow L$ and $g:L\rightarrow M$ maps and $%
L $ of $Z$-torsion, then $L_{Z}=0$ and $f_{Z}=g_{Z}h_{Z}=0.$

ii) Let $\varphi :M_{Z}\longrightarrow N_{Z}$ be a morphism of finitely
generated graded $(B_{n})_{Z}$-modules. Let $m_{1}$, $m_{2}$,... $m_{k}$ be
a set of homogeneous generators of $M_{z}$ with degree $m_{i}=d_{i}$ and let 
$d=\max \{d_{i}\}.$Then $\varphi (m_{i})=\underset{j}{\sum }n_{i,j}\otimes
Z^{k_{i.j}}$, degree $n_{i,j}+k_{i,j}=d_{i}.$

If $k_{i,j}\geq 0$, then $n_{i,j}\otimes
Z^{k_{i.j}}=n_{i,j}Z^{k_{i.j}}\otimes 1,$hence we may assume $\varphi
(m_{i})=\underset{j}{\sum }n_{i,j}\otimes Z^{k_{i.j}}$, with $k_{i,j}\leq 0$

Let $k=$max\{ $-k_{i,j}\}$. Then $\varphi (m_{i})=\underset{j}{\sum }%
n_{i,j}Z^{k+k_{i,j}}\otimes Z^{-k}=n_{i}\otimes Z^{-k}$, degree $%
n_{i}=d_{i}+k.$

Consider the restriction to $Z^{k}M$ of the map: $M\overset{j}{%
\longrightarrow }M\underset{K[Z]}{\otimes }K[Z,Z^{-1}]\overset{\varphi }{%
\longrightarrow }N\underset{K[Z]}{\otimes }K[Z,Z^{-1}]$, $%
f:Z^{k}M\longrightarrow N\otimes 1\cong N.$The map $f$ is a degree zero map.

We have an exact sequence: $0\rightarrow Z^{k}M\rightarrow M\rightarrow
M/Z^{k}M\rightarrow 0$, with $M/Z^{k}M$ of $Z$-torsion. Localizing, there
exist an isomorphism $(Z^{k}M)_{Z}\cong M_{Z}.$There is a map $%
f_{Z}:(Z^{k}M)_{Z}\rightarrow N_{Z}$, given by $f_{Z}(m/Z^{\ell
})=f(Z^{k}(m/Z^{k+\ell })=f(Z^{k}m)/Z^{\ell +k}=\varphi (Z^{k}m)/Z^{\ell
+k}=Z^{k}\varphi (m)/Z^{\ell +k}=\varphi (m/Z^{\ell })$.

iii) Let $M$ be a finitely generated $(B_{n})_{Z}$-module with homogeneous
generators: $m_{1}$, $m_{2}$,... $m_{k}$ of degree $m_{i}=d_{i}.$

By restriction, $M$ is a $B_{n}$-module. Let $\overline{M}$ be the $B_{n}$%
-submodule of $M$ generated by $m_{1}$, $m_{2}$,... $m_{k}$.

Localizing we get $\overline{M}$ $_{Z}=(B_{n})_{Z}\underset{B}{\otimes }%
\overline{M}\cong B_{n}\underset{K[Z]}{\otimes }K[Z,Z^{-1}]\underset{B}{%
\otimes }\overline{M}\cong K[Z,Z^{-1}]\underset{K[Z]}{\otimes }\overline{M}.$

Let $\mu :K[Z,Z^{-1}]\underset{K[Z]}{\otimes }\overline{M}\rightarrow M$ be
the map given by multiplication.

The homogeneous elements of $K[Z,Z^{-1}]\underset{K[Z]}{\otimes }\overline{M}
$ are of the form $Z^{-k}\otimes m$, hence $\mu (Z^{-k}\otimes m)=Z^{-k}m=0$
implies $m=Z^{k}Z^{-k}m=0$.

Let $m$ be an element of $M$ homogeneous of degree $k.$ It has form: $m=\sum
b_{i}/Z^{n_{i}}m_{i}$, degree $b_{i}+d_{i}-n_{i}=k$. Set $n=\max \{n_{i}\}$.

Then $Z^{n}m=\sum b_{i}Z^{n-n_{i}}m_{i}$ $=\overline{m}$ is an element of $%
\overline{M}$ of degree $k+n$ and $\mu (Z^{-k}\otimes \overline{m})=m$.
\end{proof}

\begin{corollary}
Let $M$, $N$ be finitely generated graded $B_{n}$-modules. A map $\varphi
:M_{Z}\rightarrow N_{Z}$ is an isomorphism if and only if there exists a map 
$f:Z^{k}M\rightarrow N$ such that $Kerf$ , $Cokerf$ are of $Z$-torsion and $%
f_{Z}=\varphi .$ If there is a map $f:Z^{k}M\rightarrow N$ such that $f_{Z}$
is an isomorphism then $K\func{erf}$ , $Cokerf$ are of $Z$-torsion even when 
$M$or $N$ are not finitely generated.
\end{corollary}

We shall define another torsion theory on $Gr_{B_{n}}$.

Let $M$ be a graded $B_{n}$-module, $t(M)=\underset{L\in J}{\sum }L$ and $%
J=\{L\mid $sub module of $M$ dim$_{K}L<\infty \}.$

Claim: $t(M/t(M))=0.$

Let $N$ be a finitely generated sub module of $M$ such that $%
N+t(M)/t(M)=N/N\cap t(M)$ is finite dimensional over $K.$Since $B_{n}$is
noetherian is finitely generated, hence of finite dimension over $K$. It
follows $N$ is finite dimensional, so $N\subset t(M).$

Let $N$ be an arbitrary sub module of $M$with $N+t(M)/t(M)$finite
dimensional over $K,$ then $N=\sum N_{i}$, with $N_{i}$ finitely generated,
each $N_{i}+t(M)/t(M$) is finite dimensional, therefore $N_{i}\subset t(M).$%
It follows $N\subset t(M)$.

\section{The derived Categories $D^{b}(QgrB_{n})$ and $D^{b}(gr(B_{n})_{Z})$.%
}

In this section we will study the relations between the derived categories $%
D^{b}(QgrB_{n})$ and $D^{b}(gr(B_{n})_{Z})$ and their relations with the
stable category \underline{$gr$}$_{B_{n}^{!}}$ of the shrike algebra $%
B_{n}^{!}$ of $B_{n}.$

\begin{definition}
We say that a (graded) $B_{n}$-module is torsion if $t(M)=M$ and torsion
free if $t(M)=0.$
\end{definition}

It is clear $t(M)$ is $Z$-torsion and $t(M)\subset t_{Z}(M)$. Therefore if $%
M $ is torsion then it is $Z$ -torsion and if $M$ is $Z$-torsion free then
it is torsion free.

The torsion free modules form a Serre (or thick) subcategory of $Gr_{B_{n}}$%
we localize with respect to this subcategory as explained in $[$Ga$]$,$[$P$]$
. Denote by $QGrB_{n}$ the quotient category and let $\pi :$ $%
Gr_{B_{n}}\rightarrow QGrB_{n}$ be the quotient functor, $QGrB_{n}=GrB_{n}/$%
Torsion, is an abelian category with enough injectives and $\pi $ is an
exact functor. When taking this quotient we are inverting the maps of $%
B_{n}- $graded modules, $f:M\rightarrow N$ such that $K\func{erf}$ and $%
Cokerf$ are torsion.

The category $QGrB_{n}$ has the same objects as $GrB_{n}$ and maps:

$Hom_{QGrB_{n}}(\pi (M),\pi (N))=\underrightarrow{\lim }Hom_{GrB_{n}}(M^{%
\prime },N/t(M))$, the limit running through all the sub modules $M^{\prime
} $of $M$ such that $M/M^{\prime }$ is torsion.

If $M$ is a finitely generated module then the limit has a simpler form:

$Hom_{QGrB_{n}}(\pi (M),\pi (N))=\underset{k}{\underrightarrow{\lim }}%
Hom_{GrB_{n}}(M_{\geq k},N/t(M)).$

In case $N$ is torsion free: $Hom_{QGrB_{n}}(\pi (M),\pi (N))=\underset{k}{%
\underrightarrow{\lim }}Hom_{GrB_{n}}(M_{\geq k},N).$

The functor $\pi :$ $Gr_{B_{n}}\rightarrow QGrB_{n}$ has a right adjoint: $%
\varpi :$ $QGrB_{n}\rightarrow Gr_{B_{n}}$such that $\pi \varpi \cong 1.$ $%
[PN].$

If we denote by $gr_{B_{n}}$ the category of finitely generated graded \ $%
B_{n}$-modules and by $QgrB_{n}$ the full subcategory of $QGrB_{n}$
consisting of the objects $\pi (N)$ with $N$ finitely generated, then the
functor $\pi $ induces by restriction a functor: $\pi :$ $%
gr_{B_{n}}\rightarrow QgrB_{n}.$The kernel of $\pi $ is: $Ker\pi =\{M\in $ $%
gr_{B_{n}}\mid \pi (M)=0\}=\{M\in $ $gr_{B_{n}}\mid t(M)=M\}.$

In the other hand, the functor: $(B_{n})_{Z}\underset{B}{\otimes }%
-:gr_{B_{n}}\rightarrow gr(B_{n})_{Z}$ has kernel: $Ker((B_{n})_{Z}\underset{%
B}{\otimes }-)=\{M\in $ $gr_{B_{n}}\mid M_{Z}=0\}.$

It follows: $Ker\pi \subset Ker((B_{n})_{Z}\underset{B}{\otimes }-).$

According to [P] (pag. 173 Cor. 3.11) there exists a unique functor $\psi $
such that the following diagram commutes:

\begin{center}
$%
\begin{array}{cccccc}
gr_{B_{n}} &  & \overset{\pi }{\rightarrow } &  & QgrB_{n} &  \\ 
(B_{n})_{Z}\underset{B}{\otimes }- & \searrow &  & \swarrow & \psi &  \\ 
&  & gr(B_{n})_{Z} &  &  & 
\end{array}%
$
\end{center}

This is: $\psi \pi =(B_{n})_{Z}\underset{B}{\otimes }-.$

\begin{proposition}
The functor $\psi :QgrB_{n}\rightarrow gr(B_{n})_{Z}$ is exact.
\end{proposition}

\begin{proof}
Let $0\rightarrow \pi (M)\overset{\overset{\wedge }{f}}{\rightarrow }\pi (N)%
\overset{\overset{\wedge }{g}}{\rightarrow }\pi (L)\rightarrow 0$ be an
exact sequence in $QgrB_{n}.$We may assume $M,N,L$ torsion free. Then::

$\overset{\wedge }{f}\in \underset{k}{\underrightarrow{\lim }}%
Hom_{GrB_{n}}(M_{\geq k},N)$ and $\overset{\wedge }{g}\in \underset{s}{%
\underrightarrow{\lim }}Hom_{GrB_{n}}(N_{\geq s},L).$

There exist exact sequences: $0\rightarrow M_{\geq k+1}\rightarrow M_{\geq
k}\rightarrow M_{\geq k}/M_{\geq k+1}\rightarrow 0$ which induces an exact
sequence:

$0\rightarrow Hom_{GrB_{n}}(M_{\geq k}/M_{\geq k+1},N)\rightarrow
Hom_{GrB_{n}}(M_{\geq k},N)\rightarrow Hom_{GrB_{n}}(M_{\geq k+1},N).$

Since we are assuming $N$ is torsion free $Hom_{GrB_{n}}(M_{\geq k}/M_{\geq
k+1},N)=0.$

Hence $\overset{\wedge }{f}\in Hom_{QGrB_{n}}(\pi (M),\pi (N))=\underset{%
k\geq 0}{\cup }Hom_{GrB_{n}}(M_{\geq k},N).$

The map $\overset{\wedge }{f}$ is represented by $f:M_{\geq k}\rightarrow N.$
Similarly, $\overset{\wedge }{g}$ is represented by a map $g:N_{\geq \ell
}\rightarrow L$ and we have a sequence: $:M_{\geq k+\ell }\overset{f}{%
\rightarrow }N_{\geq \ell }\overset{g}{\rightarrow }L$ with $\overset{\wedge 
}{(gf)}=\overset{\wedge }{g}\overset{\wedge }{f}=0$, which implies $gf$
factors through a torsion module, but $L$ torsion free implies $gf=0$. Since 
$M_{\geq k+\ell }$ is torsion free, $f$ is a monomorphism. If $Co\ker g$ is
torsion, there exists an $s\geq 0$ such that $Co\ker g_{\geq s}=0$. Taking a
large enough truncation we obtain a sequence: $M_{\geq s}\overset{f}{%
\rightarrow }N_{\geq s}\overset{g}{\rightarrow }L_{\geq s}$ with $f$ a
monomorphism, $g$ an epimorphism and $gf=0.$

Consider the exact sequences: $0\rightarrow $ $M_{\geq s}\overset{f^{\prime
\prime }}{\rightarrow }Kerg\rightarrow H\rightarrow 0$, $0\rightarrow
Kerg\rightarrow N_{\geq s}\rightarrow L_{\geq s}\rightarrow 0.$

Applying $\pi $ we obtain the following isomorphism of exact sequences:

$%
\begin{array}{ccccccc}
0\rightarrow & \pi (M_{\geq s}) & \overset{\pi f^{\prime \prime }}{%
\rightarrow } & \pi (Kerg) & \rightarrow & \pi (H) & \rightarrow 0 \\ 
& \downarrow \cong &  & \downarrow \cong &  &  &  \\ 
0\rightarrow & \pi (M) & \overset{\overset{\wedge }{f}}{\rightarrow } & Ker%
\overset{\wedge }{g} & \rightarrow & 0 & 
\end{array}%
$

It follows $\pi (H)=0$ and $H$ is torsion, so there exists an integer $t\geq
0$ such that $H_{\geq t}=0.$Finally taking a large enough truncation we get
an exact sequence: $0\rightarrow M_{\geq s}\overset{f}{\rightarrow }N_{\geq
s}\overset{g}{\rightarrow }L_{\geq s}\rightarrow 0$ such that the following
sequences are isomorphic:

$%
\begin{array}{ccccccc}
0\rightarrow & \pi (M_{\geq s}) & \overset{\pi f}{\rightarrow } & \pi
(N_{\geq s}) & \overset{\pi g}{\rightarrow } & \pi (L_{\geq s}) & 
\rightarrow 0 \\ 
& \downarrow \cong &  & \downarrow \cong &  & \downarrow \cong &  \\ 
0\rightarrow & \pi (M) & \overset{\overset{\wedge }{f}}{\rightarrow } & \pi
(N) & \overset{\overset{\wedge }{g}}{\rightarrow } & \pi (L) & \rightarrow 0%
\end{array}%
$

Applying $\psi $ we have an exact sequence: $0\rightarrow \psi \pi (M)%
\overset{\psi \overset{\wedge }{f}}{\rightarrow }\psi \pi (N)\overset{\psi 
\overset{\wedge }{g}}{\rightarrow }\psi \pi (L)\rightarrow 0$, which is
isomorphic to $0\rightarrow (M_{\geq s})_{Z}\overset{f_{Z}}{\rightarrow }%
(N_{\geq s})_{Z}\overset{g_{Z}}{\rightarrow }(L_{\geq s})_{Z}\rightarrow 0.$

We have proved $\psi $ is exact.
\end{proof}

The functor $\psi $ has a derived functor: $D(\psi ):D^{b}(QgrB_{n})$ $%
\rightarrow D^{b}(gr(B_{n})_{Z})$, we will study next its properties.

Observe $QgrB_{n}$ does not have neither enough projective nor enough
injective objects.

\begin{lemma}
Let $0\rightarrow N_{1}\overset{d_{1}}{\rightarrow }N_{2}\overset{d_{2}}{%
\rightarrow }...N_{\ell -1}\overset{d_{\ell -1}}{\rightarrow }N_{\ell
}\rightarrow 0$ be a sequence of $B_{n}$- modules and assume the
compositions $d_{i}d_{i-1}$ factors through a module of $Z$-torsion. Then
there exists a complex: $0\rightarrow N_{1}\overset{\overset{\wedge }{d_{1}}}%
{\rightarrow }N_{2}\oplus t_{Z}(N_{1})\overset{\overset{\wedge }{d_{2}}}{%
\rightarrow }N_{3}\oplus t_{Z}(N_{2})...N_{\ell -1}$ $\oplus t_{Z}(N_{\ell })%
\overset{\overset{\wedge }{d_{\ell -1}}}{\rightarrow }N_{\ell }\rightarrow 0$%
, where $\overset{\wedge }{d_{1}}=\left[ 
\begin{array}{c}
-d_{1} \\ 
s_{1}%
\end{array}%
\right] $, $\overset{\wedge }{d_{i}}=$ $\left[ 
\begin{array}{cc}
(-1)^{i}d_{i} & j_{i+1} \\ 
s_{i} & (-1)^{i}d_{i+1}^{\prime }%
\end{array}%
\right] $ and $\overset{\wedge }{d_{\ell -1}}=\left[ 
\begin{array}{cc}
-d_{\ell -1} & j_{\ell }%
\end{array}%
\right] ,$where the maps $j_{i}:t_{Z}(N_{i})\rightarrow N_{i}$ are the
natural inclusions.
\end{lemma}

\begin{proof}
Each morphism $d_{i}:N_{i}\rightarrow N_{i+1}$ induces by restriction a map $%
d_{i}^{\prime }:t_{Z}(N_{i})\rightarrow t_{Z}(N_{i+1})$ such that the
following diagram commutes:

$%
\begin{array}{ccccccccc}
t_{Z}(N_{1}) & \overset{d_{1}^{\prime }}{\rightarrow } & t_{Z}(N_{2}) & 
\overset{d_{2}^{\prime }}{\rightarrow } & t_{Z}(N_{3}) & \rightarrow ... & 
t_{Z}(N_{\ell -1}) & \overset{d_{\ell -1}^{\prime }}{\rightarrow } & 
t_{Z}(N_{\ell }) \\ 
\downarrow j_{1} &  & \downarrow j_{2} &  & \downarrow j_{3} &  & \downarrow
j_{\ell -1} &  & \downarrow j_{\ell } \\ 
N_{1} & \overset{d_{1}}{\rightarrow } & N_{2} & \overset{d_{2}}{\rightarrow }
& N_{3} &  & N_{\ell -1} & \overset{d_{\ell -1}}{\rightarrow } & N_{\ell }%
\end{array}%
$

Since the compositions $d_{i}d_{i-1}$ factors through a module of $Z$%
-torsion, there exist maps $s_{i}:N_{i}\rightarrow t_{Z}(N_{i+2})$ such that 
$j_{i+2}s_{i}=d_{i+1}d_{i}.$

We have the following equalities: $%
j_{i+2}s_{i}j_{i}=d_{i+1}d_{i}j_{i}=d_{i+1}j_{i}d_{i}^{\prime }=j_{i+2}$ $%
d_{i+1}^{\prime }d_{i}^{\prime }$ and $j_{i+2}$ a monomorphism implies $%
s_{i}j_{i}=d_{i+1}^{\prime }d_{i}^{\prime }$ .

We can easily check $\overset{\wedge }{d_{i+1}}\overset{\wedge }{d_{i}}=0$.
\end{proof}

\begin{proposition}
Denote by $Q$ the localization functor $Q=(B_{n})_{Z}\underset{B}{\otimes }-$
and by $C^{b}(-)$ , the category of bounded complexes. The induced functor $%
C^{b}(Q):$ $C^{b}(gr_{B_{n}})\rightarrow $ $C^{b}(gr_{(B_{n})_{Z}})$ is
dense.
\end{proposition}

\begin{proof}
Let $0\rightarrow \overset{\wedge }{M}_{1}\overset{\delta _{1}}{\rightarrow }
$ $\overset{\wedge }{M}_{2}\overset{\delta _{2}}{\rightarrow }...\overset{%
\wedge }{M}$ $_{\ell -1}\overset{\delta _{\ell -1}}{\rightarrow }\overset{%
\wedge }{M}$ $_{\ell }\rightarrow 0$ be a complex in $%
C^{b}(gr_{(B_{n})_{Z}}).$

For each $\overset{\wedge }{M}_{i}$there exists a finitely generated graded $%
B_{n}$-submodule $M_{i}$ such that $(M_{i})_{Z}\cong \overset{\wedge }{M}%
_{i} $and a graded morphism $d_{i}:Z^{k_{i}}M_{i}\rightarrow M_{i+1}$of $%
B_{n}$-modules such that $(d_{i})_{Z}:(Z^{k_{i}}M_{i})_{Z}\rightarrow
(M_{i+1})_{Z} $ is isomorphic $\delta _{i}:\overset{\wedge }{M}%
_{i}\rightarrow \overset{\wedge }{M}_{i+1}$. Let $k$ be $k=\overset{\ell }{%
\underset{i=0}{\sum }}k^{i} $. We then have a chain of $B_{n}$-morphisms:

$Z^{k}M_{1}\overset{d_{1}}{\rightarrow }Z^{k_{2}+..k_{\ell }}M_{2}\overset{%
d_{2}}{\rightarrow }Z^{k_{3}+..k_{\ell }}M_{3}\overset{d_{3}}{\rightarrow }%
...Z^{k_{\ell -1}+..k_{\ell }}M_{\ell -1}\overset{d_{\ell -1}}{\rightarrow }$
$Z^{k_{\ell }}M_{\ell }$. Changing notation write $M_{i}$ instead of $%
Z^{k_{i}+..k_{\ell }}M_{i}$. We then have a chain of morphisms:

$M_{1}$ $\overset{d_{1}}{\rightarrow }M_{2}\overset{d_{2}}{\rightarrow }%
...M_{\ell -1}\overset{d_{\ell -1}}{\rightarrow }M_{\ell }$ such that ($%
M_{1})_{Z}$ $\overset{d_{1_{Z}}}{\rightarrow }(M_{2})_{Z}\overset{d_{2_{Z}}}{%
\rightarrow }...(M_{\ell -1})_{Z}\overset{d_{\ell -1_{Z}}}{\rightarrow }%
(M_{\ell })_{Z}$ is isomorphic to the complex: $0\rightarrow \overset{\wedge 
}{M}_{1}\overset{\delta _{1}}{\rightarrow }$ $\overset{\wedge }{M}_{2}%
\overset{\delta _{2}}{\rightarrow }...\overset{\wedge }{M}$ $_{\ell -1}%
\overset{\delta _{\ell -1}}{\rightarrow }\overset{\wedge }{M}$ $_{\ell
}\rightarrow 0.$ This implies $(d_{i}d_{i-1})_{Z}=0$, which means $%
d_{i}d_{i-1}$ factors through a $Z$-torsion module. By lemma? there exists a
complex:

$0\rightarrow M_{1}\overset{\overset{\wedge }{d_{1}}}{\rightarrow }%
M_{2}\oplus t_{Z}(M_{1})\overset{\overset{\wedge }{d_{2}}}{\rightarrow }%
M_{3}\oplus t_{Z}(M_{2})...M_{\ell -1}$ $\oplus t_{Z}(M_{\ell })\overset{%
\overset{\wedge }{d_{\ell -1}}}{\rightarrow }M_{\ell }\rightarrow 0$ such
that $0\rightarrow (M_{1})_{Z}\overset{\overset{\wedge }{d_{1_{Z}}}}{%
\rightarrow }(M_{2}\oplus t_{Z}(M_{1}))_{Z}\overset{\overset{\wedge }{%
d_{2_{Z}}}}{\rightarrow }(M_{3}\oplus t_{Z}(M_{2}))_{Z}...(M_{\ell -1}$ $%
\oplus t_{Z}(M_{\ell }))_{Z}\overset{\overset{\wedge }{d_{\ell -1_{Z}}}}{%
\rightarrow }(M_{\ell })_{Z}\rightarrow 0$ is isomorphic to: $0\rightarrow 
\overset{\wedge }{M}_{1}\overset{\delta _{1}}{\rightarrow }$ $\overset{%
\wedge }{M}_{2}\overset{\delta _{2}}{\rightarrow }...\overset{\wedge }{M}$ $%
_{\ell -1}\overset{\delta _{\ell -1}}{\rightarrow }\overset{\wedge }{M}$ $%
_{\ell }\rightarrow 0$.
\end{proof}

\begin{corollary}
The functor $C^{b}(\psi ):$ $C^{b}(Qgr_{B_{n}})\rightarrow $ $%
C^{b}(gr_{(B_{n})_{Z}})$ is dense.
\end{corollary}

\begin{proof}
There are functors $C^{b}(\pi ):$ $C^{b}(gr_{B_{n}})\rightarrow
C^{b}(Qgr_{B_{n}})$ and $C^{b}(\psi ):$ $C^{b}(Qgr_{B_{n}})\rightarrow
C^{b}(gr_{(B_{n})_{Z}})$ such that $C^{b}(\psi )$ $C^{b}(\pi )=$ $C^{b}(Q)$
and $C^{b}(Q)$ dense implies $C^{b}(\psi )$ is dense.
\end{proof}

\begin{corollary}
The induced functors $K^{b}(Q):$ $K^{b}(gr_{B_{n}})\rightarrow $ $%
K^{b}(gr_{(B_{n})_{Z}})$ and $K^{b}(\psi ):$ $K^{b}(Qgr_{B_{n}})\rightarrow $
$K^{b}(gr_{(B_{n})_{Z}})$ are dense.
\end{corollary}

\begin{proof}
Interpreting $K^{b}(\mathcal{A})$ as the stable category of $C^{b}(\mathcal{A%
})$, denote by $\tau :C^{b}(\mathcal{A})\rightarrow K^{b}(\mathcal{A})$ the
corresponding functor. There is a commutative diagram:

$%
\begin{array}{ccc}
C^{b}(gr_{B_{n}}) & \overset{C^{b}(Q)}{\rightarrow } & 
C^{b}(gr_{(B_{n})_{Z}}) \\ 
\downarrow \tau &  & \downarrow \tau \\ 
K^{b}(gr_{B_{n}}) & \overset{K^{b}(Q)}{\rightarrow } & 
K^{b}(gr_{(B_{n})_{Z}})%
\end{array}%
$

Since the functors $\tau $ and $C^{b}(Q)$ are dense, the functor $K^{b}(Q)$
is dense.

As above we have isomorphisms: $K^{b}(\psi )$ $K^{b}(\pi )\cong $ $K^{b}(Q). 
$ It follows $K^{b}(\psi )$ is dense.
\end{proof}

\begin{corollary}
The induced functors $D^{b}(Q):D^{b}(gr_{B_{n}})$ $\rightarrow $ $%
D^{b}(gr_{(B_{n})_{Z}})$ and $D^{b}(\psi ):$ $D^{b}(Qgr_{B_{n}})\rightarrow $
$D^{b}(gr_{(B_{n})_{Z}})$ are dense.
\end{corollary}

\begin{proof}
Since the functors $\pi :$ $gr_{B_{n}}\rightarrow QgrB_{n}$ and $\psi
:QgrB_{n}\rightarrow gr(B_{n})_{Z}$ are exact they induce derived functors: $%
D^{b}(\pi ):$ $D^{b}(gr_{B_{n}})\rightarrow D^{b}(QgrB_{n})$ and $D^{b}(\psi
):D^{b}(QgrB_{n})\rightarrow D^{b}(gr(B_{n})_{Z})$ such that $D^{b}(\psi
)D^{b}(\pi )=D^{b}(Q).$

There is a commutative exact diagram:

$%
\begin{array}{ccc}
K^{b}(gr_{B_{n}}) & \overset{K^{b}(Q)}{\rightarrow } & 
K^{b}(gr_{(B_{n})_{Z}}) \\ 
\downarrow &  & \downarrow \\ 
D^{b}(gr_{B_{n}}) & \overset{D^{b}(Q)}{\rightarrow } & 
D^{b}(gr_{(B_{n})_{Z}})%
\end{array}%
$

where the functors corresponding to the columns are dense, hence $D^{b}(Q)$
is dense, which in turn implies $D^{b}(\psi )$ is dense.
\end{proof}

We will describe next he kernel of the functor $D^{b}(\psi )$. By
definition, $KerD^{b}(\psi )=\{\overset{\wedge }{M^{\circ }}\mid D^{b}(\psi
)(\overset{\wedge }{M^{\circ }})$ is acyclic$\}$.

\begin{proposition}
There is the following description of $\mathcal{T=}$ $KerD^{b}(\psi ).$ $%
KerD^{b}(\psi )=\{\pi M^{\circ }\mid M^{\circ }\in D^{b}(gr_{B_{n}})$ such
that for all $i$ $H^{i}(M^{\circ })$ is of $Z$-torsion\}.
\end{proposition}

\begin{proof}
The kernel of the functor $D^{b}(\psi )$ is the category $\mathcal{T}$ of
complexes: $\overset{\sim }{N}^{\circ }:0\rightarrow \pi N_{1}\overset{%
\overset{\wedge }{d_{1}}}{\rightarrow }\pi N_{2}\overset{\overset{\wedge }{%
d_{2}}}{\rightarrow }\pi N_{3}...\pi N_{\ell -1}$ $\overset{\overset{\wedge }%
{d_{\ell -1}}}{\rightarrow }\pi N_{\ell }\rightarrow 0$, such that:

$0\rightarrow \psi \pi N_{1}\overset{\psi \overset{\wedge }{d_{1}}}{%
\rightarrow }\psi \pi N_{2}\overset{\psi \overset{\wedge }{d_{2}}}{%
\rightarrow }\psi \pi N_{3}...\psi \pi N_{\ell -1}$ $\overset{\overset{%
\wedge }{\psi d_{\ell -1}}}{\rightarrow }\psi \pi N_{\ell }\rightarrow 0$ is
acyclic.

Proceeding as above, we may assume each $N_{i}$ and each map $\overset{%
\wedge }{d_{i}}$ lifts to a map $d_{i}:(N_{i})_{\geq k}\rightarrow N_{i+1}$
such that the map $\pi (d_{i}):\pi ((N_{i})_{\geq k})\rightarrow \pi
(N_{i+1})$ is isomorphic to $\overset{\wedge }{d_{i}}:\pi N_{i}\rightarrow
\pi N_{i+1}.$

Taking a large enough truncation we get a complex of $B_{n}$-modules: $%
N_{\geq k}^{\circ }:$ $0\rightarrow (N_{1})_{\geq k}\overset{d_{1}}{%
\rightarrow }(N_{2})_{\geq k}\overset{d_{2}}{\rightarrow }...(N_{\ell
-1})_{\geq k}\overset{d_{\ell -1}}{\rightarrow }(N_{\ell })_{\geq
k}\rightarrow 0$ such that $\pi ($ $N_{\geq k}^{\circ })\cong \overset{\sim }%
{N}^{\circ }.$

The complex: ($N_{\geq k}^{\circ })_{Z}:$ $0\rightarrow ((N_{1})_{\geq
k})_{Z}\overset{d_{1_{Z}}}{\rightarrow }((N_{2})_{\geq k})_{Z}\overset{%
d_{2_{Z}}}{\rightarrow }...((N_{\ell -1})_{\geq k})_{Z}\overset{d_{\ell
-1_{Z}}}{\rightarrow }((N_{\ell })_{\geq k})_{Z}\rightarrow 0$ is isomorphic
to $0\rightarrow \psi \pi N_{1}\overset{\psi \overset{\wedge }{d_{1}}}{%
\rightarrow }\psi \pi N_{2}\overset{\psi \overset{\wedge }{d_{2}}}{%
\rightarrow }\psi \pi N_{3}...\psi \pi N_{\ell -1}$ $\overset{\overset{%
\wedge }{\psi d_{\ell -1}}}{\rightarrow }\psi \pi N_{\ell }\rightarrow 0$ ,
hence it is acyclic.

Changing notation we have a complex of $B_{n}$-modules: $N^{\circ
}:0\rightarrow N_{1}\overset{d_{1}}{\rightarrow }N_{2}\overset{d_{2}}{%
\rightarrow }...N_{\ell -1}\overset{d_{\ell -1}}{\rightarrow }N_{\ell
}\rightarrow 0$ such that ($N^{\circ })_{Z}:$ $0\rightarrow (N_{1})_{Z}%
\overset{d_{1_{Z}}}{\rightarrow }(N_{2})_{Z}\overset{d_{2_{Z}}}{\rightarrow }%
...(N_{\ell -1})_{Z}\overset{d_{\ell -1_{Z}}}{\rightarrow }(N_{\ell
})_{Z}\rightarrow 0$ is acyclic.

We have exact sequences:

$%
\begin{array}{cccccccccc}
&  &  &  &  & \overset{0}{\downarrow } &  &  &  &  \\ 
0\rightarrow & Kerd_{1} & \rightarrow & N_{1} & \overset{d_{1}}{\rightarrow }
& \func{Im}d_{1} & \rightarrow 0 &  &  &  \\ 
&  &  &  &  & \downarrow j &  &  &  &  \\ 
&  &  &  & 0\rightarrow & Kerd_{2} & \rightarrow & N_{2} & \overset{d_{2}}{%
\rightarrow } & N_{3} \\ 
&  &  &  &  & \downarrow &  &  &  &  \\ 
&  &  &  &  & H^{1}(N^{\circ }) &  &  &  &  \\ 
&  &  &  &  & \underset{0}{\downarrow } &  &  &  & 
\end{array}%
$

Localizing we get exact sequences:

$%
\begin{array}{cccccccccc}
&  &  &  &  & \overset{0}{\downarrow } &  &  &  &  \\ 
0\rightarrow & (Kerd_{1})_{Z} & \rightarrow & (N_{1})_{Z} & \overset{%
d_{1_{Z}}}{\rightarrow } & (\func{Im}d_{1})_{Z} & \rightarrow 0 &  &  &  \\ 
&  &  &  &  & \downarrow j_{Z} &  &  &  &  \\ 
&  &  &  & 0\rightarrow & (Kerd_{2})_{Z} & \rightarrow & (N_{2})_{Z} & 
\overset{d_{2_{Z}}}{\rightarrow } & (N_{3})_{Z} \\ 
&  &  &  &  & \downarrow &  &  &  &  \\ 
&  &  &  &  & H^{1}(N^{\circ })_{Z} &  &  &  &  \\ 
&  &  &  &  & \underset{0}{\downarrow } &  &  &  & 
\end{array}%
$

where $j_{Z}$, $d_{1_{Z}}$ are isomorphisms, hence $H^{0}(N^{\circ })_{Z}=0,$
$H^{1}(N^{\circ })_{Z}=0.$Therefore $H^{0}(N^{\circ })$ and $H^{1}(N^{\circ
})$ are $Z$-torsion. More generally for each $i$ the modules $H^{i}(N^{\circ
})$ are $Z$-torsion.
\end{proof}

\begin{corollary}
\bigskip The Nakayama automorphism $\sigma :B_{n}\rightarrow B_{n}$ induces
an autoequivalence $D^{b}(\sigma ):D^{b}(gr_{B_{n}})\rightarrow
D^{b}(gr_{B_{n}})$ and $\mathcal{T}$ is invariant under $D^{b}(\sigma ).$
\end{corollary}

\begin{proof}
We saw in Section one that given an automorphism of graded algebras $\sigma
:B_{n}\rightarrow B_{n}$, there is an autoequivalence $gr_{B_{n}}\rightarrow
gr_{B_{n}}$, that we also denote by $\sigma $, such that $\sigma (M)$ is the
module $M$ with twisted multiplication $b\in B_{n}$ and $m\in M$, $b\ast
m=\sigma (b)m$, clearly $\sigma $ is an exact functor that sends modules of
finite length into modules of finite length. Then $\sigma $ induces an exact
functor: $\sigma :QgrB_{n}\rightarrow QgrB_{n}.$Therefor: an
autoequivalence: $D^{b}(\sigma ):D^{b}(gr_{B_{n}})\rightarrow
D^{b}(gr_{B_{n}}).$ If $M$ is a module of $Z$-torsion, then $\sigma M$ is of 
$Z$-torsion. From this it is clear that $D^{b}(\sigma )$ sends an element of 
$\mathcal{T}$ to an element of $\mathcal{T}$.
\end{proof}

The category $\mathcal{T}=KerD(\psi )$ is "epasse" (thick) and we can take
the Verdier quotient $D^{b}(QgrB_{n})$ /$\mathcal{T}.$

Our aim is to prove the main result of the section:

\begin{theorem}
There exists an equivalence of triangulated categories: $D^{b}(QgrB_{n})/%
\mathcal{T}\cong D^{b}(gr(B_{n})_{Z})$ .
\end{theorem}

Let $\overset{\wedge }{s}^{-1}\overset{\wedge }{f}:$ $\overset{\wedge }{X}%
^{\circ }\rightarrow \overset{\wedge }{Y}^{\circ }$be a map in $\Psi (%
\mathcal{T})$ (in Miyachi's notation). This is a roof: $%
\begin{array}{ccccc}
&  & \overset{\wedge }{K}^{\circ } &  &  \\ 
& \overset{\wedge }{f}\nearrow &  & \nwarrow \overset{\wedge }{s} &  \\ 
\overset{\wedge }{X}^{\circ } &  &  &  & \overset{\wedge }{Y}^{\circ }%
\end{array}%
$, where $\overset{\wedge }{X}^{\circ }$is a complex of the form: $\overset{%
\wedge }{X}^{\circ :}:0\rightarrow \pi X^{n_{0}}\overset{\overset{\wedge }{d}%
}{\rightarrow }\pi X^{n_{0+1}}\overset{\overset{\wedge }{d}}{\rightarrow }%
...\rightarrow \pi X^{n_{0}+\ell -1}\overset{\overset{\wedge }{d}}{%
\rightarrow }$ $\pi X^{n_{0}+\ell }\rightarrow 0.$

After a proper truncation there exists complexes of graded $B_{n}$-modules $%
X^{\circ }$, $K^{\circ }$, $Y^{\circ }$ such that $\pi X^{\circ }\cong $ $%
\overset{\wedge }{X}^{\circ }$, $\pi K^{\circ }\cong \overset{\wedge }{K}%
^{\circ :}$,$\pi $ $Y^{\circ }\cong \overset{\wedge }{Y}^{\circ :}$and
graded maps $f:X^{\circ }\rightarrow K^{\circ }$, $s:Y^{\circ }\rightarrow
K^{\circ }$ such that $\pi f=\overset{\wedge }{f}$, $\pi s=\overset{\wedge }{%
s}$, the roof $\overset{\wedge }{s}^{-1}\overset{\wedge }{f}$ becomes: $%
\begin{array}{ccccc}
&  & \pi K^{\circ } &  &  \\ 
& \pi f\nearrow &  & \nwarrow \pi s &  \\ 
\pi X^{\circ } &  &  &  & \pi Y^{\circ }%
\end{array}%
$, where $\pi s$ is a quasi isomorphism.

There is a triangle in $K^{b}(grB_{n}):X^{\circ }\overset{f}{\rightarrow }%
K^{\circ }\overset{g}{\rightarrow }Z^{\circ }\overset{h}{\rightarrow }%
X^{\circ }[-1]$ which induces a morphism of triangles:

\begin{center}
$%
\begin{array}{ccccccc}
X^{\circ } & \overset{f}{\rightarrow } & K^{\circ } & \overset{g}{%
\rightarrow } & Z^{\circ } & \overset{h}{\rightarrow } & X^{\circ }[-1] \\ 
\uparrow u &  & \uparrow s &  & \uparrow 1 &  & \uparrow u[-1] \\ 
X^{\prime \circ } & \rightarrow & Y^{\circ } & \overset{gs}{\rightarrow } & 
Z^{\circ } & \rightarrow & X^{\prime \circ }[-1]%
\end{array}%
$
\end{center}

Applying $\pi $ we obtain a morphism of triangles:

\begin{center}
$%
\begin{array}{ccccccc}
\pi X^{\circ } & \overset{\pi f}{\rightarrow } & \pi K^{\circ } & \overset{%
\pi g}{\rightarrow } & \pi Z^{\circ } & \pi \overset{h}{\rightarrow } & \pi
X^{\circ }[-1] \\ 
\uparrow \pi u &  & \uparrow \pi s &  & \uparrow 1 &  & \uparrow \pi u[-1]
\\ 
\pi X^{\prime \circ } & \rightarrow & \pi Y^{\circ } & \overset{\pi gs}{%
\rightarrow } & \pi Z^{\circ } & \rightarrow & \pi X^{\prime \circ }[-1]%
\end{array}%
$
\end{center}

By definition of $\Psi (\mathcal{T})$ the object $\pi Z^{\circ }\in \mathcal{%
T},$which means $Z^{\circ }$ has homology of $Z$-torsion. The maps $\pi s$, $%
\pi u$ are quasi isomorphisms. Applying the functor $\psi $ we obtain a
triangle: $\psi \pi X^{\circ }\overset{\psi \pi f}{\rightarrow }\psi \pi
K^{\circ }\overset{\psi \pi g}{\rightarrow }\psi \pi Z^{\circ }\overset{\psi
\pi h}{\rightarrow }\psi \pi X^{\circ }[-1]$ where $\psi \pi Z^{\circ }$is
acyclic. It follows $\psi \pi f$ is invertible in $D^{b}(gr(B_{n})_{Z}).$

We have proved the functor: $D^{b}(\psi ):$ $D^{b}(Qgr_{B_{n}})\rightarrow $ 
$D^{b}(gr_{(B_{n})_{Z}})$ sends elements of $\Psi (\mathcal{T})$ to
invertible elements in $D^{b}(gr(B_{n})_{Z}).$By $[Mi]$ Prop. 712, there
exists a functor $\theta :D^{b}(Qgr_{B_{n}})/\mathcal{T}\rightarrow
D^{b}(gr_{(B_{n})_{Z}})$ such that the triangle:

\begin{center}
$%
\begin{array}{ccccc}
D^{b}(Qgr_{B_{n}}) &  & \overset{D^{b}(\psi )}{\rightarrow } &  & 
D^{b}(gr(B_{n})_{Z}) \\ 
& Q\searrow &  & \nearrow \theta &  \\ 
&  & D^{b}(QgrB_{n})/\mathcal{T} &  & 
\end{array}%
$, commutes.
\end{center}

Since $D^{b}(\psi )$is dense, so is $\theta .$

Before proving $\theta $ is an equivalence, we will need two lemmas:

\begin{lemma}
Let $K^{\circ }$, $L^{\circ }$ be complexes in $C^{b}(gr_{B_{n}})$ and let $%
\overset{\wedge }{f}:K_{Z}^{\circ }\rightarrow $ $L_{Z}^{\circ }$ be a
morphism of complexes of graded $(B_{n})_{Z}$-modules. Then there exists a
bounded complex of graded $B_{n}$-modules $N^{\circ }$ and a map of
complexes $f:N^{\circ }\rightarrow L^{\circ }$ such that $N_{Z}^{\circ
}\cong K_{Z}^{\circ }$ and $f_{Z}=\overset{\wedge }{f}.$
\end{lemma}

\begin{proof}
Let $K^{\circ }$, $L^{\circ }$ be the complexes: $K^{\circ }$: $0\rightarrow
K^{0}\overset{d}{\rightarrow }$ $K^{1}\overset{d}{\rightarrow }...K^{n-1}%
\overset{d}{\rightarrow }K^{n}\rightarrow 0$ and $L^{\circ }:$ $0\rightarrow
L^{0}\overset{d}{\rightarrow }$ $L^{1}\overset{d}{\rightarrow }...L^{n-1}%
\overset{d}{\rightarrow }L^{n}\rightarrow 0$.

Each map $\overset{\wedge }{f}_{i}:K_{Z}^{i}\rightarrow $ $L_{Z}^{i}$ lifts
to a map $f_{i}:$ $Z^{k_{i}}K^{i}\rightarrow L^{i}$ such that $(f_{i})_{Z}=%
\overset{\wedge }{f}_{i}$. Let $k$ be $k=\max \{k_{j}\}$. Then we have the
following diagram:

$%
\begin{array}{ccccccccc}
0\rightarrow & Z^{k}K^{0} & \overset{d}{\rightarrow } & Z^{k}K^{1} & \overset%
{d}{\rightarrow } & Z^{k}K^{2} & \overset{d}{\rightarrow }... & Z^{k}K^{n} & 
\rightarrow 0 \\ 
& \downarrow f_{0} &  & \downarrow f_{1} &  & \downarrow f_{2} &  & 
\downarrow f_{n} &  \\ 
0\rightarrow & L^{0} & \overset{d}{\rightarrow } & L^{1} & \overset{d}{%
\rightarrow } & L^{2} & \overset{d}{\rightarrow }... & L^{n} & \rightarrow 0%
\end{array}%
$

where $(df_{i-1}-f_{i}d)_{Z}=d\overset{\wedge }{f}_{i-1}-\overset{\wedge }{f}%
_{i}d=0.$The map $df_{i-1}-f_{i}d$ factors though $t_{Z}(L^{i}).$

There exist maps $s_{i-1}:$ $Z^{k}K^{i-1}\rightarrow t_{Z}(L^{i})$ and $%
j_{i}:t_{Z}(L^{i})\rightarrow L^{i}$ such that $f_{i}d-df_{i-1}=j_{i}s_{i-1}$
and the diagrams:

$%
\begin{array}{ccc}
t_{Z}(L^{i-1}) & \overset{d^{\prime }}{\rightarrow } & t_{Z}(L^{i}) \\ 
\downarrow j_{i-1} &  & \downarrow j_{i} \\ 
L^{i-1} & \overset{d}{\rightarrow } & L^{i}%
\end{array}%
$commute.

We have the following equalities: $(f_{i}d-df_{i-1})d=j_{i}s_{i-1}d$, $%
-df_{i-1}d=j_{i}s_{i-1}d$ and $%
d(f_{i-1}d-df_{i-2})=df_{i-1}f=dj_{i-1}s_{i-2}=j_{i}ds_{i-2}.$

But $j_{i}$ mono implies $s_{i-1}d+ds_{i-2}=0.$

We have proved that:

$N^{\circ }:$ $0\rightarrow Z^{k}K^{0}\overset{\overset{\wedge }{d_{0}}}{%
\rightarrow }$ $Z^{k}K^{1}\oplus t_{Z}(L^{1})\overset{\overset{\wedge }{d_{1}%
}}{\rightarrow }$ $Z^{k}K^{2}\oplus t_{Z}(L^{2})...$ $Z^{k}K^{n-1}\oplus
t_{Z}(L^{n-1})\overset{\overset{\wedge }{d_{\ell -1}}}{\rightarrow }%
Z^{k}K^{n}\rightarrow 0,$ where the maps have the following form: $\overset{%
\wedge }{d_{0}}$ = $\left[ 
\begin{array}{c}
d \\ 
s_{0}%
\end{array}%
\right] $, $\overset{\wedge }{d_{i}}$ =$\left[ 
\begin{array}{cc}
d & 0 \\ 
s_{i} & d^{\prime }%
\end{array}%
\right] $, $\overset{\wedge }{d_{\ell -1}}$=$\left[ 
\begin{array}{cc}
d & 0%
\end{array}%
\right] $ is a complex of $B_{n}$-modules and $(f_{i},-j_{i}):Z^{k}K^{i}%
\oplus t_{Z}(L^{i})\rightarrow L^{i}$, $(f,-j):N^{\circ }\rightarrow
L^{\circ }$ is a map of complexes such that $N_{Z}^{\circ }\cong
K_{Z}^{\circ }$ and $(f,j)_{Z}=\overset{\wedge }{f}.$
\end{proof}

\begin{lemma}
Let $K^{\circ }$, $L^{\circ }$ be complexes in $C^{b}(gr_{B_{n}})$ and $%
\overset{\wedge }{f}:L_{Z}^{\circ }\rightarrow $ $K_{Z}^{\circ }$ be a
morphism of complexes of graded $(B_{n})_{Z}$-modules which is homotopic to
zero. Then there exist bounded complexes of $B_{n}$-modules, $M^{\circ }$, $%
N^{\circ }$and a map of complexes $f:N^{\circ }\rightarrow M^{\circ }$, such
that $f$ is homotopic to zero, $N_{Z}^{\circ }\cong L_{Z}^{\circ }$, $%
M_{Z}^{\circ }\cong K_{Z}^{\circ }$ and $f_{Z}\cong $ $\overset{\wedge }{f}$.
\end{lemma}

\begin{proof}
Consider the following diagram:

$%
\begin{array}{ccccccccc}
0\rightarrow & L_{Z}^{0} & \overset{d_{Z}}{\rightarrow } & L_{Z}^{1} & 
\overset{d_{Z}}{\rightarrow } & L_{Z}^{2} & \overset{d_{Z}}{\rightarrow }...
& L_{Z}^{m} & \rightarrow 0 \\ 
& \downarrow \overset{\wedge }{f}_{0} & s_{1}\swarrow & \downarrow \overset{%
\wedge }{f_{1}} & s_{2}\swarrow & \downarrow \overset{\wedge }{f}_{2} & 
s_{m}\swarrow & \downarrow \overset{\wedge }{f}_{m} &  \\ 
0\rightarrow & K_{Z}^{0} & \overset{d_{Z}}{\rightarrow } & K_{Z}^{1} & 
\overset{d_{Z}}{\rightarrow } & K_{Z}^{2} & \overset{d_{Z}}{\rightarrow }...
& L_{Z}^{n} & \rightarrow 0%
\end{array}%
$, where $s:L_{Z}^{\circ }\rightarrow K_{Z}^{\circ }[-1]$ is the homotopy,
hence $\overset{\wedge }{f}_{i}=d_{Z}s_{i}+s_{i+1}d_{Z}.$

For each $i$ there exist integers $k_{i}$, $k_{i}^{\prime }$ and maps $%
t_{i}:Z^{k_{i}}L^{i}\rightarrow K^{i-1}$ and $f_{i}^{\prime
}:Z^{k_{i}^{\prime }}L^{i}\rightarrow K^{i}$ such that $(t_{i})_{Z}=s_{i}$
and $(f_{i}^{\prime })_{Z}=\overset{\wedge }{f}_{i}$. Taking $k=\max
\{k_{j}\}$we have maps:

$%
\begin{array}{ccccccccc}
0\rightarrow & Z^{k}L^{0} & \overset{d}{\rightarrow } & Z^{k}L^{1} & \overset%
{d}{\rightarrow } & Z^{k}L^{2} & \overset{d}{\rightarrow }... & Z^{k}L^{m} & 
\rightarrow 0 \\ 
& \downarrow f_{0}^{\prime } & t_{1}\swarrow & \downarrow f_{1}^{\prime } & 
t_{2}\swarrow & \downarrow f_{2}^{\prime } & t_{m}\swarrow & \downarrow
f_{m}^{\prime } &  \\ 
0\rightarrow & K^{0} & \overset{d}{\rightarrow } & K^{1} & \overset{d}{%
\rightarrow } & K^{2} & \overset{d}{\rightarrow }... & K^{m} & \rightarrow 0%
\end{array}%
$

Consider the map: $(f_{i}^{\prime }-(t_{i+1}d+dt_{i}))_{Z}=\overset{\wedge }{%
f}_{i}-(s_{i+1}d_{Z}+d_{Z}s_{i}=0.$As above, $f_{i}^{\prime
}-(t_{i+1}d+dt_{i})$ factors through a $Z$-torsion module and there exist
maps: $v_{i}:Z^{k}L^{i}\rightarrow t_{Z}(K^{i})$ and inclusions $%
j_{i}:t_{Z}(K^{i})\rightarrow K^{i}$ such that $f_{i}^{\prime
}-(t_{i+1}d+dt_{i})=-j_{i}v_{i}$ or $f_{i}^{\prime
}+j_{i}v_{i}=t_{i+1}d+dt_{i}.$

Set $f_{i}=f_{i}^{\prime }+j_{i}v_{i}.$Then $(f_{i})_{Z}=(f_{i}^{\prime
})_{Z}=\overset{\wedge }{f}_{i}.$

But we have now $f_{i}=t_{i+1}d+dt_{i}$, $f_{i-1}=t_{i}d+dt_{i-1}$ imply $%
f_{i}d=dt_{i}d=f_{i-1}d.$
\end{proof}

We can prove now the theorem.

i) $\theta $ is faithful.

Let $\mathcal{T}_{1}$ be $\mathcal{T}_{1}=\{X^{\circ }\in C^{b}(gr_{B_{n}})$ 
$\mid H^{i}(X^{\circ })$ is torsion for all $i\}$and $\mathcal{T}%
_{2}=\{X^{\circ }\in C^{b}(gr_{B_{n}})$ $\mid H^{i}(X^{\circ })$is $Z$%
-torsion for all $i\}$.

A map in $D^{b}(QgrB_{n})/\mathcal{T}$ can be written as follows:

\begin{center}
$%
\begin{array}{ccccccccc}
&  & \pi K^{\circ } &  &  &  & \pi L^{\circ } &  &  \\ 
& \pi f\nearrow &  & \nwarrow \pi s &  & \pi t\nearrow &  & \nwarrow \pi g & 
\\ 
\pi X^{\circ } &  &  &  & \pi Y^{\circ } &  &  &  & \pi Z^{\circ }%
\end{array}%
$
\end{center}

where $t$, $s\in \Psi (\mathcal{T}_{1})$ and $g\in \Psi (\mathcal{T}_{2}).$

In $K^{b}(gr_{B_{n}})$ we have maps:

\begin{center}
$%
\begin{array}{ccccccccc}
&  & K^{\circ } &  &  &  & L^{\circ } &  &  \\ 
& f\nearrow &  & \nwarrow s &  & t\nearrow &  & \nwarrow g &  \\ 
X^{\circ } &  &  &  & Y^{\circ } &  &  &  & Z^{\circ }%
\end{array}%
$
\end{center}

We have an exact sequence of complexes:

\begin{center}
$%
\begin{array}{ccccccc}
0\rightarrow & Y^{\circ } & \overset{\mu }{\rightarrow } & K^{\circ }\oplus
L^{\circ }\oplus I^{\circ } & \overset{\upsilon }{\rightarrow } & W^{\circ }
& \rightarrow 0%
\end{array}%
$
\end{center}

Where $I^{\circ }$ is a complex which is a sum of complexes of the form: $%
0\rightarrow X\overset{1}{\rightarrow }X\rightarrow 0$, hence acyclic. The
maps $\mu ,\upsilon $ are of the form: $\mu =\left[ 
\begin{array}{c}
s \\ 
t \\ 
u%
\end{array}%
\right] $ and $\upsilon =\left[ 
\begin{array}{ccc}
t^{\prime } & s^{\prime } & v%
\end{array}%
\right] .$

By the long homology sequence, there is an exact sequence: *)

\begin{center}
...$\rightarrow H^{i+1}(W^{\circ })\rightarrow H^{i}(Y^{\circ })\overset{%
H^{i}(\mu )}{\rightarrow }H^{i}(K^{\circ })\oplus H^{i}(L^{\circ })\overset{%
H^{i}(\upsilon )}{\rightarrow }H^{i}(W^{\circ })\rightarrow H^{i-1}(Y^{\circ
})\rightarrow ...$
\end{center}

Since $\pi $ is an exact functor, for any complex $\pi H^{i}(X^{\circ
})\cong H^{i}(\pi X^{\circ })$ and the exact sequence *) induces an exact
sequence: **)

\begin{center}
...$\rightarrow \pi H^{i+1}(W^{\circ })\rightarrow \pi H^{i}(Y^{\circ })%
\overset{\pi H^{i}(\mu )}{\rightarrow }\pi H^{i}(K^{\circ })\oplus \pi
H^{i}(L^{\circ })\overset{\pi H^{i}(\upsilon )}{\rightarrow }\pi
H^{i}(W^{\circ })\rightarrow \pi H^{i-1}(Y^{\circ })\rightarrow ...$
\end{center}

Which is isomorphic to the complex:

\begin{center}
...$\rightarrow H^{i+1}(\pi W^{\circ })\rightarrow H^{i}(\pi Y^{\circ })%
\overset{H^{i}(\pi \mu )}{\rightarrow }H^{i}(\pi K^{\circ })\oplus H^{i}(\pi
L^{\circ })\overset{H^{i}(\pi \upsilon )}{\rightarrow }H^{i}(\pi W^{\circ
})\rightarrow H^{i-1}(\pi Y^{\circ })\rightarrow ...$
\end{center}

The maps $H^{i}(\pi s),H^{i}(\pi t)$ are isomorphisms. Hence it follows $%
H^{i}(\pi \mu )$ is for each $i$ a splittable monomorphism and for each $i$
there is an exact sequence:

\begin{center}
0$\rightarrow H^{i}(\pi Y^{\circ })\overset{H^{i}(\pi \mu )}{\rightarrow }%
H^{i}(\pi K^{\circ })\oplus H^{i}(\pi L^{\circ })\overset{H^{i}(\pi \upsilon
)}{\rightarrow }H^{i}(\pi W^{\circ })\rightarrow 0$
\end{center}

which can be embedded in a commutative exact diagram:

\begin{center}
$%
\begin{array}{ccccccc}
&  &  & 0 &  & 0 &  \\ 
&  &  & \downarrow &  & \downarrow &  \\ 
& 0 &  & H^{i}(\pi L^{\circ }) & \overset{1}{\rightarrow } & H^{i}(\pi
L^{\circ }) &  \\ 
& \downarrow &  & \downarrow \left[ 
\begin{array}{c}
0 \\ 
1%
\end{array}%
\right] &  & \downarrow H^{i}(\pi s^{\prime }) &  \\ 
0\rightarrow & H^{i}(\pi Y^{\circ }) & \rightarrow & H^{i}(\pi K^{\circ
})\oplus H^{i}(\pi L^{\circ }) & \rightarrow & H^{i}(\pi W^{\circ }) & 
\rightarrow 0 \\ 
& \downarrow H^{i}(\pi s) &  & \downarrow \left[ 
\begin{array}{cc}
1 & 0%
\end{array}%
\right] &  & \downarrow &  \\ 
0\rightarrow & H^{i}(\pi K^{\circ }) & \overset{1}{\rightarrow } & H^{i}(\pi
K^{\circ }) & \rightarrow & 0 &  \\ 
& \downarrow &  & \downarrow &  &  &  \\ 
& 0 &  & 0 &  &  & 
\end{array}%
$
\end{center}

By this and a similar diagram it follows $H^{i}(\pi s^{\prime }),H^{i}(\pi
t^{\prime })$ are isomorphisms.

We have a commutative diagram in $K^{b}(Qgr_{B_{n}}):$

\begin{center}
\bigskip $%
\begin{array}{ccccccccc}
&  &  &  &  &  &  &  &  \\ 
&  &  &  & \pi W^{\circ } &  &  &  &  \\ 
&  &  & \pi t^{\prime }\nearrow &  & \nwarrow \pi s^{\prime } &  &  &  \\ 
&  & \pi K^{\circ } &  &  &  & \pi L^{\circ } &  &  \\ 
& \pi f\nearrow &  & \nwarrow \pi s &  & \pi t\nearrow &  & \nwarrow \pi g & 
\\ 
\pi X^{\circ } &  &  &  & \pi Y^{\circ } &  &  &  & \pi Z^{\circ }%
\end{array}%
$
\end{center}

Then $\theta ($ ($\pi g)^{-1}\pi t(\pi s)^{-1}\pi f)=D^{b}(\psi )((\pi
s^{\prime }\pi g)^{-1}\pi t^{\prime }\pi f))=(s_{Z}^{\prime
}g_{Z})^{-1}t_{Z}^{\prime }f_{Z}=0.$

But $s_{Z}^{\prime }$, $g_{Z}$, $t_{Z}^{\prime }$ are isomorphisms in $%
D^{b}(grB_{z}).$ It follows $f_{Z}=0$ in $D^{b}(grB_{z}).$

Then there is a quasi isomorphism of complexes $\upsilon :\overset{\wedge }{N%
}^{\circ }\rightarrow X_{Z}^{\circ }$ such that $f_{Z}\upsilon $ is
homotopic to zero. By Lemma ?, there is a bounded complex $N^{\circ }$ of $%
B_{n}$-modules and a map $\nu :N^{\circ }\rightarrow X^{\circ }$such that $%
N_{Z}^{\circ }\cong \overset{\wedge }{N}^{\circ }$and $\nu _{Z}$ can be
identified with $\upsilon .$

According to Lemma ??. there is an integer $k\geq 0$ such that the
composition of maps $Z^{k}N^{\circ }\overset{res\nu }{\rightarrow }X^{\circ }%
\overset{f}{\rightarrow }K^{\circ }$ is homotopic to zero and ($res\nu
)_{Z}=\nu _{Z}$ is a quasi isomorphism. This implies $res\nu \in \Psi (%
\mathcal{T}_{2})$ and $\pi f=0$ in $D^{b}(QgrB_{n})/\mathcal{T}.$

Therefore $\pi g)^{-1}\pi t(\pi s)^{-1}\pi f=0$ in $D^{b}(QgrB_{n})/\mathcal{%
T}.$

ii) $\theta $ is full.

Let $%
\begin{array}{ccccc}
&  & K_{Z}^{\circ } &  &  \\ 
& \overset{\wedge }{s}\swarrow &  & \searrow \overset{\wedge }{f} &  \\ 
X_{Z}^{\circ } &  &  &  & Y_{Z}^{\circ }%
\end{array}%
$be a map in $D^{b}(gr_{(B_{n})_{Z}})$. By lemma ?, there exists a complex:

$N^{\circ }:$ $0\rightarrow Z^{k}K^{0}\overset{\overset{\wedge }{d_{0}}}{%
\rightarrow }$ $Z^{k}K^{1}\oplus t_{Z}(Y^{1})\overset{\overset{\wedge }{d_{1}%
}}{\rightarrow }$ $Z^{k}K^{2}\oplus t_{Z}(Y^{2})...$ $Z^{k}K^{n-1}\oplus
t_{Z}(Y^{n-1})\overset{\overset{\wedge }{d_{\ell -1}}}{\rightarrow }%
Z^{k}K^{n}\rightarrow 0,$ where the maps have the following form: $\overset{%
\wedge }{d_{0}}$ = $\left[ 
\begin{array}{c}
d \\ 
s_{0}%
\end{array}%
\right] $, $\overset{\wedge }{d_{i}}$ =$\left[ 
\begin{array}{cc}
d & 0 \\ 
s_{i} & d^{\prime }%
\end{array}%
\right] $, $\overset{\wedge }{d_{\ell -1}}$=$\left[ 
\begin{array}{cc}
d & 0%
\end{array}%
\right] $ and a map $f:N^{\circ }\rightarrow Y^{\circ }$such that $%
N_{Z}^{\circ }\cong K_{Z}^{\circ }$ and $f_{Z}=$ $\overset{\wedge }{f}.$%
Changing $N^{\circ }$ for $K^{\circ }$ we may assume $\overset{\wedge }{f}$
is a localized map $f_{Z}$ and get a roof:

$%
\begin{array}{ccccc}
&  & N_{Z}^{\circ } &  &  \\ 
& \overset{\wedge }{s}\swarrow &  & \searrow f_{Z} &  \\ 
X_{Z}^{\circ } &  &  &  & Y_{Z}^{\circ }%
\end{array}%
.$

We now lift $\overset{\wedge }{s}$ to a map of complexes $s:\overset{\wedge }%
{N}^{\circ }\rightarrow X^{\circ }$:

\begin{center}
$%
\begin{array}{cccccccccc}
\text{0}\rightarrow & \text{Z}^{k}\text{N}^{0} & \overset{\overset{\wedge }{d%
}_{0}}{\rightarrow } & \text{Z}^{k}\text{N}^{1}\oplus \text{t}_{Z}\text{(X}%
^{1}\text{)} & \overset{\overset{\wedge }{d}_{1}}{\rightarrow } & \text{Z}%
^{k}\text{N}^{2}\oplus \text{t}_{Z}\text{(X}^{2}\text{)} & \text{...} & 
\overset{\overset{\wedge }{d_{m-1}}}{\rightarrow } & \text{Z}^{k}\text{N}^{m}
& \rightarrow \text{0} \\ 
s: & \downarrow \text{s}_{0} &  & \downarrow \text{s}_{1} &  & \downarrow 
\text{s}_{2} &  &  & \downarrow \text{s}_{m} &  \\ 
\text{0}\rightarrow & \text{X}^{0} & \overset{d}{\rightarrow } & \text{X}^{1}
& \overset{d}{\rightarrow } & \text{X}^{2} & \text{...} & \overset{d}{%
\rightarrow } & \text{X}^{m} & \rightarrow \text{0}%
\end{array}%
s$
\end{center}

with $s_{z}=\overset{\wedge }{s}.$

We have a commutative diagram:\newline
$%
\begin{array}{cccccccccc}
\text{0}\rightarrow & \text{Z}^{k}\text{N}^{0} & \overset{\overset{\wedge }{d%
}_{0}}{\rightarrow } & \text{Z}^{k}\text{N}^{1}\oplus \text{t}_{Z}\text{(X}%
^{1}\text{)} & \overset{\overset{\wedge }{d}_{1}}{\rightarrow } & \text{Z}%
^{k}\text{N}^{2}\oplus \text{t}_{Z}\text{(X}^{2}\text{)} & \text{...} & 
\overset{\overset{\wedge }{d_{m-1}}}{\rightarrow } & \text{Z}^{k}\text{N}^{m}
& \rightarrow \text{0} \\ 
(10) & \downarrow \text{1} &  & \downarrow \text{(10)} &  & \downarrow \text{%
(10)} &  &  & \downarrow \text{1} &  \\ 
\text{0}\rightarrow & \text{Z}^{k}\text{N}^{0} & \overset{d}{\rightarrow } & 
\text{Z}^{k}\text{N}^{1} & \overset{d}{\rightarrow } & \text{Z}^{k}\text{N}%
^{2} & \text{...} & \overset{d}{\rightarrow } & \text{Z}^{k}\text{N}^{m} & 
\rightarrow \text{0}%
\end{array}%
$

We obtain the following roof:

\begin{center}
$%
\begin{array}{ccccccc}
&  & \overset{\wedge }{N}^{\circ } &  &  &  &  \\ 
& {\small s}\swarrow &  & \searrow {\small (10)} &  &  &  \\ 
X^{\circ } &  &  &  & Z^{k}N^{\circ } &  &  \\ 
&  &  &  &  & \searrow {\small f} &  \\ 
&  &  &  &  &  & Y^{\circ }%
\end{array}%
$
\end{center}

Localizing we obtain:

\begin{center}
$%
\begin{array}{ccccccc}
&  & \overset{\wedge }{N_{Z}}^{\circ } &  &  &  &  \\ 
& {\small s}_{Z}\swarrow &  & \searrow {\small (10)}_{Z} &  &  &  \\ 
X_{Z}^{\circ } &  &  &  & Z^{k}N_{Z}^{\circ } &  &  \\ 
&  &  &  &  & \searrow {\small f}_{Z} &  \\ 
&  &  &  &  &  & Y_{Z}^{\circ }%
\end{array}%
$
\end{center}

with $\overset{\wedge }{N_{Z}}^{\circ }\overset{{\small (10)}_{Z}}{%
\rightarrow }Z^{k}N_{Z}^{\circ }\cong N_{Z}^{\circ }$ isomorphisms, $s_{Z}=%
\overset{\wedge }{s}$, $f_{z}=\overset{\wedge }{f}.$

We have proved $\theta $ is full.

\section{The category of $\mathcal{T}$-local objects.}

Let $\mathcal{F}$ be the full subcategory of $D^{b}(Qgr_{B_{n}})$ consisting
of $\mathcal{T}$-local objects, this is: $\mathcal{F}=\{X^{\circ }\in
D^{b}(Qgr_{B_{n}})\mid Hom_{D^{b}(Qgr_{B_{n}})}(\mathcal{T},X^{\circ })=0\}$

According to $[Mi]$, Prop. 9.8, for each $Y^{\circ }\in D^{b}(Qgr_{B_{n}})$
and $X^{\circ }\in \mathcal{F}$, \linebreak\ $Hom_{D^{b}(Qgr_{B_{n}})}(Y^{%
\circ },X^{\circ })=Hom_{D^{b}(Qgr_{B_{n}})/\mathcal{T}}(QY^{\circ
},QX^{\circ })\cong $

$Hom_{D^{b}(gr_{(B_{n})_{Z}})}(\psi Y^{\circ },\psi X^{\circ })$.

In particular there is a full embedding of $\mathcal{F}$ in $%
D^{b}(gr_{(B_{n})_{Z}}).$

According to $[MM]$ and $[MS]$ there is a duality of triangulated categories 
$\overline{\phi }:\underline{gr}_{_{B_{n}^{!}}}\rightarrow
D^{b}(Qgr_{B_{n}^{op}})$ induced by the duality $\phi
:gr_{_{_{B_{n}^{!}}}}\rightarrow \mathcal{LC}P_{B_{n}}$, with $\mathcal{LC}%
P_{B_{n}}$ the category of linear complexes of graded projective $B_{n}$%
-modules. If $M=\underset{i\geq k_{0}}{\oplus }M_{i}$ is a graded $%
B_{n}^{!}- $module, then $\phi (M)$ is a complex of the form:

$D(M)\otimes B_{n}:\rightarrow ...D(M_{k_{0}+n})\underset{B_{0}}{\otimes }%
B_{n}[-k_{0}-n]\rightarrow D(M_{k_{0}+n-1})\underset{B_{0}}{\otimes }%
B_{n}[-k_{0}-n+1]\rightarrow ...D(M_{k_{0}+1}\underset{B_{0}}{\otimes }%
B_{n})[-k_{0}-1]\rightarrow D(M_{k_{0}})\underset{B_{0}}{\otimes }%
B_{n}[-k_{0}]\rightarrow 0.$

$\overline{\phi }(M)$ is the complex:

$:\rightarrow ...\pi (D(M_{k_{0}+n})\underset{B_{0}}{\otimes }%
B_{n})[-k_{0}-n]\rightarrow \pi (D(M_{k_{0}+n-1})\underset{B_{0}}{\otimes }%
B_{n})[-k_{0}-n+1]\rightarrow ...\pi (D(M_{k_{0}+1})\underset{B_{0}}{\otimes 
}B_{n})[-k_{0}-1]\rightarrow \pi (D(M_{k_{0}})\underset{B_{0}}{\otimes }%
B_{n})[-k_{0}]\rightarrow 0$

If we compose with the usual duality we obtain an equivalence of
triangulated categories: $\overline{\phi }D:\underline{gr}%
_{_{B_{n}^{!}}}\rightarrow D^{b}(Qgr_{B_{n}}).$

Under the duality there is a pair $(\mathcal{F}^{\prime },\mathcal{T}%
^{\prime }\mathcal{)}$ such that $\mathcal{F}^{\prime }\rightarrow \mathcal{F%
}$ and $\mathcal{T}^{\prime }\rightarrow \mathcal{T}$ corresponds to the
pair: $(\mathcal{T}$, $\mathcal{F}).$

We want to characterize the subcategories $\mathcal{F}^{\prime },\mathcal{T}%
^{\prime }$of $\underline{gr}_{_{B_{n}^{!}}}$.

We shall start by recalling some properties of the finitely generated graded 
$B_{n}$-modules.

The algebra $B_{n}$ is a Koszul algebra of finite global dimension, under
such conditions, for any finitely generated graded $B_{n}$-module $M$ there
is a truncation $M_{\geq k}$ such that $M_{\geq k}[k]$ is Koszul $[M]$. But
in $Qgr_{B_{n}}$ the objects $\pi M$ and $\pi M_{\geq k}$ are isomorphic,
hence we can consider only Koszul $B_{n}$-modules and their shifts. Assume $%
M $ is finitely generated but of infinite dimension over $K$. The torsion
part $t(M)$ is finite dimensional over $K$, hence there is a torsion free
truncation $M_{\geq k}$ of $M$, so we may assume $M$ torsion free and Koszul.

Let%
\'{}%
s suppose $M$ is of $Z$-torsion.There exists an integer $n$ such that $%
Z^{n-1}M\neq 0$ and $Z^{n}M=0$. There is a filtration of $M$: $M\supset
ZM\supset Z^{2}M...\supset Z$ $^{n-1}M\supset 0$. Since $Z$ is an element of
degree one $(ZM)_{i}=ZM_{i-1}$, which implies $(Z^{j}M)_{\geq k}=Z^{j}(M$ $%
_{\geq k-j})$.

Truncation of Koszul is Koszul and we can take large enough truncation in
order to have $(Z^{j}M)_{\geq k}$ Koszul for all $j.$Changing $M$ for $%
M_{\geq k}$ we may assume all $Z^{j}M$ are Koszul.

There is a commutative exact diagram:

\begin{center}
$%
\begin{array}{ccccccccc}
&  &  & \overset{0}{\downarrow } &  & \overset{0}{\downarrow } &  &  &  \\ 
&  & 0\rightarrow & \Omega (M) & \rightarrow & \Omega (M/ZM) & \rightarrow & 
ZM & \rightarrow 0 \\ 
&  &  & \downarrow &  & \downarrow &  &  &  \\ 
&  &  & P & \overset{1}{\rightarrow } & P &  &  &  \\ 
&  &  & \downarrow &  & \downarrow &  &  &  \\ 
0\rightarrow & ZM & \rightarrow & M & \rightarrow & M/ZM & \rightarrow 0 & 
&  \\ 
&  &  & \underset{0}{\downarrow } &  & \underset{0}{\downarrow } &  &  & 
\end{array}%
$
\end{center}

the modules $\Omega (M)$, $ZM$ are Koszul generated in the same degree, it
follows $M/ZM$ is Koszul and for any integer $k\geq 1$ there is an exact
sequence:

$%
\begin{array}{ccccccc}
0\rightarrow & \Omega ^{k}(M) & \rightarrow & \Omega ^{k}(M/ZM) & \rightarrow
& \Omega ^{k-1}(ZM) & \rightarrow 0%
\end{array}%
.$By $[$GM1$]$ there is an exact sequence:

$0\rightarrow Hom_{B_{n}}(\Omega ^{k-1}(ZM),B_{_{n}0})\rightarrow
Hom_{B_{n}}(\Omega ^{k}(M/ZM),B_{_{n}0})\rightarrow $

$Hom_{B_{n}}(\Omega ^{k}(M),B_{_{n}0})\rightarrow 0$ or an exact sequence:

*) $0\rightarrow Ext_{B_{n}}^{k-1}(ZM,B_{_{n}0})\rightarrow
Ext_{B_{n}}^{k}(M/ZM,B_{_{n}0})\rightarrow
Ext_{B_{n}}^{k}(M,B_{_{n}0})\rightarrow 0$.

We will denote by $F_{B_{n}}(N)=\underset{k\geq 0}{\oplus }%
Ext_{B_{n}}^{k}(N,B_{_{n}0})$ the Koszul duality functor $%
F_{B_{n}}:K_{B_{n}}\rightarrow K_{B_{n}^{!}}$ .

Adding all sequences *) we obtain an exact sequence:

$0\rightarrow F_{B_{n}}(ZM)[-1]\rightarrow F_{B_{n}}(M/ZM)\rightarrow
F_{B_{n}}(M)\rightarrow 0$

We can apply the same argument to any module $Z^{j}M$ to get an exact
sequence:

$0\rightarrow F_{B_{n}}(Z^{j+1}M)[-j-1]\rightarrow
F_{B_{n}}(Z^{j}M/Z^{j+1}M)[-j]\rightarrow F_{B_{n}}(Z^{j}M)[-j]\rightarrow
0. $

Gluing all short exact sequences we obtain a long exact sequence of Koszul
up to shifting $B_{n}^{!}$-modules:

**) $0\rightarrow F_{B_{n}}(Z^{n-1}M)[-n+1]\rightarrow
F_{B_{n}}(Z^{n-2}M/Z^{n-1}M)[-n+2]...\rightarrow $

$F_{B_{n}}(M/ZM)\rightarrow F_{B_{n}}(M)\rightarrow 0$

It will be enough to study non semismple Koszul $B_{n}$-modules $N$ such
that $ZN=0$. They can be considered as $C_{n}$-modules.

We have the following commutative exact diagram:

\begin{center}
$%
\begin{array}{ccccccc}
& \overset{0}{\downarrow } &  & \overset{0}{\downarrow } &  &  &  \\ 
& ZB_{n}^{n_{0}} & \overset{1}{\rightarrow } & ZB_{n}^{n_{0}} &  &  &  \\ 
& \downarrow &  & \downarrow &  &  &  \\ 
0\rightarrow & \Omega _{B}(N) & \rightarrow & B_{n}^{n_{0}} & \rightarrow & N
& \rightarrow 0 \\ 
& \downarrow &  & \downarrow &  & \downarrow 1 &  \\ 
0\rightarrow & \Omega _{C}(N) & \rightarrow & C_{n}^{n_{0}} & \rightarrow & N
& \rightarrow 0 \\ 
& \underset{0}{\downarrow } &  & \underset{0}{\downarrow } &  &  & 
\end{array}%
$
\end{center}

The algebra $B_{n}$ is an integral domain and in consequence the free $B_{n}$%
-modules are torsion free and $ZB_{n}^{n_{0}}$is isomorphic to $%
B_{n}^{n_{0}}[-1]$.

The exact sequence: $0\rightarrow ZB_{n}^{n_{0}}\rightarrow \Omega
_{B}(N)\rightarrow \Omega _{C}(N)\rightarrow 0$ consists of graded modules
generated in degree one and the first two term are Koszul, by $[GM]$ this
implies $\Omega _{C}(N)$ is Koszul as $B_{n}$-module.

There is a commutative exact diagram:

\begin{center}
$%
\begin{array}{ccccccc}
&  &  & \overset{0}{\downarrow } &  & \overset{0}{\downarrow } &  \\ 
& 0 & \rightarrow & B_{n}^{n_{0}}[-1] & \rightarrow & ZB_{n}^{n_{0}} & 
\rightarrow 0 \\ 
& \downarrow &  & \downarrow &  & \downarrow &  \\ 
0\rightarrow & \Omega _{B}^{2}(N) & \rightarrow & B_{n}^{n_{0}+n_{1}}[-1] & 
\rightarrow & \Omega _{B}(N) & \rightarrow 0 \\ 
& \downarrow &  & \downarrow &  & \downarrow &  \\ 
0\rightarrow & \Omega _{B}\Omega _{C}(N) & \rightarrow & B_{n}^{n_{1}}[-1] & 
\rightarrow & \Omega _{C}(N) & \rightarrow 0 \\ 
& \underset{0}{\downarrow } &  & \underset{0}{\downarrow } &  & \underset{0}{%
\downarrow } & 
\end{array}%
$
\end{center}

In particular $\Omega _{B}^{2}(N)\cong \Omega _{B}\Omega _{C}(N).$

Since $\Omega _{C}(N)\subset C_{n}^{n_{0}}$it is a $C_{n}$-module and we
have the following commutative exact diagram:

\begin{center}
$%
\begin{array}{ccccccc}
& \overset{0}{\downarrow } &  & \overset{0}{\downarrow } &  &  &  \\ 
& ZB_{n}^{n_{1}}[-1] & \overset{1}{\rightarrow } & ZB_{n}^{n_{1}}[-1] &  & 
&  \\ 
& \downarrow &  & \downarrow &  &  &  \\ 
0\rightarrow & \Omega _{B}\Omega _{C}(N) & \rightarrow & B_{n}^{n_{1}}[-1] & 
\rightarrow & \Omega _{C}(N) & \rightarrow 0 \\ 
& \downarrow &  & \downarrow &  & \downarrow 1 &  \\ 
0\rightarrow & \Omega _{C}^{2}(N) & \rightarrow & C_{n}^{n_{1}} & \rightarrow
& \Omega _{C}(N) & \rightarrow 0 \\ 
& \underset{0}{\downarrow } &  & \underset{0}{\downarrow } &  &  & 
\end{array}%
$
\end{center}

and an exact sequence: $0\rightarrow B_{n}^{n_{1}}[-2]\rightarrow \Omega
_{B}^{2}(N)\rightarrow \Omega _{C}^{2}(N)\rightarrow 0$.

\bigskip In general there exist exact sequences:

$0\rightarrow B_{n}^{n_{k-1}}[-k]\rightarrow \Omega _{B}^{k}(N)\rightarrow
\Omega _{C}^{k}(N)\rightarrow 0$.

which induce exact sequences:\newline
$0\rightarrow $Hom$_{B_{n}}(\Omega _{C}^{k}$(N),B$_{_{n}0}$)$\rightarrow $Hom%
$_{B_{n}}$($\Omega _{B}^{k}$(N),B$_{_{n}0}$)$\rightarrow $Hom$_{B_{n}}$(B$%
_{n}^{n_{k-1}}$[-k],B$_{_{n}0}$)$\rightarrow 0.$

The module $\Omega _{C}^{k}(N)$ is annihilated by $Z$ which implies $%
J_{B}\Omega _{C}^{k}(N)=J_{C}\Omega _{C}^{k}(N).$The module $B_{_{n}0}\cong
C_{_{n}0}\cong K$.

There are isomorphisms:\newline
$Hom_{B_{n}}(\Omega _{C}^{k}(N),B_{_{n}0})\cong Hom_{B_{n}0}(\Omega
_{C}^{k}(N)/J_{B}\Omega _{C}^{k}(N),B_{_{n}0})\cong $

$Hom_{C_{n}0}(\Omega _{C}^{k}(N)/J_{c}\Omega _{C}^{k}(N),C_{_{n}0})\cong
Hom_{C_{n}}(\Omega _{C}^{k}(N),C_{_{n}0})\cong Ext_{C_{n}}^{k}(N,C_{_{n}0}).$

We then have an exact sequence: *) $0\rightarrow F_{C_{n}}(N)\overset{\alpha 
}{\rightarrow }$ $F_{B_{n}}(N)\rightarrow \underset{k=1}{\overset{m}{\oplus }%
S^{n_{k-1}}[k]\rightarrow 0}$

\begin{lemma}
The map $\alpha $ is a morphism of $C_{n}^{!}$-modules.
\end{lemma}

\begin{proof}
Let $x$ be an element of $Ext_{C_{n}}^{k}(K,K)$ and $y\in
Ext_{C_{n}}^{k}(N,K)$ we want to prove $\alpha (xy)=x\alpha (y).$

The element $x$ is an extension: $0\rightarrow K\rightarrow E\rightarrow
K\rightarrow 0$ and $y:$ $0\rightarrow K\rightarrow V\rightarrow \Omega
_{C}^{k-1}(N)\rightarrow 0$, the induced map $f$ given below corresponds to $%
y:$

$%
\begin{array}{ccccccc}
0\rightarrow & \Omega _{C}^{k}(N) & \rightarrow & C_{n}^{n_{k-1}} & 
\rightarrow & \Omega _{C}^{k-1}(N) & \rightarrow 0 \\ 
& \downarrow f &  & \downarrow &  & \downarrow 1 &  \\ 
0\rightarrow & K & \rightarrow & V & \rightarrow & \Omega _{C}^{k-1}(N) & 
\rightarrow 0%
\end{array}%
$

Consider the following pull back:

$%
\begin{array}{ccccccc}
0\rightarrow & K & \rightarrow & L & \rightarrow & \Omega _{C}^{k}(N) & 
\rightarrow 0 \\ 
& \downarrow 1 &  & \downarrow &  & \downarrow f &  \\ 
0\rightarrow & K & \rightarrow & E & \rightarrow & K & 
\end{array}%
$

The exact sequence: $0\rightarrow B_{n}^{n_{k-1}}[-k]\rightarrow \Omega
_{B}^{k}(N)\overset{\pi _{k}}{\rightarrow }\Omega _{C}^{k}(N)\rightarrow 0$
induces a pull back of $B_{n}$-modules:

$%
\begin{array}{ccccccc}
0\rightarrow & K & \rightarrow & W & \rightarrow & \Omega _{B}^{k}(N) & 
\rightarrow 0 \\ 
& \downarrow 1 &  & \downarrow &  & \downarrow \pi _{k} &  \\ 
0\rightarrow & K & \rightarrow & L & \rightarrow & \Omega _{C}^{k}(N) & 
\end{array}%
$

It was proved above the existence of commutative exact diagrams:

$%
\begin{array}{ccccccc}
&  &  & \overset{0}{\downarrow } &  & \overset{0}{\downarrow } &  \\ 
& 0 & \rightarrow & B_{n}^{n_{k-2}} & \rightarrow & B_{n}^{n_{k-2}} & 
\rightarrow 0 \\ 
& \downarrow &  & \downarrow &  & \downarrow &  \\ 
0\rightarrow & \Omega _{B}^{k}(N) & \rightarrow & B_{n}^{n_{k-2}+n_{k-1}} & 
\rightarrow & \Omega _{B}^{k-1}(N) & \rightarrow 0 \\ 
& \downarrow &  & \downarrow &  & \downarrow &  \\ 
0\rightarrow & \Omega _{B}^{k}(N) & \rightarrow & B_{n}^{n_{k-1}} & 
\rightarrow & \Omega _{C}^{k-1}(N) & \rightarrow 0 \\ 
& \underset{0}{\downarrow } &  & \underset{0}{\downarrow } &  & \underset{0}{%
\downarrow } & 
\end{array}%
$

and $%
\begin{array}{ccccccc}
0\rightarrow & \Omega _{B}^{k}(N) & \rightarrow & B_{n}^{n_{k-1}} & 
\rightarrow & \Omega _{C}^{k-1}(N) & \rightarrow 0 \\ 
& \downarrow &  & \downarrow &  & \downarrow 1 &  \\ 
0\rightarrow & \Omega _{C}^{k}(N) & \rightarrow & C_{n}^{n_{k-1}} & 
\rightarrow & \Omega _{C}^{k-1}(N) & \rightarrow 0 \\ 
& \underset{0}{\downarrow } &  & \underset{0}{\downarrow } &  &  & 
\end{array}%
$

Gluing diagrams we obtain a commutative diagram with exact rows:\newline
$%
\begin{array}{cccccccccc}
& \text{0}\rightarrow & \Omega _{B}^{k+1}\text{(N)} & \rightarrow & \text{B}%
_{n}^{n_{k}+n_{k-1}} & \rightarrow & \text{B}_{n}^{n_{k-2}+n_{k-1}} & 
\rightarrow & \Omega _{B}^{k-1}\text{(N)} & \rightarrow \text{0} \\ 
&  & \varphi \downarrow &  & \downarrow & \searrow & \downarrow &  & 
\downarrow &  \\ 
\alpha \text{(xy):} & \text{0}\rightarrow & \text{K} & \rightarrow & \text{W}
& \rightarrow & \text{B}_{n}^{n_{k-1}} & \rightarrow & \Omega _{C}^{k-1}%
\text{(N)} & \rightarrow \text{0} \\ 
&  & \downarrow \text{1} &  & \downarrow &  & \downarrow &  & \downarrow 
\text{1} &  \\ 
\text{xy:} & \text{0}\rightarrow & \text{K} & \rightarrow & \text{L} & 
\rightarrow & \text{C}_{n}^{n_{k-1}} & \rightarrow & \Omega _{C}^{k-1}\text{%
(N)} & \rightarrow \text{0} \\ 
&  & \downarrow \text{1} &  & \downarrow &  & \downarrow &  & \downarrow 
\text{1} &  \\ 
\text{xy:} & \text{0}\rightarrow & \text{K} & \rightarrow & \text{E} & 
\rightarrow & \text{V} & \rightarrow & \Omega _{C}^{k-1}\text{(N)} & 
\rightarrow \text{0}%
\end{array}%
$

$\alpha (xy)=\phi \in Ext_{B_{n}}^{k+1}(N,K)$ .

In the other hand we have the following commutative diagram with exact rows:

$%
\begin{array}{ccccccc}
\text{0}\rightarrow & \Omega _{B}^{k+1}\text{(N)} & \rightarrow & \text{B}%
_{n}^{n_{k}+n_{k-1}} & \rightarrow & \Omega _{B}^{k}\text{(N)} & \rightarrow 
\text{0} \\ 
& \downarrow \text{1} &  & \downarrow &  & \downarrow \pi _{k} &  \\ 
\text{0}\rightarrow & \Omega _{B}^{k+1}\text{(N)} & \rightarrow & \text{B}%
_{n}^{n_{k}} & \rightarrow & \Omega _{C}^{k}\text{(N)} & \rightarrow \text{0}
\\ 
& \downarrow \pi _{k+1} &  & \downarrow &  & \downarrow \text{1} &  \\ 
\text{0}\rightarrow & \Omega _{C}^{k+1}\text{(N)} & \rightarrow & \text{C}%
_{n}^{n_{k}} & \rightarrow & \Omega _{C}^{k}\text{(N)} & \rightarrow \text{0}
\\ 
& \downarrow \varphi ^{\prime } &  & \downarrow &  & \downarrow \text{1} & 
\\ 
\text{0}\rightarrow & \text{K} & \rightarrow & \text{L} & \rightarrow & 
\Omega _{C}^{k}\text{(N)} & \rightarrow \text{0} \\ 
& \downarrow \text{1} &  & \downarrow &  & \downarrow \text{f} &  \\ 
\text{0}\rightarrow & \text{K} & \rightarrow & \text{E} & \rightarrow & 
\text{K} & \rightarrow \text{0}%
\end{array}%
$

Gluing diagrams we obtain the following commutative exact diagrams:

$%
\begin{array}{ccccccc}
\text{0}\rightarrow & \Omega _{B}^{k+1}\text{(N)} & \rightarrow & \text{B}%
_{n}^{n_{k}+n_{k-1}} & \rightarrow & \Omega _{B}^{k}\text{(N)} & \rightarrow 
\text{0} \\ 
& \downarrow \varphi ^{\prime }\pi _{k+1} &  & \downarrow &  & \downarrow 
\text{f}\pi _{k} &  \\ 
\text{0}\rightarrow & \text{K} & \rightarrow & \text{E} & \rightarrow & 
\text{K} & \rightarrow \text{0}%
\end{array}%
$

and $%
\begin{array}{ccccccc}
\text{0}\rightarrow & \Omega _{B}^{k+1}\text{(N)} & \rightarrow & \text{B}%
_{n}^{n_{k}+n_{k-1}} & \rightarrow & \Omega _{B}^{k}\text{(N)} & \rightarrow 
\text{0} \\ 
& \downarrow \varphi &  & \downarrow &  & \downarrow \text{f}\pi _{k} &  \\ 
\text{0}\rightarrow & \text{K} & \rightarrow & \text{E} & \rightarrow & 
\text{K} & \rightarrow \text{0}%
\end{array}%
$

The map $\varphi ^{\prime }$ corresponds with $xf\in Ext_{C_{n}}^{k+1}(N,K)$%
, $\alpha (y)=\pi _{k}f$, $x\alpha (y)=\varphi .$

Then we have: $\alpha (xy)=\alpha (xf)=\varphi ^{\prime }\pi _{k+1}=\Omega
(f\pi _{k})=\varphi =x\alpha (y).$
\end{proof}

\begin{lemma}
There is an isomorphism: $B_{n}^{!}F_{C}(M)=F_{B}(M)$.
\end{lemma}

\begin{proof}
Let $x$ be an element of $Ext_{B_{n}}^{k}(M,K)$ and $\varphi $ the
corresponding morphism: $\varphi :\Omega _{B}^{k}(M)\rightarrow K.$As above
there exists the following commutative exact diagram:

$%
\begin{array}{ccccccc}
&  &  & \overset{0}{\downarrow } &  & \overset{0}{\downarrow } &  \\ 
& 0 & \rightarrow & B_{n}^{n_{k-2}}[-k+1] & \rightarrow & 
B_{n}^{n_{k-2}}[-k+1] & \rightarrow 0 \\ 
& \downarrow &  & \downarrow &  & \downarrow &  \\ 
0\rightarrow & \Omega _{B}^{k}(M) & \rightarrow & 
B_{n}^{n_{k-2}+n_{k-1}}[-k+1] & \rightarrow & \Omega _{B}^{k-1}(M) & 
\rightarrow 0 \\ 
& \downarrow &  & \downarrow &  & \downarrow &  \\ 
0\rightarrow & \Omega _{B}\Omega _{C}^{k-1}(M) & \rightarrow & 
B_{n}^{n_{k-1}}[-k+1] & \rightarrow & \Omega _{C}^{k-1}(M) & \rightarrow 0
\\ 
& \underset{0}{\downarrow } &  & \underset{0}{\downarrow } &  & \underset{0}{%
\downarrow } & 
\end{array}%
$

Since the module $\Omega _{B}^{k}(M)$ is generated in degree $k,$there
exists an exact sequence of graded modules generated in degree $k$:

$0\rightarrow \Omega _{B}^{k}(M)\rightarrow
JB_{n}^{n_{k-1}}[-k+1]\rightarrow J\Omega _{C}^{k-1}(M)\rightarrow 0$, which
in turn induces an exact sequence: $0\rightarrow J\Omega
_{B}^{k}(M)\rightarrow J^{2}B_{n}^{n_{k-1}}[-k+1]\rightarrow J^{2}\Omega
_{C}^{k-1}(M)\rightarrow 0$ and there exists a commutative exact diagram:%
\newline
$%
\begin{array}{ccccccc}
& \overset{0}{\downarrow } &  & \overset{0}{\downarrow } &  & \overset{0}{%
\downarrow } &  \\ 
\text{0}\rightarrow & \text{J}\Omega _{B}^{k}\text{(M)} & \rightarrow & 
\text{J}^{2}\text{B}_{n}^{n_{k-1}} & \rightarrow & \text{J}^{2}\Omega
_{C}^{k-1}\text{(M)} & \rightarrow \text{0} \\ 
& \downarrow &  & \downarrow &  & \downarrow &  \\ 
\text{0}\rightarrow & \Omega _{B}^{k}\text{(M)} & \overset{j}{\rightarrow }
& \text{JB}_{n}^{n_{k-1}} & \rightarrow & \text{J}\Omega _{C}^{k-1}\text{(M)}
& \rightarrow \text{0} \\ 
& \downarrow \pi &  & \downarrow \overline{\pi } &  & \downarrow &  \\ 
\text{0}\rightarrow & \Omega _{B}^{k}\text{(M)/J}\Omega _{B}^{k}\text{(M)} & 
\underset{\overline{q}}{\overset{\overline{j}}{\rightleftarrows }} & \text{JB%
}_{n}^{n_{k-1}}\text{/J}^{2}\text{B}_{n}^{n_{k-1}} & \rightarrow & \text{J}%
\Omega _{C}^{k-1}\text{(M)/J}^{2}\Omega _{C}^{k-1}\text{(M)} & \rightarrow 
\text{0} \\ 
& \underset{0}{\downarrow } &  & \underset{0}{\downarrow } &  & \underset{0}{%
\downarrow } & 
\end{array}%
$

Since $\overline{q}\overline{j}=1$, it follows $\overline{q}\overline{\pi }%
j= $ $\overline{q}\overline{j}\pi =\pi $.

Being $K$ semisimple, the map $\varphi $ factors as follows:

$%
\begin{array}{ccccc}
\Omega _{B}^{k}(M) &  & \overset{\varphi }{\rightarrow } &  & K \\ 
& \pi \searrow &  & \nearrow t &  \\ 
&  & \Omega _{B}^{k}(M)/J\Omega _{B}^{k}(M) &  & 
\end{array}%
.$

Set $f=t\overline{q}\overline{\pi }$, $f:$ $JB_{n}^{n_{k-1}}[-k+1]%
\rightarrow K$. Then $fj=t\overline{q}\overline{\pi }j=t\pi =\varphi .$

Consider the commutative diagram with exact rows:

$%
\begin{array}{cccccccc}
& 0\rightarrow & \Omega _{B}\Omega _{C}^{k-1}(M) & \rightarrow & 
B_{n}^{n_{k-1}} & \rightarrow & \Omega _{C}^{k-1}(M) & \rightarrow 0 \\ 
&  & \downarrow j &  & \downarrow 1 &  & \downarrow \rho &  \\ 
& 0\rightarrow & JB^{n_{k-1}}[-k+1] & \rightarrow & B_{n}^{n_{k-1}} & 
\rightarrow & \Omega _{C}^{k-1}(M)/J\Omega _{C}^{k-1}(M) & \rightarrow 0 \\ 
&  & \downarrow f &  & \downarrow &  & \downarrow \cong &  \\ 
x: & 0\rightarrow & K & \rightarrow & E & \rightarrow & \underset{n_{k-1}}{%
\oplus K} & \rightarrow 0%
\end{array}%
.$

The map corresponding to the last column is: $p=\left[ 
\begin{array}{c}
p_{1} \\ 
\overset{.}{\underset{.}{.}} \\ 
p_{n_{k-1}}%
\end{array}%
\right] :\Omega _{C}^{k-1}(M)\rightarrow \underset{n_{k-1}}{\oplus K}$. Each 
$p_{i}:\Omega _{C}^{k-1}(M)\rightarrow K$ corresponds to an element of $%
Ext_{C_{n}}^{k-1}(M,K).$

$x=(x_{1},x_{2}...x_{n_{k-1}})$ and each $x_{i}$ is an extension: $%
0\rightarrow K\rightarrow E_{i}\rightarrow K\rightarrow 0.$Taking pull backs:

$%
\begin{array}{cccccccc}
x_{i}p_{i}: & 0\rightarrow & K & \rightarrow & L_{i} & \rightarrow & \Omega
_{C}^{k-1}(M) & \rightarrow 0 \\ 
&  & \downarrow 1 &  & \downarrow &  & \downarrow p_{i} &  \\ 
x_{i}: & 0\rightarrow & K & \rightarrow & E_{i} & \rightarrow & K & 
\rightarrow 0%
\end{array}%
$

where each $x_{i}p_{i}\in B_{n}^{!}Ext_{C_{n}}^{k-1}(M,K)$ and $xp=\sum
x_{i}p_{i}\in B_{n}^{!}Ext_{C_{n}}^{k-1}(M,K).$

There is also the following induced diagram with exact rows:

$%
\begin{array}{cccccccc}
& 0\rightarrow & \Omega _{B}\Omega _{C}^{k-1}(M) & \rightarrow & 
B_{n}^{n_{k-1}}[k+1] & \rightarrow & \Omega _{C}^{k-1}(M) & \rightarrow 0 \\ 
&  & \downarrow h &  & \downarrow &  & \downarrow 1 &  \\ 
& 0\rightarrow & K & \rightarrow & L & \rightarrow & \Omega _{C}^{k-1}(M) & 
\rightarrow 0 \\ 
&  & \downarrow 1 &  & \downarrow &  & \downarrow p &  \\ 
x: & 0\rightarrow & K & \rightarrow & E & \rightarrow & \underset{n_{k-1}}{%
\oplus K} & \rightarrow 0%
\end{array}%
$

Gluing the diagrams we obtain the following commutative exact diagram:

$%
\begin{array}{cccccccc}
& 0\rightarrow & \Omega _{B}\Omega _{C}^{k-1}(M) & \rightarrow & 
B_{n}^{n_{k-1}}[k+1] & \rightarrow & \Omega _{C}^{k-1}(M) & \rightarrow 0 \\ 
&  & \downarrow h &  & \downarrow &  & \downarrow p &  \\ 
x: & 0\rightarrow & K & \rightarrow & E & \rightarrow & \underset{n_{k-1}}{%
\oplus K} & \rightarrow 0%
\end{array}%
$

It follows $h=fj=\varphi $ up to homotopy.

We have proved $B_{n}^{!}Ext_{C_{n}}^{k-1}(M,K)=Ext_{B_{n}}^{k}(M,K).$

It follows by induction $B_{n}^{!}F_{C}(M)=F_{B}(M)$.
\end{proof}

We can prove now the following:

\begin{proposition}
Let $M$ be a Koszul non semisimple $B_{n}$-module with $ZM=0,$ $%
F_{B}:K_{B}\rightarrow K_{B^{!}}$, $F_{C}:K_{C}\rightarrow K_{C^{!}}$ Koszul
dualities. Then there is an isomorphism of $B_{n}^{!}$-modules: $B_{n}^{!}%
\underset{C^{!}}{\otimes }F_{C}(M)\cong F_{B}(M)$.
\end{proposition}

\begin{proof}
We proved in the previous lemma that the map $\mu :B_{n}^{!}\underset{C^{!}}{%
\otimes }F_{C}(M)\rightarrow F_{B}(M)$ given by multiplication is surjective
and we know that $B_{n}^{!}=C_{n}^{!}\oplus ZC_{n}^{!}$, so there is a
splittable sequence of $C_{n}^{!}$-modules: $0\rightarrow
C_{n}^{!}\rightarrow B_{n}^{!}\rightarrow ZC_{n}^{!}\rightarrow 0$ which
induces a commutative exact diagram:

$%
\begin{array}{ccccccc}
0\rightarrow & C_{n}^{!}\underset{C^{!}}{\otimes }F_{C}(M) & \rightarrow & 
B_{n}^{!}\underset{C^{!}}{\otimes }F_{C}(M) & \rightarrow & ZC_{n}^{!}%
\underset{C^{!}}{\otimes }F_{C}(M) & \rightarrow 0 \\ 
& \downarrow \cong &  & \downarrow \mu &  & \downarrow \mu ^{\prime \prime }
&  \\ 
0\rightarrow & F_{C}(M) & \rightarrow & F_{B}(M) & \rightarrow & \overset{m}{%
\oplus }S^{n_{k-1}}[k] & \rightarrow 0 \\ 
&  &  & \underset{0}{\downarrow } &  & \underset{0}{\downarrow } & 
\end{array}%
$

By dimensions $\mu ^{\prime \prime }$ is an isomorphism, therefore $\mu $ is
an isomorphism.
\end{proof}

It as proved in $[MM]$, $[MS]$ that the duality $\phi
:gr_{_{_{B_{n}^{!}}}}\rightarrow \mathcal{LC}P_{B_{n}}$, with $\mathcal{LC}%
P_{B_{n}}$ the category of linear complexes of graded projective $B_{n}$%
-modules, induces a duality of triangulated categories $\overline{\phi }:%
\underline{gr}_{_{B_{n}^{!}}}\rightarrow D^{b}(Qgr_{B_{n}^{op}})$. In
particular given a complex $\pi X^{\circ }$ in $D^{b}(Qgr_{B_{n}^{op}})$,
there is a totally linear complex (see $[MM]$ for definition) $Y^{\circ }$
such that $\pi Y^{\circ }$ is isomorphic to $\pi X^{\circ }$, Moreover, $%
Y^{\circ }$ is quasi isomorphic to a linear complex of projectives $P^{\circ
}$ and by $[MS]$ $P^{\circ }=\phi (M).$Therefore $\pi X^{\circ }$ is quasi
isomorphic to $\pi \phi (M).$

It was proved in $[MS]$ that $H^{i}(\phi (M))=0$ for all $i\neq 0$ if and
only if $M$ is Koszul and in this case if $G_{B_{n}^{!}}:K_{B_{n}^{!}}%
\rightarrow K_{B_{n}}$ is Koszul duality, then $H^{0}(\phi (M))\cong
G_{B_{n}^{!}}(M)$.

Since $B_{n}^{!}$ is a finite dimensional algebra, it follows by $[MZ]$ that
for any finitely generated $B_{n}^{!}$-module $M$ there exists an integer $%
k\geq 0$ such that $\Omega ^{k}M$ is weakly Koszul.

since $\Omega $ is the shift in the triangulated category $\underline{gr}%
_{_{B_{n}^{!}}}$and $\overline{\phi }$ is a duality it follows $\overline{%
\phi }(\Omega ^{k}M)\cong \overline{\phi }(M)[k].$ Being the category $%
\mathcal{T}$ triangulated, it is invariant under shift and $\pi \phi (M)\in 
\mathcal{T}$ if and only if $\pi \phi (\Omega ^{k}M)\in \mathcal{T}$.

We may assume $M$ is weakly Koszul and $M=\underset{i\geq 0}{\sum }M_{i}$, $%
M_{0}\neq 0$. By $[MZ]$ there exists an exact sequence: $0\rightarrow
K_{M}\rightarrow M\rightarrow L\rightarrow 0$ with $K_{M}<M_{0}>$ generated
by the degree zero part of $M$, $K_{M}$ is Koszul and $J^{j}K_{M}=J^{j}M\cap
K_{M}$ for all $j>0.$

Being $\phi $ an exact functor there is an exact sequence: $0\rightarrow
\phi (L)\rightarrow \phi (M)\rightarrow \phi (K_{M})\rightarrow 0$ of
complexes of $B_{n}$-modules, which induces a long exact sequence:

...$\rightarrow H^{1}(\phi (L))\rightarrow $ $H^{1}(\phi (M))\rightarrow
H^{1}(\phi (K_{M}))\rightarrow H^{0}(\phi (L))\rightarrow H^{0}(\phi
(M))\rightarrow H^{0}(\phi (K_{M}))\rightarrow 0$

where $H^{0}(\phi (M))\cong H^{0}(\phi (K_{M}))$ and $H^{i}(\phi (L))\cong
H^{i}(\phi (M)$ for all $i\neq 0$. Being $K_{M}$ Koszul $H^{0}(\phi
(K_{M}))\cong G_{B_{n}^{!}}(K_{M})$ and $G_{B_{n}^{!}}(K_{M})$ is of $Z$%
-torsion.

According to $[MZ]$ there is a filtration: $M=U_{p}\supset U_{p-1}\supset
...U_{1}\supset U_{0}=K_{M}$ such that $U_{i}/U_{i-1}$is Koszul and $%
J^{k}U_{i}\cap U_{i-1}=J^{k}U_{i-1}.$

The module $L$ is weakly Koszul and it has a filtration: $L=$ $%
U_{p}/U_{0}\supset U_{p-1}/U_{0}\supset ...U_{1}/U_{0}$ with factors Koszul,
it follows by induction each $G_{B_{n}^{!}}(U_{i}/U_{i-1})=V_{i}$ is a
Koszul $B_{n}$-module of $Z$-torsion.

Each $V_{i}$ has a filtration: $V_{i}\supset ZV_{i}\supset Z^{2}V_{i}\supset
...\supset Z^{k_{i}}V_{i}\supset 0$, $Z^{k_{i}}V_{i}\neq 0$, $%
Z^{k_{i}+1}V_{i}=0$. After a truncation $V_{i\geq n_{i}}$we may assume all $%
Z^{j}V_{i}$ Koszul. But $V_{i\geq n_{i}}=J^{n_{i}}V_{i}\cong
G_{B_{n}^{!}}(\Omega ^{n_{i}}(U_{i}/U_{i-1}))$. Taking $n=\max \{n_{i}\}$ we
change $M$ for $\Omega ^{n}(M),$which is weakly Koszul with filtration: $%
\Omega ^{n}M=\Omega ^{n}U_{p}\supset \Omega ^{n}U_{p-1}\supset ...\Omega
^{n}U_{1}\supset \Omega ^{n}U_{0}.$

We may assume all $Z^{j}V_{i}$ are Koszul. There exist exact sequences:

*) $0\rightarrow F_{B_{n}}(Z^{k_{i}}V_{i})[-k_{i}]\rightarrow
F_{B_{n}}(Z^{k_{i}-1}V_{i}/Z^{k_{i}}V_{i})[-k_{i}+1]...\rightarrow $

$F_{B_{n}}(V_{i}/ZV_{i})\rightarrow U_{i}/U_{i-1}\rightarrow 0$

where each $F_{B_{n}}(Z^{j}V_{i}/Z^{j+1}V_{i})$ $\cong B_{n}^{!}\underset{%
C_{n}^{!}}{\otimes }X_{ij}$ is an induced module of a Koszul $C_{n}^{!}$%
-module $X_{ij}.$

\begin{lemma}
Let $R$ be a $Z$-graded $K$-algebra, with $K$ a field, $M$ a graded left $R$%
-module and $N$ a graded right $R$-module. Then $M\underset{R}{\otimes }N$
is a graded $K$-module such that $M\underset{R}{\otimes }N[j]\cong (M%
\underset{R}{\otimes }N)[j]$ as graded $K$-modules.
\end{lemma}

\begin{proof}
Recall the definition of the graded tensor product $[Mac]:$

Let $\psi :M\underset{K}{\otimes }R\underset{K}{\otimes }N\rightarrow M%
\underset{K}{\otimes }N$ be the map: $\psi (m\otimes r\otimes n)=mr\otimes
n-m\otimes rn$. Then $Cok\psi =M\underset{R}{\otimes }N.$

The $K$- module $M\underset{K}{\otimes }N$ has grading: ($M\underset{K}{%
\otimes }N)_{k}=\underset{i+j=k}{\sum }$ $M_{i}\underset{K}{\otimes }N_{j}$.
It follows $M\underset{K}{\otimes }N$ $[j]\cong (M\underset{K}{\otimes }N$ )$%
[j]$.

and there is an isomorphism of exact sequences:

$%
\begin{array}{cccccc}
M\underset{K}{\otimes }R\underset{K}{\otimes }(N[j]) & \rightarrow & M%
\underset{K}{\otimes }(N[j]) & \rightarrow & M\underset{R}{\otimes (}N[j]) & 
\rightarrow 0 \\ 
\downarrow \cong &  & \downarrow \cong &  & \downarrow \cong &  \\ 
(M\underset{K}{\otimes }R\underset{K}{\otimes }N)[j] & \rightarrow & (M%
\underset{K}{\otimes }N)[j] & \rightarrow & (M\underset{R}{\otimes }N)[j] & 
\rightarrow 0%
\end{array}%
$
\end{proof}

\begin{lemma}
Let $B_{n}^{!}$ and $C_{n}^{!}$ be the algebras given above, for any
finitely generated graded $C_{n}^{!}$-module $M$ there is an isomorphism: $%
\Omega _{B_{n}^{!}}(B_{n}^{!}\underset{C_{n}^{!}}{\otimes }M)\cong B_{n}^{!}%
\underset{C_{n}^{!}}{\otimes }\Omega _{C_{n}^{!}}(M).$

\begin{proof}
Let $0\rightarrow \Omega _{C_{n}^{!}}(M)\rightarrow F\rightarrow
M\rightarrow 0$ be an exact sequence with $F$ free of rank $r$, the graded
projective cover of $M$. Then $\Omega _{C_{n}^{!}}(M)\subset J_{C_{n}^{!}}F. 
$

We proved $B_{n}^{!}=C_{n}^{!}\oplus ZC_{n}^{!}$, therefore $%
J_{B_{n}^{!}}=J_{C_{n}^{!}}+ZC_{n}^{!}.$ It follows: $B_{n}^{!}\underset{%
C_{n}^{!}}{\otimes }J_{C_{n}^{!}}=C_{n}^{!}\underset{C_{n}^{!}}{\otimes }%
J_{C_{n}^{!}}+ZC_{n}^{!}\underset{C_{n}^{!}}{\otimes }%
J_{C_{n}^{!}}=J_{C_{n}^{!}}+Z\underset{C_{n}^{!}}{\otimes }%
J_{C_{n}^{!}}\subset J_{B_{n}^{!}}.$

Therefore: $B_{n}^{!}\underset{C_{n}^{!}}{\otimes }\Omega
_{C_{n}^{!}}(M)\subset B_{n}^{!}\underset{C_{n}^{!}}{\otimes \underset{r}{%
\oplus }}J_{C_{n}^{!}}\cong \underset{r}{\oplus }B_{n}^{!}\underset{C_{n}^{!}%
}{\otimes }J_{C_{n}^{!}}\subset \underset{r}{\oplus }J_{B_{n}^{!}}\cong
J_{B_{n}^{!}}(B_{n}^{!}\underset{C_{n}^{!}}{\otimes }F).$

It follows: $0\rightarrow B_{n}^{!}\underset{C_{n}^{!}}{\otimes }\Omega
_{C_{n}^{!}}(M)\rightarrow B_{n}^{!}\underset{C_{n}^{!}}{\otimes }%
F\rightarrow B_{n}^{!}\underset{C_{n}^{!}}{\otimes }M\rightarrow 0$ is exact
and $B_{n}^{!}\underset{C_{n}^{!}}{\otimes }F$ is the graded projective
cover of $B_{n}^{!}\underset{C_{n}^{!}}{\otimes }M$. Then $\Omega
_{B_{n}^{!}}(B_{n}^{!}\underset{C_{n}^{!}}{\otimes }M)\cong B_{n}^{!}%
\underset{C_{n}^{!}}{\otimes }\Omega _{C_{n}^{!}}(M).$
\end{proof}
\end{lemma}

\begin{lemma}
Let $B_{n}^{!}$ and $C_{n}^{!}$ be the algebras given above and let $M$ be a
Koszul $C_{n}^{!}$-module. Then $B_{n}^{!}\underset{C_{n}^{!}}{\otimes }M$
is Koszul and $G_{B_{n}^{!}}(B_{n}^{!}\underset{C_{n}^{!}}{\otimes }M)\cong
G_{C_{n}^{!}}(M)..$
\end{lemma}

\begin{proof}
Let ...$\rightarrow F_{n}[-n]\rightarrow F_{n-1}[-n+1]\rightarrow
...F_{1}[-1]\rightarrow F_{0}\rightarrow M\rightarrow 0$ be a graded
projective resolution of $M,$with each $F_{i}$ free of rank $r_{i}.$%
Tensoring with $B_{n}^{!}\underset{C_{n}^{!}}{\otimes }$ we obtain a graded
projective resolution of $B_{n}^{!}\underset{C_{n}^{!}}{\otimes }M:$ $%
\rightarrow (B_{n}^{!}\underset{C_{n}^{!}}{\otimes }F_{n})[-n]\rightarrow
(B_{n}^{!}\underset{C_{n}^{!}}{\otimes }F_{n-1})[-n+1]\rightarrow
...(B_{n}^{!}\underset{C_{n}^{!}}{\otimes }F_{1})[-1]\rightarrow B_{n}^{!}%
\underset{C_{n}^{!}}{\otimes }F_{0}\rightarrow B_{n}^{!}\underset{C_{n}^{!}}{%
\otimes }M\rightarrow 0$ with each $B_{n}^{!}\underset{C_{n}^{!}}{\otimes }%
F_{i}$ free $B_{n}^{!}$-modules of rank $r_{i}.$

Moreover, $Ext_{B_{n}^{!}}^{n}(B_{n}^{!}\underset{}{\underset{C_{n}^{!}}{%
\otimes }M,K)\cong Hom_{B_{n}^{!}}(\Omega ^{n}}(B_{n}^{!}\underset{C_{n}^{!}}%
{\otimes }M),K)\cong \underset{}{Hom_{B_{n}^{!}}(}B_{n}^{!}\underset{%
C_{n}^{!}}{\otimes }\Omega ^{n}M,K)\cong \underset{}{Hom_{C_{n}^{!}}(}\Omega
^{n}M,K)\cong Ext_{C_{n}^{!}}^{n}(M,K).$

Therefore: $G_{B_{n}^{!}}(B_{n}^{!}\underset{C_{n}^{!}}{\otimes }M)\cong
G_{C_{n}^{!}}(M)$ .
\end{proof}

\begin{remark}
\bigskip To be $G_{B_{n}^{!}}(B_{n}^{!}\underset{C_{n}^{!}}{\otimes }M)$ a $%
C_{n}$-module means: $ZG_{B_{n}^{!}}(B_{n}^{!}\underset{C_{n}^{!}}{\otimes }%
M)=0$.
\end{remark}

We know $B_{n}\cong \underset{m\geq 0}{\oplus }Ext_{B_{n}^{!}}^{m}(K,K)$, $%
C_{n}\underset{m\geq 0}{\oplus }Ext_{C_{n}^{!}}^{m}(K,K)$, and $%
B_{n}/ZB_{n}\cong C_{n}$. Since $C_{n}^{!}$ is a sub algebra of $B_{n}^{!}$,
given an extension $x:0\rightarrow K\rightarrow E_{1}\rightarrow
E_{2}\rightarrow ...E_{n}\rightarrow K$ $\rightarrow 0$ of $B_{n}^{!}$, we
obtain by restriction of scalars an extension $resx:$ $0\rightarrow
resK\rightarrow resE_{1}\rightarrow resE_{2}\rightarrow
...resE_{n}\rightarrow resK$ $\rightarrow 0$ of $C_{n}^{!}$-modules, where $%
resM$ is the module $M$ with multiplication of scalars restricted to $%
C_{n}^{!}.$It is clear $res(xy)=res(x)res(y)$ and restriction gives an
homomorphism of graded $k$-algebras: $res:\underset{m\geq 0}{\oplus }%
Ext_{B_{n}^{!}}^{m}(K,K)\rightarrow \underset{m\geq 0}{\oplus }%
Ext_{C_{n}^{!}}^{m}(K,K)$.

\begin{lemma}
There is an homomorphism: $\rho :Ext_{B_{n}^{!}}^{1}(K,K)\rightarrow
Ext_{B_{n}^{!}}^{1}((B_{n}^{!}\underset{C_{n}^{!}}{\otimes }K,K)$ , given by
the Yoneda product $\rho (x)=x\mu $ (pull back) of the exact sequence $x$
with the multiplication map $\mu :$ $B_{n}^{!}\underset{C_{n}^{!}}{\otimes }%
K\rightarrow K$, such that the composition of the map, $\psi
_{1}:Ext_{B_{n}^{!}}^{1}((B_{n}^{!}\underset{C_{n}^{!}}{\otimes }%
K,K)\rightarrow Ext_{C_{n}^{!}}^{1}(K,K)$ in the previous lemma with $\rho $%
, is the restriction: $\psi \rho $=$res$.
\end{lemma}

\begin{proof}
Let $x$ be the extension: $x:$ $0\rightarrow K\rightarrow E\rightarrow
K\rightarrow 0.$ Since $B_{n}^{!}$ is a free $C_{n}^{!}$-module, there is a
commutative exact diagram

$%
\begin{array}{ccccccc}
0\rightarrow & B_{n}^{!}\underset{C_{n}^{!}}{\otimes K} & \rightarrow & 
B_{n}^{!}\underset{C_{n}^{!}}{\otimes }E & \rightarrow & B_{n}^{!}\underset{%
C_{n}^{!}}{\otimes }K & \rightarrow 0 \\ 
& \downarrow \mu &  & \downarrow \mu &  & \downarrow \mu &  \\ 
0\rightarrow & K & \rightarrow & E & \rightarrow & K & \rightarrow 0%
\end{array}%
$

with $\mu $ multiplication.

This diagram splits in two diagrams:

$%
\begin{array}{ccccccc}
0\rightarrow & B_{n}^{!}\underset{C_{n}^{!}}{\otimes K} & \rightarrow & 
B_{n}^{!}\underset{C_{n}^{!}}{\otimes }E & \rightarrow & B_{n}^{!}\underset{%
C_{n}^{!}}{\otimes }K & \rightarrow 0 \\ 
& \downarrow \mu &  & \downarrow &  & \downarrow 1 &  \\ 
0\rightarrow & K & \rightarrow & W & \rightarrow & B_{n}^{!}\underset{%
C_{n}^{!}}{\otimes }K & \rightarrow 0 \\ 
& \downarrow 1 &  & \downarrow &  & \downarrow \mu &  \\ 
0\rightarrow & K & \rightarrow & E & \rightarrow & K & \rightarrow 0%
\end{array}%
$

Then $\rho (x)=x\mu =\mu (B_{n}^{!}\otimes x).$

For any finitely generated $C_{n}^{!}$-module $M$ there is an isomorphism $%
\alpha $ obtained as the composition of the natural isomorphisms:\newline
$Hom_{B_{n}^{!}}(B_{n}^{!}\underset{C_{n}^{!}}{\otimes }M,K)\cong
Hom_{C_{n}^{!}}(M,Hom_{B_{n}^{!}}(B_{n}^{!},K))\cong Hom_{C_{n}^{!}}(M,K).$

If $j:M\rightarrow B_{n}^{!}\underset{_{C_{n}^{!}}}{\otimes }M$ be the map $%
j(m)=1\otimes m$ and $f:B_{n}^{!}\underset{C_{n}^{!}}{\otimes }M\rightarrow
K $ is any map, then $\alpha (f)=fj.$

Then $\psi \rho (x)=\psi (x\mu )$ is the top sequence in the commutative
exact diagram:

$%
\begin{array}{ccccccc}
0\rightarrow & K & \rightarrow & L & \rightarrow & K & \rightarrow 0 \\ 
& \downarrow 1 &  & \downarrow &  & \downarrow j &  \\ 
0\rightarrow & K & \rightarrow & W & \rightarrow & B_{n}^{!}\underset{%
C_{n}^{!}}{\otimes }K & \rightarrow 0 \\ 
& \downarrow 1 &  & \downarrow &  & \downarrow \mu &  \\ 
0\rightarrow & K & \rightarrow & E & \rightarrow & K & \rightarrow 0%
\end{array}%
$

Since $\mu j=1,$gluing both diagrams we obtain $\psi \rho (x)=resx.$
\end{proof}

\begin{lemma}
Under the conditions of the previous lemma the map $\rho $ is surjective.
\end{lemma}

\begin{proof}
Since $B_{n}^{!}=C_{n}^{!}\oplus ZC_{n}^{!}$ , $B_{n}^{!}\underset{C_{n}^{!}}%
{\otimes }K$ is a graded vector space of dimension two with one copy of $K$
in degree zero and one copy of $K$ in degree one. Hence the multiplication
map $\mu :B_{n}^{!}\underset{C_{n}^{!}}{\otimes }K\rightarrow K$ is an
epimorphsi with kernel $u:K[-1]\rightarrow $ $B_{n}^{!}\underset{C_{n}^{!}}{%
\otimes }K.$

Let $y:$ $0\rightarrow K[-1]\rightarrow E\rightarrow B_{n}^{!}\underset{%
C_{n}^{!}}{\otimes }K\rightarrow 0$ be an element of $%
Ext_{B_{n}^{!}}^{1}(K,B_{n}^{!}\underset{C_{n}^{!}}{\otimes }K)$ and take
the pullback:

$%
\begin{array}{ccccccc}
0\rightarrow & K[-1] & \rightarrow & N & \rightarrow & K[-1] & \rightarrow 0
\\ 
& \downarrow 1 &  & \downarrow &  & \downarrow u &  \\ 
0\rightarrow & K[-1] & \rightarrow & E & \rightarrow & B_{n}^{!}\underset{%
C_{n}^{!}}{\otimes }K & \rightarrow 0%
\end{array}%
$

But the top exact sequence split because the ends are generated in the same
degree and the algebra is Koszul or equivalently there is a lifting $%
v:K[-1]\rightarrow E$ of $u$ and we get a commutaive exact diagram:

$%
\begin{array}{ccccccc}
& 0 & \rightarrow & K[-1] & \rightarrow & K[-1] & \rightarrow 0 \\ 
& \downarrow 1 &  & \downarrow v &  & \downarrow u &  \\ 
0\rightarrow & K[-1] & \rightarrow & E & \rightarrow & B_{n}^{!}\underset{%
C_{n}^{!}}{\otimes }K & \rightarrow 0 \\ 
& \downarrow 1 &  & \downarrow &  & \downarrow \mu &  \\ 
0\rightarrow & K[-1] & \rightarrow & L & \rightarrow & K & \rightarrow 0%
\end{array}%
$

Proving $\rho $ is surjective.
\end{proof}

\begin{corollary}
The map $res:\underset{m\geq 0}{\oplus }Ext_{B_{n}^{!}}^{m}(K,K)\rightarrow 
\underset{m\geq 0}{\oplus }Ext_{C_{n}^{!}}^{m}(K,K)$ is a surjective
homomorphism of algebras and the kernel of $res$ is the ideal $%
ZB_{n}=B_{n}Z. $
\end{corollary}

\begin{proof}
Since both $B_{n}^{!}$ and $C_{n}^{!}$ are Koszul algebras they are graded
algebras generated in degree one, and it follows from the lemma that for any 
$m>0$ the map $res:Ext_{B_{n}^{!}}^{m}(K,K)\rightarrow
Ext_{C_{n}^{!}}^{m}(K,K)$ is surjective.

Observe that for any homomorphism $f:B_{n}\rightarrow C_{n}$ $Z$ is in the
kernel.

We have in $B_{n}$ the equality $X_{1}\delta _{1}-\delta _{1}X_{1}=Z^{2}$.
Since $C_{n}$ is commutative, $f($ $X_{1}\delta _{1}-\delta
_{1}X_{1})=f(X_{1})f(\delta _{1})-f(\delta _{1})f(X_{1})=f(Z)^{2}=0,$

But since $C_{n}$ is an integral domain, it follows $f(Z)=0$.

In particular $ZB_{n}\subseteq Ker(res)$ and there is a factorization: $%
\begin{array}{ccc}
B_{n} & \rightarrow & C_{n} \\ 
\searrow &  & \nearrow \alpha \\ 
& B_{n}/ZB_{n} & 
\end{array}%
$ and since $B_{n}/ZB_{n}\cong C_{n}$ it follows by dimension, that $\alpha $
is an isomorphism.
\end{proof}

\begin{lemma}
With the same notation as in the previous lemma, let $M$ be a Koszul $%
C_{n}^{!}$-module and $\psi :G_{B_{n}^{!}}(B_{n}^{!}\underset{C_{n}^{!}}{%
\otimes }M)\rightarrow G_{C_{n}^{!}}(M)$, the isomorphism in the previous
lemma.

Then given $y\in Ext_{B_{n}^{!}}^{m}(B_{n}^{!}\underset{C_{n}^{!}}{\otimes }%
M,K)$ and $c\in Ext_{B_{n}^{!}}^{1}(K,K),$ we have $\psi (cy)=res(c)\psi
(y). $
\end{lemma}

\begin{proof}
The map $f:B_{n}^{!}\underset{C_{n}^{!}}{\otimes }\Omega ^{m}(M)\rightarrow
K $ corresponding to the extension $y$ is the map in the commutative exact
diagram:

$%
\begin{array}{ccccccccc}
0\rightarrow & B_{n}^{!}\underset{C_{n}^{!}}{\otimes }\Omega ^{m}\text{(M)}
& \rightarrow & B_{n}^{!}\underset{C_{n}^{!}}{\otimes }\text{C}%
_{n}^{!k_{m-1}} & \text{...}\rightarrow & B_{n}^{!}\underset{C_{n}^{!}}{%
\otimes }\text{C}_{n}^{!k_{0}} & \rightarrow & B_{n}^{!}\underset{C_{n}^{!}}{%
\otimes }\text{M} & \rightarrow 0 \\ 
& \downarrow f &  & \downarrow &  & \downarrow &  & \downarrow 1 &  \\ 
0\rightarrow & K & \rightarrow & E_{1} & \text{...}\rightarrow & E_{m} & 
\rightarrow & B_{n}^{!}\underset{C_{n}^{!}}{\otimes }\text{M} & \rightarrow 0%
\end{array}%
$

where $y$ is the bottom raw.

If $j:\Omega ^{m}(M)\rightarrow B_{n}^{!}\underset{C_{n}^{!}}{\otimes }%
\Omega ^{m}(M)$ is the map $j(m)=1\otimes m$, then $\psi (y)$ is the
extension corresponding to the map $fj.$

Consider the commutative diagram with exact raws:

$%
\begin{array}{ccccccc}
0\rightarrow & \Omega ^{m+1}(M) & \rightarrow & C_{n}^{!k_{m}} & \rightarrow
& \Omega ^{m}(M) & \rightarrow 0 \\ 
& j\downarrow &  & j\downarrow &  & j\downarrow &  \\ 
0\rightarrow & B_{n}^{!}\underset{C_{n}^{!}}{\otimes }B\Omega ^{m+1}(M) & 
\rightarrow & B_{n}^{!}\underset{C_{n}^{!}}{\otimes }C_{n}^{!k_{m}} & 
\rightarrow & B_{n}^{!}\underset{C_{n}^{!}}{\otimes }\Omega ^{m}(M) & 
\rightarrow 0 \\ 
& \Omega f\downarrow &  & \downarrow &  & f\downarrow &  \\ 
0\rightarrow & JB_{n}^{!} & \rightarrow & B_{n}^{!} & \rightarrow & K & 
\rightarrow 0 \\ 
& g\downarrow &  & \downarrow &  & 1\downarrow &  \\ 
0\rightarrow & K & \rightarrow & L & \rightarrow & K & \rightarrow 0%
\end{array}%
$

where $c$ is the bottom sequence.

Since $B_{n}^{!}=C_{n}^{!}\oplus C_{n}^{!}Z$ as $C_{n}^{!}$-module, the map $%
g$ restricted to $C_{n}^{!}$ represents the extension $res(c).$

Taking the pullback we obtain a commutative diagram with exact raws:

$%
\begin{array}{ccccccc}
0\rightarrow & \Omega ^{m+1}(M) & \rightarrow & C_{n}^{!k_{m}} & \rightarrow
& \Omega ^{m}(M) & \rightarrow 0 \\ 
& g\Omega (f)j\downarrow &  & \downarrow &  & 1\downarrow &  \\ 
0\rightarrow & K & \rightarrow & W & \rightarrow & \Omega ^{m}(M) & 
\rightarrow 0 \\ 
& 1\downarrow &  & \downarrow &  & fj\downarrow &  \\ 
0\rightarrow & K & \rightarrow & L & \rightarrow & K & \rightarrow 0%
\end{array}%
$

and $\Omega (fj)=\Omega (f)j.$

It follows $\psi (cy)=res(c)\psi (y)$.
\end{proof}

As a corollary we obtain the following:

\begin{proposition}
Let $M$ be a Koszul $C_{n}^{!}$-module and $G_{B_{n}^{!}}=\underset{m\geq 0}{%
\oplus }Ext_{C_{n}^{!}}^{m}(-,K)$ Koszul duality. Then $%
Z(G_{B_{n}^{!}}(B_{n}^{!}\underset{C_{n}^{!}}{\otimes }M))=0$.
\end{proposition}

\begin{proof}
Denote by $z$ the extension corresponding to $Z$ under the isomorphism $%
B_{n}\cong $ $\underset{m\geq 0}{\oplus }Ext_{B_{n}^{!}}^{m}(K,K).$ By the
previous lemma, for any extension $y\in Ext_{B_{n}^{!}}^{m}(B_{n}^{!}%
\underset{C_{n}^{!}}{\otimes }M),K)$ , $\psi (zy)=res(z)\psi (y)$ and by
lemma ?, $res(z)=0.$ Since $\psi $ is an isomorphism, it follows $zy=0$,
hence $Z(G_{B_{n}^{!}}(B_{n}^{!}\underset{C_{n}^{!}}{\otimes }M))=0$.
\end{proof}

\begin{proposition}
Let $B_{n}^{!}$ and $C_{n}^{!}$ be the algebras given above. Then for any
induced module $B_{n}^{!}\underset{C_{n}^{!}}{\otimes }M$ when we apply the
duality $\overline{\phi }$ to $B_{n}^{!}\underset{C_{n}^{!}}{\otimes }M$ we
obtain an element of $\mathcal{T}$.
\end{proposition}

\begin{proof}
There exists some integer $n\geq 0$ such that $\Omega ^{n}M$ and $\Omega
^{n}(B_{n}^{!}\underset{C_{n}^{!}}{\otimes }M)$ are weakly Koszul. Since $%
\overline{\phi }(\Omega ^{n}(B_{n}^{!}\underset{C_{n}^{!}}{\otimes }M))\cong 
\overline{\phi }(B_{n}^{!}\underset{C_{n}^{!}}{\otimes }M)[n]$.The object $%
\overline{\phi }(\Omega ^{n}(B_{n}^{!}\underset{C_{n}^{!}}{\otimes }M))$ is
in $\mathcal{T}$ if and only if $\overline{\phi }(B_{n}^{!}\underset{%
C_{n}^{!}}{\otimes }M)$ is in $\mathcal{T}$. We may assume $M$ and $B_{n}^{!}%
\underset{C_{n}^{!}}{\otimes }M$ are weakly Koszul.

The module $M$ has a filtration: $M=U_{p}\supset U_{p-1}\supset
...U_{1}\supset U_{0}$ such that $U_{i}/U_{i-1}$ is Koszul, hence; $B_{n}^{!}%
\underset{C_{n}^{!}}{\otimes }M$ has a filtration: $B_{n}^{!}\underset{%
C_{n}^{!}}{\otimes }M=B_{n}^{!}\underset{C_{n}^{!}}{\otimes }U_{p}\supset
B_{n}^{!}\underset{C_{n}^{!}}{\otimes }U_{p-1}\supset ...B_{n}^{!}\underset{%
C_{n}^{!}}{\otimes }U_{1}\supset B_{n}^{!}\underset{C_{n}^{!}}{\otimes }%
U_{0} $ such that $B_{n}^{!}\underset{C_{n}^{!}}{\otimes }U_{i}/B_{n}^{!}%
\underset{C_{n}^{!}}{\otimes }U_{i-1}\cong B_{n}^{!}\underset{C_{n}^{!}}{%
\otimes }$ $U_{i}/U_{i-1}$ is Koszul.

The exact sequence: $0\rightarrow B_{n}^{!}\underset{C_{n}^{!}}{\otimes }%
U_{0}\rightarrow B_{n}^{!}\underset{C_{n}^{!}}{\otimes }U_{1}\rightarrow
B_{n}^{!}\underset{C_{n}^{!}}{\otimes }U_{1}/U_{0}\rightarrow 0$ induces an
exact sequence of complexes: $0\rightarrow \phi ($ $B_{n}^{!}\underset{%
C_{n}^{!}}{\otimes }U_{1}/U_{0})\rightarrow \phi (B_{n}^{!}\underset{%
C_{n}^{!}}{\otimes }U_{1})\rightarrow \phi (B_{n}^{!}\underset{C_{n}^{!}}{%
\otimes }U_{0})\rightarrow 0$which induces a long exact sequence:

...$\rightarrow H^{1}(\phi ($ $B_{n}^{!}\underset{C_{n}^{!}}{\otimes }%
U_{1}/U_{0}))\rightarrow H^{1}(\phi (B_{n}^{!}\underset{C_{n}^{!}}{\otimes }%
U_{1}))\rightarrow H^{1}(\phi (B_{n}^{!}\underset{C_{n}^{!}}{\otimes }%
U_{0}))\rightarrow H^{0}(\phi ($ $B_{n}^{!}\underset{C_{n}^{!}}{\otimes }%
U_{1}/U_{0}))\rightarrow H^{0}(\phi (B_{n}^{!}\underset{C_{n}^{!}}{\otimes }%
U_{1}))\rightarrow H^{0}(\phi (B_{n}^{!}\underset{C_{n}^{!}}{\otimes }%
U_{0}))\rightarrow 0$

where $H^{i}(\phi (B_{n}^{!}\underset{C_{n}^{!}}{\otimes }U_{0}))=0$ for $%
i\neq 0$ and $H^{0}(\phi (B_{n}^{!}\underset{C_{n}^{!}}{\otimes }%
U_{0}))=G_{B_{n}^{!}}(B_{n}^{!}\underset{C_{n}^{!}}{\otimes }U_{0})\cong
G_{C_{n}^{!}}(U_{0})$ of $Z$-torsion, $H^{0}(\phi (B_{n}^{!}\underset{%
C_{n}^{!}}{\otimes }U_{1}))\cong H^{0}(\phi (B_{n}^{!}\underset{C_{n}^{!}}{%
\otimes }U_{0}))$ and $H^{i}(\phi ($ $B_{n}^{!}\underset{C_{n}^{!}}{\otimes }%
U_{1}/U_{0}))\cong H^{i}(\phi (B_{n}^{!}\underset{C_{n}^{!}}{\otimes }%
U_{1})) $ for $i\neq 0.$It follows $H^{i}(\phi (B_{n}^{!}\underset{C_{n}^{!}}%
{\otimes }U_{1}))$ is of $Z$-torsion for all $i$. By induction $H^{i}(\phi
(B_{n}^{!}\underset{C_{n}^{!}}{\otimes }M))$ is of $Z$-torsion for all $i$.

We have proved $\phi (B_{n}^{!}\underset{C_{n}^{!}}{\otimes }M)$ $\in 
\mathcal{T}$.
\end{proof}

\begin{lemma}
Let $M$ be a $B_{n}^{!}$-module and assume there is an integer $n\geq 0$
such that $\Omega ^{n}M$ $=N$ has the following properties:

The module $N$ is weakly Koszul, it has a filtration: $N=U_{p}\supset
U_{p-1}\supset ...U_{1}\supset U_{0}$ such that $U_{i}/U_{i-1}$ is Koszul,
and for all $k\geq 0$, $J^{k}U_{i}\cap U_{i-1}=J^{k}U_{i-1}.$

The Koszul modules $G_{B_{n}^{!}}(U_{i}/U_{i-1})=V_{i}$ are of $Z$-.torsion.

Then $\phi (M)$ is in $\mathcal{T}$.
\end{lemma}

\begin{proof}
As above, $\phi (M)$ is in $\mathcal{T}$ if and only if $\phi (N)$ is in $%
\mathcal{T}$.

The exact sequence: $0\rightarrow U_{0}$ $\rightarrow U_{1}\rightarrow $ $%
U_{1}/$ $U_{0}\rightarrow 0$ induces an exact sequence: $0\rightarrow \phi ($
$U_{1}/$ $U_{0})\rightarrow \phi (U_{1})\rightarrow \phi (U_{0})\rightarrow
0 $ such that $H^{0}(\phi (U_{1}))\cong H^{0}(\phi (U_{0}))$ $\cong
G_{B_{n}^{!}}(U_{0})$ is of $Z$-torsion and $H^{i}(\phi ($ $U_{1}/$ $%
U_{0}))\cong H^{i}(\phi (U_{1}))$ is of $Z$-torsion for all $i\neq 0.$ By
induction, $H^{i}(\phi (N))$ is of $Z$-torsion for all $i,$hence $\phi (N)$
is in $\mathcal{T}$.
\end{proof}

\begin{theorem}
Let $\mathcal{T}$ $^{\prime }$ be the subcategory of \underline{$gr$}$%
_{B_{n}^{!}}$ corresponding to \emph{T }under the duality: $\overline{\phi }:%
\underline{gr}_{_{B_{n}^{!}}}\rightarrow D^{b}(Qgr_{B_{n}^{op}})$. This is: $%
\overline{\phi }(\mathcal{T}$ $^{\prime })=$\emph{T }$\emph{.}$Then $%
\mathcal{T}$ $^{\prime }$ is the smallest triangulated subcategory of 
\underline{$gr$}$_{B_{n}^{!}}$containing the induced modules and closed
under the Nakayama automorphism.
\end{theorem}

\begin{proof}
Let $\mathcal{B}$ be a triangulated subcategory of \underline{$gr$}$%
_{B_{n}^{!}}$containing the induced modules and closed under the Nakayama
automorphism. Let $M\in $ $\mathcal{T}$ $^{\prime }$ and $\Omega ^{n}M=N$
weakly Koszul with a filtration $N=U_{p}\supset U_{p-1}\supset
...U_{1}\supset U_{0}$ such that $U_{i}/U_{i-1}$ is Koszul, and for all $%
k\geq 0$, $J^{k}U_{i}\cap U_{i-1}=J^{k}U_{i-1}$.

Since $\mathcal{T}$ $^{\prime }$ is closed under the shift the module $N$ is
also in $\mathcal{T}$ $^{\prime }$. We prove by induction on the length of
the filtration that for each $i$ the modules $U_{i}$, $U_{i}/U_{i-1}$are in $%
\mathcal{T}$ $^{\prime }$.

We have an exact sequence of complexes: $0\rightarrow \phi ($ $N/$ $%
U_{0})\rightarrow \phi (N)\rightarrow \phi (U_{0})\rightarrow 0$

By the long homology sequence there is an exact sequence:

...$H_{i+1}(\phi (U_{0})\rightarrow H_{i}(\phi ($ $N/$ $U_{0}))\rightarrow
H_{i}(\phi (N))\rightarrow H_{i}(\phi (U_{0}))\rightarrow H_{i-1}(\phi ($ $%
N/ $ $U_{0}))$...$\rightarrow H_{0}(\phi ($ $N/$ $U_{0}))\rightarrow
H_{0}(\phi (N))\rightarrow H_{0}(\phi (U_{0}))\rightarrow 0$

By $[MZ]$, $H_{0}(\phi (N))=H_{0}(\phi (U_{0}))$, $H_{0}(\phi ($ $N/$ $%
U_{0}))=0$ and $H_{i}(\phi (U_{0}))=0$ for all $i\neq 0$, . Then $H_{i}(\phi
($ $N/$ $U_{0}))=H_{i}(\phi (N))$ for all $i\neq 0$ and $H_{0}(\phi (U_{0}))$
is of $Z$-torsion and $H_{i}(\phi ($ $N/$ $U_{0}))$ is of Z -torsion for all 
$i.$

It follows by induction, $U_{i}$, $U_{i}/U_{i-1}$are in $\mathcal{T}$ $%
^{\prime }$ for all $i$.

The Koszul modules $G_{B_{n}^{!}}(U_{i}/U_{i-1})=V_{i}$ are of $Z$-torsion
and each $Z^{j}V_{i}$ is Koszul.

There exists an exact sequence:

$0\rightarrow F_{B_{n}}(Z^{k_{i}}V_{i})[-k_{i}]\rightarrow
F_{B_{n}}(Z^{k_{i}-1}V_{i}/Z^{k_{i}}V_{i})[-k_{i}+1]...$

$\rightarrow F_{B_{n}}(V_{i}/ZV_{i})\rightarrow U_{i}/U_{i-1}\rightarrow 0$
where each $F_{B_{n}}(Z^{j}V_{i}/Z^{j+1}V_{i})$ $\cong B_{n}^{!}\underset{%
C_{n}^{!}}{\otimes }X_{ij}$ is an induced module of a Koszul $C_{n}^{!}$%
-module $X_{ij}$.

Then each $F_{B_{n}}(Z^{j}V_{i}/Z^{j+1}V_{i})$ $\cong B_{n}^{!}\underset{%
C_{n}^{!}}{\otimes }X_{ij}$ is in $\mathcal{B}$ .

Moreover, the exact sequences: 0$\rightarrow $B$_{n}^{!}\underset{C_{n}^{!}}{%
\otimes }$X$_{ik_{i}}\rightarrow $B$_{n}^{!}\underset{C_{n}^{!}}{\otimes }$X$%
_{ik_{i}-1}\rightarrow $K$_{k_{i}-2}\rightarrow $0 gives rise to triangles: $%
B_{n}^{!}\underset{C_{n}^{!}}{\otimes }X_{ik_{i}}\rightarrow B_{n}^{!}%
\underset{C_{n}^{!}}{\otimes }X_{ik_{i}-1}\rightarrow K_{k_{i}-2}\rightarrow
\Omega ^{-1}(B_{n}^{!}\underset{C_{n}^{!}}{\otimes }X_{ik_{i}})$. Therefore $%
K_{k_{i}-2}\in \mathcal{B}$. It follows by induction, $U_{i}/U_{i-1}\in 
\mathcal{B}$.

The filtration $N=U_{p}\supset U_{p-1}\supset ...U_{1}\supset U_{0}$ induces
triangles: $U_{0}\rightarrow U_{1}\rightarrow U_{1}/U_{0}\rightarrow \Omega
^{-1}(U_{0})$ with $U_{0}$,$U_{1}/U_{0}\in \mathcal{B}$. It follows $%
U_{1}\in \mathcal{B}$.

By induction, $N\in \mathcal{B}$.

We have proved $\mathcal{T}^{\prime }\subset \mathcal{B}$.
\end{proof}

\begin{theorem}
\bigskip Let $\mathcal{T}$ $^{\prime }$ be the subcategory of \underline{$gr$%
}$_{B_{n}^{!}}$ corresponding to \emph{T }under the duality: $\overline{\phi 
}:\underline{gr}_{_{B_{n}^{!}}}\rightarrow D^{b}(Qgr_{B_{n}^{op}})$. This
is: $\overline{\phi }(\mathcal{T}$ $^{\prime })=$\emph{T }. Then $\mathcal{T}
$ $^{\prime }$has Auslander Reiten triangles and they are of type $%
ZA_{\infty }$.
\end{theorem}

\begin{proof}
Let $M$ be an indecomposable non projective module in $\mathcal{T}$ $%
^{\prime }$. Then we have almost split sequences: $0\rightarrow \sigma
\Omega ^{2}M\rightarrow E\rightarrow M\rightarrow 0$ and $0\rightarrow
M\rightarrow F\rightarrow \sigma ^{-1}\Omega ^{-2}M\rightarrow 0$, since the
category \emph{T }is closed under the Nakayama automorphism, $\mathcal{T} $ $%
^{\prime }$ is also closed under the Nakayama automorphism and $\sigma
\Omega ^{2}M$, $\sigma ^{-1}\Omega ^{-2}M$ are objects in $\mathcal{T}$ $%
^{\prime }$. From the exact sequences of complexes: $0\rightarrow \phi
(M)\rightarrow \phi ($ $E)\rightarrow \phi (\sigma \Omega ^{2}M)\rightarrow
0 $ and $0\rightarrow \phi (\sigma ^{-1}\Omega ^{-2}M)\rightarrow \phi
(F)\rightarrow \phi (M)\rightarrow 0$ and the long homology sequence we get
that both $\phi ($ $E)$ and $\phi ($ $F)$ are in \emph{T}. Therefore: $E$
and $F$ are in $\mathcal{T}$ $^{\prime }$. We have proved $\mathcal{T}$ $%
^{\prime }$has almost split sequences and they are almost split sequences in 
$gr_{_{B_{n}^{!}}}$. We proved in [MZ] that the Auslander Reiten components
of \underline{$gr$}$_{B_{n}^{!}}$ are of type $ZA_{\infty }$. It follows $%
\sigma \Omega ^{2}M\rightarrow E\rightarrow M\rightarrow \sigma \Omega
^{2}M[-1]$ and $M\rightarrow F\rightarrow \sigma ^{-1}\Omega
^{-2}M\rightarrow M[-1]$ are Auslander Reiten triangles and that the
Auslander Reiten components are of type $ZA_{\infty }$.
\end{proof}

We will characterize now the full subcategory $\mathcal{F}$ $^{\prime }$ of 
\underline{$gr$}$_{B_{n}^{!}}$such that $\overline{\phi }(\mathcal{F}$ $%
^{\prime })=\mathcal{F}$\emph{\ .}

\begin{theorem}
The subcategory $\mathcal{F}$ $^{\prime }$ of \underline{$gr$}$_{B_{n}^{!}}$%
such that $\overline{\phi }(\mathcal{F}$ $^{\prime })=\mathcal{F}$ consists
of the graded $B_{n}^{!}$-modules $M$ such that the restriction of $M$ to $%
C_{n}^{!}$ is injective.
\end{theorem}

\begin{proof}
Let $M\in \mathcal{F}$ $^{\prime }.$ There is an isomorphism: $%
Hom_{D^{b}(Qgr_{B_{n}^{op}})}(\mathcal{T}$, $\pi \phi (M))\cong \underline{%
Hom}_{grB_{n}^{!}}(M,\mathcal{T}$ $^{\prime })=0$, which implies $\underline{%
Hom}_{B_{n}^{!}}(M,\mathcal{T}$ $^{\prime })=0$

In particular for any induced module $B_{n}^{!}\underset{C_{n}^{!}}{\otimes }%
\Omega ^{2}L$ we have:

$\underline{Hom}_{B_{n}^{!}}(M,B_{n}^{!}\underset{C_{n}^{!}}{\otimes }\Omega
^{2}L)=0=\underline{Hom}_{B_{n}^{!}}(\Omega ^{-2}M,B_{n}^{!}\underset{%
C_{n}^{!}}{\otimes }L).$

By Auslander-Reiten formula:

$D(\underline{Hom}_{B_{n}^{!}}(\Omega ^{-2}M,B_{n}^{!}\underset{C_{n}^{!}}{%
\otimes }L))=Ext_{B_{n}^{!}}^{1}(B_{n}^{!}\underset{C_{n}^{!}}{\otimes }%
L,M)=0$ for all $L\in gr_{C_{n}^{!}}.$

Consider the exact sequences: $0\rightarrow \Omega
_{C_{n}^{!}}(L)\rightarrow F\rightarrow L\rightarrow 0$, with $F$ the
projective cover of $L.$It induces an exact sequence:\newline
$0\rightarrow B_{n}^{!}\underset{C_{n}^{!}}{\otimes }\Omega
_{C_{n}^{!}}(L)\rightarrow B_{n}^{!}\underset{C_{n}^{!}}{\otimes }%
F\rightarrow B_{n}^{!}\underset{C_{n}^{!}}{\otimes }L\rightarrow 0$

By the long homology sequence, there is an exact sequence:

$0\rightarrow Hom_{B_{n}^{!}}(B_{n}^{!}\underset{C_{n}^{!}}{\otimes }%
L,M)\rightarrow Hom_{B_{n}^{!}}(B_{n}^{!}\underset{C_{n}^{!}}{\otimes }F,M)$ 
$\rightarrow Hom_{B_{n}^{!}}(B_{n}^{!}\underset{C_{n}^{!}}{\otimes }\Omega
_{C_{n}^{!}}(L),M)\rightarrow Ext_{B_{n}^{!}}^{1}(B_{n}^{!}\underset{%
C_{n}^{!}}{\otimes }L,M)\rightarrow 0$ which by the adjunction isomorphism
are isomorphic to the exact sequences: $0\rightarrow
Hom_{C_{n}^{!}}(L,M)\rightarrow Hom_{C_{n}^{!}}(F,M)$ $\rightarrow
Hom_{C_{n}^{!}}(\Omega _{C_{n}^{!}}(L),M)\rightarrow
Ext_{C_{n}^{!}}^{1}(L,M)\rightarrow 0.$

It follows, $Ext_{C_{n}^{!}}^{1}(L,M)\cong Ext_{B_{n}^{!}}^{1}(B_{n}^{!}%
\underset{C_{n}^{!}}{\otimes }L,M)$ and by dimension shift

$Ext_{C_{n}^{!}}^{k}(L,M)\cong Ext_{B_{n}^{!}}^{k}(B_{n}^{!}\underset{%
C_{n}^{!}}{\otimes }L,M)=0$ for all $k\geq 1.$

We have proved the restriction of $M$ to $C_{n}^{!}$ is injective.

Let's assume now the restriction of $M$ to $C_{n}^{!}$ is injective:

Then for any integer $n$ the restriction of $\Omega ^{n}M$ to $C_{n}^{!}$ is
injective.

Let $X\in \mathcal{T}^{\prime }$. There exists an integer $n\geq 0$ such
that $\Omega ^{n}X$ $=Y$, is weakly Koszul and it has a filtration: $%
Y=U_{p}\supset U_{p-1}\supset ...U_{1}\supset U_{0}$ such that $%
U_{i}/U_{i-1} $ is Koszul, and for all $k\geq 0$, $J^{k}U_{i}\cap
U_{i-1}=J^{k}U_{i-1}.$The Koszul modules $G_{B_{n}^{!}}(U_{i}/U_{i-1})=V_{i}$
are of $Z$-.torsion and each $Z^{j}V_{i}$ is Koszul.

Set $N=\Omega ^{n}M$, the restriction of $N$ to $C_{n}^{!}$ is injective.

There exist exact sequences:

$0\rightarrow F_{B_{n}}(Z^{k_{i}}V_{i})[-k_{i}]\rightarrow
F_{B_{n}}(Z^{k_{i}-1}V_{i}/Z^{k_{i}}V_{i})[-k_{i}+1]...$

$\rightarrow F_{B_{n}}(V_{i}/ZV_{i})\rightarrow U_{i}/U_{i-1}\rightarrow 0$
where each $F_{B_{n}}(Z^{j}V_{i}/Z^{j+1}V_{i})$ $\cong B_{n}^{!}\underset{%
C_{n}^{!}}{\otimes }X_{ij}$ is an induced module of a Koszul $C_{n}^{!}$%
-module $X_{ij}$.

The exact sequences: $0\rightarrow B_{n}^{!}\underset{C_{n}^{!}}{\otimes }%
X_{ik_{i}}\rightarrow B_{n}^{!}\underset{C_{n}^{!}}{\otimes }%
X_{ik_{i}-1}\rightarrow K_{k_{i}-2}\rightarrow 0$ induce exact sequences:

$0\rightarrow Hom_{B_{n}^{!}}(K_{k_{i}-2},N)\rightarrow
Hom_{B_{n}^{!}}(B_{n}^{!}\underset{C_{n}^{!}}{\otimes }X_{ik_{i}-1}$ $,N)$ $%
\rightarrow Hom_{B_{n}^{!}}(B_{n}^{!}\underset{C_{n}^{!}}{\otimes }%
X_{ik_{i}},N)$

$\rightarrow Ext_{B_{n}^{!}}^{1}(K_{k_{i}-2},N)\rightarrow
Ext_{B_{n}^{!}}^{1}($ $B_{n}^{!}\underset{C_{n}^{!}}{\otimes }%
X_{ik_{i}-1},N)\rightarrow Ext_{B_{n}^{!}}^{1}(B_{n}^{!}\underset{C_{n}^{!}}{%
\otimes }X_{ik_{i}},N)$

$\rightarrow Ext_{B_{n}^{!}}^{2}(K_{k_{i}-2},N)\rightarrow
Ext_{B_{n}^{!}}^{2}($ $B_{n}^{!}\underset{C_{n}^{!}}{\otimes }%
X_{ik_{i}-1},N)\rightarrow ..$

where $Ext_{B_{n}^{!}}^{j}($ $B_{n}^{!}\underset{C_{n}^{!}}{\otimes }%
X_{ik_{i}-l},N)\cong Ext_{C_{n}^{!}}^{j}(X_{ik_{i}-1},N)=0$ for all $j\geq
1. $

It follows: $Ext_{B_{n}^{!}}^{j}(K_{k_{i}-2},N)=0$ for all $j\geq 2.$But $%
Ext_{B_{n}^{!}}^{j}(K_{k_{i}-2},N)$

$\cong Ext_{B_{n}^{!}}^{j-1}(K_{k_{i}-2},\Omega ^{-1}N)$ and $%
Ext_{B_{n}^{!}}^{j}(K_{k_{i}-2},\Omega ^{-1}N)=0$ for all $j\geq 1$.

The sequences: $0\rightarrow K_{k_{i}-2}\rightarrow B_{n}^{!}\underset{%
C_{n}^{!}}{\otimes }X_{ik_{i}-2}\rightarrow K_{k_{i}-3}\rightarrow 0$ induce
exact sequences:

$Ext_{B_{n}^{!}}^{j}(K_{k_{i}-3},\Omega ^{-1}N)\rightarrow
Ext_{B_{n}^{!}}^{j}($ $B_{n}^{!}\underset{C_{n}^{!}}{\otimes }%
X_{ik_{i}-2},\Omega ^{-1}N)\rightarrow
Ext_{B_{n}^{!}}^{j}(K_{k_{i}-2},\Omega ^{-1}N)$

$\rightarrow Ext_{B_{n}^{!}}^{j+1}(K_{k_{i}-3},\Omega ^{-1}N)\rightarrow
Ext_{B_{n}^{!}}^{j+1}($ $B_{n}^{!}\underset{C_{n}^{!}}{\otimes }%
X_{ik_{i}-2},\Omega ^{-1}N)\rightarrow ..$

Therefore $Ext_{B_{n}^{!}}^{j+1}(K_{k_{i}-3},\Omega ^{-1}N)=0$ for $j\geq 1$
which implies

$Ext_{B_{n}^{!}}^{j}(K_{k_{i}-3},\Omega ^{-2}N)=0$ for $j\geq 1.$

Continuing by induction there exist some $m\geq 0$ such that

$Ext_{B_{n}^{!}}^{j}(U_{i}/U_{i-1},\Omega ^{-m}N)=0$ for $j\geq 1.$

By induction on $p$ we obtain $Ext_{B_{n}^{!}}^{j}(Y,\Omega ^{-m}N)=0$ for $%
j\geq 1,$in particular $Ext_{B_{n}^{!}}^{1}(Y,\Omega ^{-m}N)=0.$

By Auslander-Reiten formula, $Ext_{B_{n}^{!}}^{1}(Y,\Omega ^{-m}N)\cong D(%
\underline{Hom}_{B_{n}^{!}}(\Omega ^{-m}N,\Omega ^{2}Y))\cong D(\underline{%
Hom}_{B_{n}^{!}}(N,\Omega ^{2+m}Y))\cong D(\underline{Hom}%
_{B_{n}^{!}}(\Omega ^{n}M,\Omega ^{2+m}Y)).$

It follows \underline{$Hom$}$_{B_{n}^{!}}(\Omega ^{n}M,\Omega ^{2+m+n}X)=0$
which implies \underline{$Hom$}$_{B_{n}^{!}}(M,\Omega ^{2+m}X)=0.$

Observe $m$ depends only on $X$. Taking $\Omega ^{2+m}M$ instead of $M$ we
obtain\linebreak\ \underline{$Hom$}$_{B_{n}^{!}}(\Omega ^{2+m}M,\Omega
^{2+m}X)=$\underline{$Hom$}$_{B_{n}^{!}}(M,X)=0.$

Therefore \underline{$Hom$}$_{grB_{n}^{!}}(M,X)=0$. It follows $M\in 
\mathcal{F}^{\prime }$.
\end{proof}

\begin{theorem}
\bigskip The category $\mathcal{F}^{\prime }$ is closed under the Nakayama
automorphism, $\mathcal{F}^{\prime }$ has Auslander Reiten sequences and
they are of the form $ZA_{\infty }$. Moreover, $\mathcal{F}^{\prime }$ is a
triangulated category with Auslander-Reiten triangles and they are of type $%
ZA_{\infty }$.
\end{theorem}

\begin{proof}
Let $M$ be an indecomposable non projective object in $\mathcal{F}^{\prime }$
and $0\rightarrow \Omega M\rightarrow P\rightarrow M\rightarrow 0$ exact
with $P$ the projective cover of $M.$Since the restriction of $P$ to $%
C_{n}^{!}$ is projective and restriction is an exact functor, it follows $%
\Omega M$ is in $\mathcal{F}^{\prime }$. Similarly, $\Omega ^{-1}M$ is in $%
\mathcal{F}^{\prime }$.

If $P$ is a projective $B_{n}^{!}$ -module and $\sigma $ the Nakayama
automorphism, then $\sigma P$ is also projective. Therefore: $\mathcal{F}%
^{\prime }$ is closed under the Nakayama automorphism.

It is clear now that $\mathcal{F}^{\prime }$ has Auslander Reiten sequences
and they are of the form $ZA_{\infty }$, by [MZ].

Let $f:M\rightarrow N$ be a homomorphism with $M$, $N$ in $\mathcal{F}%
^{\prime }$ and let $j:M\rightarrow P$ be the injective envelope of $M$.
There is an exact sequence: $0\rightarrow M\rightarrow P\oplus N\rightarrow
L\rightarrow 0$ with $M$ and $P\oplus N$ in $\mathcal{F}^{\prime }$. Then $L$
is also in $\mathcal{F}^{\prime }$ and the triangle $M\rightarrow
N\rightarrow L\rightarrow \Omega ^{-1}M$ is a triangle in $\mathcal{F}%
^{\prime }$.
\end{proof}

We have characterized the pair $(\mathcal{F}^{\prime },\mathcal{T}^{\prime
}) $ corresponding to $(\mathcal{T},\mathcal{F})$ under the duality $%
\overline{\phi }:\underline{gr}_{_{B_{n}^{!}}}\rightarrow
D^{b}(Qgr_{B_{n}^{op}}).$ Applying the usual duality $D:\underline{gr}%
_{_{B_{n}^{!}}}\rightarrow $ \underline{$gr$}$_{_{B_{n}^{!op}}}$we obtain a
pair $(D(\mathcal{T}^{\prime }),D(\mathcal{F}^{\prime }))$ which corresponds
to $(\mathcal{T},\mathcal{F}) $ under the equivalence: $\overline{\phi }D:$ 
\underline{$gr$}$_{_{B_{n}^{!op}}}\rightarrow D^{b}(Qgr_{B_{n}^{op}})$.

Observe the following:

From the bimodule isomorphism $B_{n}^{!}\sigma ^{-1}\cong D(B_{n}^{!})$, for
any induced $B_{n}^{!}$-module $B_{n}^{!}\underset{C_{n}^{!}}{\otimes }X$,
there are natural isomorphisms:

$Hom_{B_{n}^{!}}(B_{n}^{!}\underset{C_{n}^{!}}{\otimes }X,$ $%
D(B_{n}^{!}))\cong D(B_{n}^{!}\underset{C_{n}^{!}}{\otimes }X)\cong
Hom_{B_{n}^{!}}(B_{n}^{!}\underset{C_{n}^{!}}{\otimes }X,$ $B_{n}^{!}\sigma
^{-1})\cong $

$Hom_{C_{n}^{!}}(X,$ $B_{n}^{!}\sigma ^{-1})\cong Hom_{C_{n}^{!}}(X,$ $%
C_{n}^{!})\underset{C_{n}^{!}}{\otimes }B_{n}^{!}\sigma ^{-1}$.

For any finitely generated right $C_{n}^{!}$-module $Y$ there exists a left $%
C_{n}^{!}$-module $X$ such that $Hom_{C_{n}^{!}}(X,$ $C_{n}^{!})\cong Y$,
hence $D(B_{n}^{!}\underset{C_{n}^{!}}{\otimes }X)\sigma \cong Y\underset{%
C_{n}^{!}}{\otimes }B_{n}^{!}$. Since $\mathcal{T}^{\prime }$ is invariant
under $\sigma $, $D(\mathcal{T}^{\prime })$ is also invariant under $\sigma $
and $D(\mathcal{T}^{\prime })$ contains the induced modules.

Let $\emph{B}$ be a triangulated subcategory of $\underline{gr}%
_{B_{n}^{!op}} $ containing the induced modules. A triangle $A$ $\overset{f}{%
\rightarrow }B\overset{g}{\rightarrow }C\overset{h}{\rightarrow }A[1]$ in $%
\underline{gr}_{B_{n}^{!}}$ comes from an exact sequence \linebreak $%
0\rightarrow A$ $\overset{\left( 
\begin{array}{c}
f \\ 
u%
\end{array}%
\right) }{\rightarrow }B\oplus P\overset{(g,v)}{\rightarrow }C\rightarrow 0$
with \ $P$ a projective module, hence $D(C)\overset{D(g)}{\rightarrow }D(B)%
\overset{D(f)}{\rightarrow }D(A)\rightarrow D(C)[1]$ is a triangle in 
\underline{$gr$}$_{_{B_{n}^{!op}}}$. Therefore: $D(\emph{B})$ is a
triangulated category containing the duals of the induced modules $D(Y%
\underset{C_{n}^{!}}{\otimes }B_{n}^{!})\cong D($ $D(B_{n}^{!}\underset{%
C_{n}^{!}}{\otimes }X)\sigma )\cong D(Hom_{C_{n}^{!}}(X,$ $C_{n}^{!})%
\underset{C_{n}^{!}}{\otimes }B_{n}^{!})\cong Hom_{B_{n}^{!}}($ $%
Hom_{C_{n}^{!}}(X,$ $C_{n}^{!})\underset{C_{n}^{!}}{\otimes }%
B_{n}^{!},\sigma B_{n}^{!})\cong \sigma B_{n}^{!}\underset{C_{n}^{!}}{%
\otimes }X.$ Clearly $\sigma D(\emph{B})$ is a triangulated category
containing the induced modules. Therefore: $\mathcal{T}^{\prime }\subset $ $%
\sigma D(\emph{B}).$ Since $\mathcal{T}^{\prime }$ is closed under Nakayama%
\'{}%
s automorphism $\sigma $, $\mathcal{T}^{\prime }\subset $ $D(\emph{B})$. It
follows $D(\mathcal{T}^{\prime })\subset $ $\emph{B}$ and $D(\mathcal{T}%
^{\prime })$ can be described as the smallest triangulated subcategory of 
\underline{$gr$}$_{_{B_{n}^{!op}}}$ that contains the induced modules.

The usual duality $D$ induces an isomorphism:

$D:Ext_{C_{n}^{!}}^{i}(M,N)\rightarrow Ext_{C_{n}^{op!}}^{i}(D(N),D(M)).$

It follows that the restriction of $M$ to $C_{n}^{!}$ is injective if and
only if the restriction of $D(M)$ to $C_{n}^{op!}$ is projective
(injective). It follows $D(\mathcal{F}^{\prime })$ is the category of $%
B_{n}^{op!}$-modules whose restriction to $C_{n}^{op!}$ is injective.

\emph{T=}$D(\mathcal{T}^{\prime })$ is a "epasse" subcategory of \underline{$%
gr$}$_{_{B_{n}^{!op}}}.$The functor $\overline{\phi }D$ induces an
equivalence of categories: \underline{$gr$}$_{_{B_{n}^{!op}}}/\emph{T}\cong
D^{b}(Qgr_{B_{n}^{op}})/\mathcal{T}$ and we proved $D^{b}(Qgr_{B_{n}^{op}})/%
\mathcal{T}\cong D^{b}(gr_{(B_{n})_{Z}}).$

The equivalence $gr_{(B_{n})_{Z}}\cong \func{mod}_{A_{n}}$ induces an
equivalence: $D^{b}(gr_{(B_{n})_{Z}})\cong D^{b}(\func{mod}_{A_{n}})$.

We have proved:

\begin{theorem}
There is an equivalence of triangulated categories:

\underline{$gr$}$_{_{B_{n}^{!}}}/\emph{T}\cong D^{b}(\func{mod}_{A_{n}}).$
\end{theorem}

The category \emph{F}$=D(\mathcal{F}^{\prime })$ is the category of all 
\emph{T} -local objects it is triangulated. By [Mi], there is a full
embedding: $\emph{F}\rightarrow $\underline{$gr$}$_{_{B_{n}^{!op}}}/\emph{T}%
\cong D^{b}(\func{mod}_{A_{n}})$.

\begin{proposition}
The category $ind_{C_{n}^{!}}$ of all induced $B_{n}^{!}$-modules is
contravariantly finite in $gr_{B_{n}^{!}}.$
\end{proposition}

\begin{proof}
Let $M$ be a $B_{n}^{!}$-module and $\mu :$ $B_{n}^{!}\underset{C_{n}^{!}}{%
\otimes }M\rightarrow M$ the map given by multiplication. Let $\alpha
:Hom_{B_{n}^{!}}(B_{n}^{!}\underset{C_{n}^{!}}{\otimes }M,M)\rightarrow
Hom_{C_{n}^{!}}(M,M)$ the morphism giving the adjunction. It is easy to see
that $\alpha (\mu )=1_{M}$.

Let $\varphi :B_{n}^{!}\underset{C_{n}^{!}}{\otimes }N\rightarrow M$ be any
map and $\alpha (\varphi )=f:N\rightarrow M$ the map given by adjunction.
There is a commutative square

$%
\begin{array}{ccc}
Hom_{B_{n}^{!}}(B_{n}^{!}\underset{C_{n}^{!}}{\otimes }M,M) & \overset{%
\alpha _{M}}{\rightarrow } & Hom_{C_{n}^{!}}(M,M) \\ 
\downarrow (1\otimes f,M) &  & \downarrow (f,M) \\ 
Hom_{B_{n}^{!}}(B_{n}^{!}\underset{C_{n}^{!}}{\otimes }N,M) & \overset{%
\alpha _{N}}{\rightarrow } & Hom_{C_{n}^{!}}(N,M)%
\end{array}%
$

from the commutativity of the diagram $f=\alpha _{N}(\varphi )=\alpha
_{N}(\mu 1\otimes f)$ implies $\varphi =\mu 1\otimes f$.

We have proved the triangle: $%
\begin{array}{ccc}
&  & B_{n}^{!}\underset{C_{n}^{!}}{\otimes }N \\ 
& 1\otimes f\swarrow & \downarrow \phi \\ 
B_{n}^{!}\underset{C_{n}^{!}}{\otimes }M & \overset{\mu }{\rightarrow } & M%
\end{array}%
$ commutes.
\end{proof}

\begin{corollary}
add(ind$_{C_{n}^{!}})$ is contravariantly finite.
\end{corollary}

\begin{corollary}
ind$_{C_{n}^{!}}$is functorialy finite.

\begin{proof}
It is clear from the duality $D($ind$_{C_{n}^{!}})\cong $ind$_{C_{n}^{op!}}$
\end{proof}
\end{corollary}

Observe (ind$_{C_{n}^{!}})^{\perp }\cong \mathcal{F}^{\prime }$ however ind$%
_{C_{n}^{!}}$ is not necessary closed under extensions and we can not
conclude $\mathcal{F}^{\prime }$ contravariantly finite.

For the notions of contravariantly finite, covariantly finite and
functorialy finite, we refer to [AS]$.$

\begin{center}
{\LARGE References}
\end{center}

[AS] Auslander, M., Smalo, S. Auslander, M.; Smal\o , Sverre O. Almost split
sequences in subcategories. J. Algebra 69 (1981), no. 2, 426--454.

[Co] Coutinho S.C. A Primer of Algebraic D-modules, London Mathematical
Society, Students Texts 33, 1995

[Da] Dade, Everett C. Group-graded rings and modules. Math. Z. 174 (1980),
no. 3, 241--262.

[Ga] Gabriel, Pierre Des cat\'{e}gories ab\'{e}liennes. Bull. Soc. Math.
France 90 1962 323--448.

[GH]\ Green, E.; Huang, R. Q. Projective resolutions of straightening closed
algebras generated by minors. Adv. Math. 110 (1995), no. 2, 314--333.

[GM1] Green, E. L.; Mart\'{\i}nez Villa, R.; Koszul and Yoneda algebras.
Representation theory of algebras (Cocoyoc, 1994), 247--297, CMS Conf.
Proc., 18, Amer. Math. Soc., Providence, RI, 1996.

[GM2] Green, E. L.; Mart\'{\i}nez-Villa, R.; Koszul and Yoneda algebras. II.
Algebras and modules, II (Geiranger, 1996), 227--244, CMS Conf. Proc., 24,
Amer. Math. Soc., Providence, RI, 1998.

[Li] Li Huishi. Noncommutative Groebner Bases and Filtered-Graded Transfer,
Lecture Notes in Mathematics 1795, Springer, 2002

[Mac] Maclane S. Homology, Band 114, Springer, 1975.

[MM] Mart\'{\i}nez-Villa, R., Martsinkovsky; A. Stable Projective Homotopy
Theory of Modules, Tails, and Koszul Duality, (aceptado 11 de septiembre
2009), Comm. Algebra 38 (2010), no. 10, 3941--3973.

[MS] Mart\'{\i}nez Villa, Roberto; Saor\'{\i}n, M.; Koszul equivalences and
dualities. Pacific J. Math. 214 (2004), no. 2, 359--378.

[MZ] Mart\'{\i}nez-Villa, Roberto; Zacharia, Dan; Approximations with
modules having linear resolutions. J. Algebra 266 (2003), no. 2, 671--697.

[Mil] Mili\v{c}i\'{c}, D. Lectures on Algebraic Theory of D-modules.
University of Utah 1986.

[Mi] Miyachi, Jun-Ichi; Derived Categories with Applications to
Representations of Algebras, Chiba Lectures, 2002.

[Mo] Mondragon J. Sobre el algebra de Weyl homgenizada, tesis doctoral (in
process)

[P] Popescu N.; Abelian Categories with Applications to Rings and Modules;,
L.M.S. Monographs 3, Academic Press 1973.

[Sm] Smith, P.S.; Some finite dimensional algebras related to elliptic
curves, Rep. Theory of Algebras and Related Topics, CMS Conference
Proceedings, Vol. 19, 315-348, Amer. Math. Soc 1996.

[Y] Yamagata, K.; Frobenius algebras. Handbook of algebra, Vol. 1, 841--887,
North-Holland, Amsterdam, 1996.

\end{document}